\documentclass[10pt]{article}

\usepackage{amsmath,amssymb,hyperref,mathabx,amsthm,chngpage,microtype,xcolor,bm}

\usepackage{empheq}

\usepackage{comment}
\usepackage{yfonts}
\usepackage[final]{showlabels}
\usepackage{enumerate}
\usepackage[dvips]{graphicx}
\usepackage[T1]{fontenc}
\usepackage{mathrsfs}
\usepackage{stmaryrd}

\numberwithin{equation}{section}

\newtheorem{thm}{Theorem}[section]

\newtheorem{stm}{Statement}

\newtheorem{defn}[thm]{Definition}
\newtheorem{prop}[thm]{Proposition}
\newtheorem{lemma}[thm]{Lemma}
\newtheorem{cor}[thm]{Corollary}

\newtheoremstyle{boldremark}
    {\dimexpr\topsep/2\relax} 
    {\dimexpr\topsep/2\relax} 
  {}          
    {}          
    {\bfseries} 
    {.}         
    {.5em}      
    {}          

\theoremstyle{boldremark}
\newtheorem{remark}[thm]{Remark}

\newcommand{\wasstwospace}{\mathscr{P}(\mathbb{T}^d)}

\newcommand{\dtorus}{\mathbb{T}^d}

\newcommand{\N}{{\ensuremath{\mathbb{N}}}}
\newcommand{\Z}{{\ensuremath{\mathbb{Z}}}}

\newcommand{\R}{{\ensuremath{\mathbb{R}}}}

\newcommand{\frmbar}{\ensuremath{\bar{\mathfrak{m}}}}
\newcommand{\frm}{\ensuremath{\mathfrak{m}}}

\newcommand{\eps}{\varepsilon}

\usepackage{pdfsync}
\usepackage{doi}

\title{Short time solution to the master equation of a first order mean field game system}
\author{Sergio Mayorga\footnote{\href{mailto:smayorga3@gatech.edu}{smayorga3@gatech.edu}} \\ {\footnotesize School of Mathematics, Georgia Institute of Technology, Atlanta, GA 30332, USA}  }

\begin{document}
\maketitle
\begin{abstract} 
The goal of this paper is to show existence of short-time classical solutions to the so called Master Equation of \emph{first order} Mean Field Games, which can be thought of as the limit of the corresponding master equation of a stochastic mean field game as the individual noises approach zero. Despite being the equation of an idealistic model, its study is justified as a way of understanding mean field games in which the individual players'~randomness is negligible; in this sense it can be compared to the study of ideal fluids. We restrict ourselves to mean field games with smooth coefficients but do not impose any monotonicity conditions on the running and initial costs, and we do not require convexity of the Hamiltonian, thus extending the result of Gangbo and Swiech to a considerably broader class of Hamiltonians. 
\end{abstract}


\tableofcontents

\section{Introduction} The \textit{master equation} (ME, for short) of \emph{first order} mean field games, namely,  \begin{equation}\label{eq:ME}  \left\{ \begin{aligned}   \partial_s &u(s,q,\mu) + \int_{\dtorus} \nabla_{\mu} u(s,q,\mu)(x) \cdot \nabla_pH(x,\nabla_q u(s,x,\mu)) \mu(dx) \\ &\ \quad  + H(q,\nabla_q u(s,q,\mu)) + F(q,\mu) = 0  \  \qquad \qquad \textrm{ in } (0,T)\times\dtorus\times\wasstwospace , \\ &\ \hspace{4cm} u(0,q,\mu) =   g(q,\mu) ,  \qquad  \textrm{ on } \dtorus\times\wasstwospace  \end{aligned} \right. \end{equation} is a non-local, infinite-dimensional partial differential equation that arises in mean field game (abbreviated MFG) theory and can be interpreted as either the limit, as $N\to\infty,$ of a system of $N$ coupled Hamilton-Jacobi equations that represent the Nash equilibrium of a differential game played by $N$ interacting deterministic particles, or as the limiting case (formally, at least) of the master equation of a stochastic differential game when the viscosity parameter, associated with the intrinsic noise of the infinitesimal particles, tends to zero. In (\ref{eq:ME}), $\wasstwospace$ is the Wasserstein space of Borel probability measures on the $d$-dimensional torus $\dtorus.$ The objective of this paper is to construct a short-time solution to (\ref{eq:ME}) for an arbitrary smooth Hamiltonian $H.$ To our knowledge, existence of solutions to (\ref{eq:ME}) has only been shown for the particular case \cite{mfgmain} of the quadratic Hamiltonian $H(q,p)=\frac{1}{2}|p|^2.$ 

Mean-field theory in differential games began with the works of P.L.~Lions and J.M.~Lasry \cite{lionsjapanese} and M.~Huang, P.E.~Caines, R.P.~Malham\'e \cite{huangcainesmalhame}, attracting great interest since then for its numerous applications and posing challenging theoretical questions. Equation (\ref{eq:ME}) and its higher-order version were first introduced by Lasry and Lions \cite{courscollegefrance}, motivated by, among other reasons \cite{fourhorses}, the need to clarify the connection between games with finitely but many players, and MFGs. Regarding the latter, and considerably more scrutinized so far than (\ref{eq:ME}), is the so called \textit{first order mean field game system}: \begin{empheq}[left=\empheqlbrace]{align}  \partial_t U(t,q) + H(q,\nabla_q U(t,q) ) + F(q,\sigma_t) = &\ 0 \quad \textrm{ in } (0,s)\times\dtorus , \label{eq:mfg1} \\ \partial_t \sigma_t  + \textrm{div}(\sigma_t \nabla_pH(q,\nabla_q U )) = &\ 0 \quad \textrm{ in }   \mathcal{D}'((0,s)\times\dtorus),  \label{eq:mfg2} \\ U(0,\cdot) = &\ g (\cdot,\sigma_0), \label{eq:mfg3} \\ \sigma_s = &\ \mu ,\label{eq:mfg4}  \end{empheq}  which describes a Nash-type equilibrium state of a differential game\footnote{Systems such as (\ref{eq:mfg1}-\ref{eq:mfg4}) are often called deterministic because they derive from MFGs where the differential equation that governs the evolution of each player has no stochastic terms, with the resulting MFG system featuring no second order derivatives.} played by a continuum of players on $\dtorus$ who seek to minimize a certain cost function that depends on the collective behavior of all the players; in such a state, $U(t,q)$ represents the value function of a typical player $q$ at time $t$ and $\sigma_t$ is the mass distribution of all the players at time $t$, represented by a Borel probability measure on $\dtorus.$ The first equation in (\ref{eq:mfg1}-\ref{eq:mfg4}) is a forward Hamilton-Jacobi equation and the second a backward continuity equation. These equations can be derived as the optimality conditions of the aforementioned game (see, e.g.~\cite{benamousantambrogio,cardaliaguetnotes,gueant}) or as the limit of approximate Nash equilibria in finitely-many player games (e.g.~\cite{lackerlimit}). Alternatively, a solution to (\ref{eq:ME}) can be used directly to construct an optimal control for the Nash equilibrium of the mean field game \cite{gangboberkeleynotes}. The now extensive and rapidly growing literature on MFGs includes surveys and books \cite{bensoussanbook,cardaliaguetnotes,carmonabible,diogomodelssurvey} that give comprehensive accounts of the theory, its models and applications, with at least one book \cite{diogoregularitybook} devoted entirely to regularity theory of MFGs.

In the present work, the Hamiltonian $H$ is smooth, and not necessarily convex in $p$ (which is commonly required), while the couplings $F$ and $g$ are smooth and jointly Lipschitz, with a continuous ``mixed'' derivative (see Section \ref{subsection:dataformfg} for the assumptions). In the master equation, $F$ and $g$ are required to possess some more regularity in both variables. For convenience, we refer the reader to Section \ref{section:summaryofresults} for a summary of our results. When $H(q,\cdot)$ is not convex, the typical interpretation of the MFG system as an optimization problem in mean field games is lost, since the Legendre transform of $H$ is not defined (however, nonconvex MFGs appear in the literature \cite{trannonconvex}, \cite{bardicirant}). Nevertheless, this does not prevent us from looking at (\ref{eq:ME}) and (\ref{eq:mfg1}-\ref{eq:mfg4}) as a legitimate problem in PDEs. Considerable understanding of the relationship between master equations and mean field game systems has been gained since their inception, and, indeed, an established strategy \cite{fourhorses}, which we also use here, to construct a solution to the master equation is to use solutions to the MFG system as the starting point; a probabilistic approach can be found, e.g., in \cite{chassagneux,carmonabible}. However, it is MFG systems that have seen more rigorous and abundant treatment in the literature. Theoretically, this is mostly due to the derivative in measure that appears in (\ref{eq:ME}). The main issue with the strategy just mentioned becomes to prove that the solutions to (\ref{eq:mfg1}-\ref{eq:mfg4}) behave nicely enough with respect to the terminal measure $\mu.$ 

Concerning the system\footnote{Our presentation of the MFG is reversed in time with respect to the most frequent one encountered in literature, i.e., with a terminal condition for $g$ and an initial one for $\sigma,$ and with minus signs in front of the time derivatives.} (\ref{eq:mfg1}-\ref{eq:mfg4}), weak solutions in the viscosity sense were initially obtained on $\dtorus$ and $\R^d$ for regularizing couplings, Hamiltonians with quadratic growth in the momentum variable and arbitrary time horizons \cite{averbukh14,cardaliaguetnotes,cardaliaguetweakkam,lionsjapanese}. Uniqueness is usually obtained by imposing monotonicity conditions on the couplings. Several significant modifications and refinements of these results have been obtained since then, e.g., existence and uniqueness of weak solutions in the case of local couplings \cite{cardaliaguetgraber}, first-order \cite{cardaliaguetporretta,lavenantsantambrogio,santambrogioregularity} and higher \cite{grabermeszaros} Sobolev estimates of such solutions, with different growth conditions of the Hamiltonian, whose convexity in the momentum variable is always required, and absolute continuity of the terminal measure $\mu;$  all the approaches in these developments work for arbitrary time horizons. Regarding second order MFGs, i.e., \begin{equation}\label{eq:mfgsystem2} \left\{ \begin{aligned} \epsilon \Delta U + \partial_t U(t,q) +  H(q,\nabla_q U(t,q) ) + F(q,\sigma_t) = &\ 0 \quad \textrm{ in } [0,s]\times\dtorus ,  \\ -\epsilon \Delta U + \partial_t \sigma_t + \textrm{div}(\sigma_t \nabla_pH(q,\nabla_q U )) = &\ 0 \quad \textrm{ in }   \mathcal{D}'([0,s]\times\dtorus),  \end{aligned} \right. \end{equation} strong solutions for considerably larger classes of coefficients, in which the interaction of the particles enters the equation directly in $H$ and not necessarily through $F$ as a separate summand, have been obtained by Ambrose \cite{ambrosesmallstrong,ambrosestrong}, with the requirement of certain smallness assumptions on the data (see, also, \cite{ciranttonon}). In general, the presence of the viscosity term in (\ref{eq:mfgsystem2}) affords solutions to enjoy better regularity properties. For this reason, and because (\ref{eq:mfgsystem2}) accommodates models where players act non-deterministically, second order mean field games have received more attention in the scientific literature \cite{diogomodelssurvey,diogoregularitybook}. Incidentally, we see that due to the sign of the viscosity term in the first equation of (\ref{eq:mfgsystem2}), the system can be well posed only if an initial condition is prescribed while the opposite sign in the second equation makes it mandatory to prescribe the terminal value of $\sigma.$ 

In the course of our work towards the ME, we obtain classical solutions for short times of the system (\ref{eq:mfg1}-\ref{eq:mfg4}), with the conditions on the coefficients $H, F, g$ mentioned above.

With respect to master equations, most available literature has dealt with the higher-order variants \cite{bensoussanthree, bensoussanthree2,carmonabible,carmonadelaruelarge}. The recent paper by Cardaliaguet et al.~\cite{fourhorses} includes rigorous proofs of classical solutions for the master equations of second-order MFGs, such as (\ref{eq:mfg1}-\ref{eq:mfg4}) and their characterization as the convergence of $N$-player Nash systems as $N\to\infty.$ As for the first-order master equation (\ref{eq:ME}), a major step was achieved by Gangbo-\'Swi\k{e}ch in \cite{mfgmain},  where short-time strong solutions to both the MFG system (\ref{eq:mfg1}-\ref{eq:mfg4}) and the ME are obtained for quadratic Hamiltonians and potential-derived couplings. The same result was later proved by Bessi \cite{bessi} using different techniques. The present paper works with general smoothness conditions on $H$, $F$ and $g$ that include \cite{mfgmain} as a particular case, but which constitute a solid generalization, especially since it does not require any geometric assumptions: the Hamiltonian $H$ may be nonconvex in $p.$ We use similar ideas and techniques to arrive at the ME, but our route to the MFG system is different. Let us explain the differences with \cite{mfgmain}. Given $H, F, g$ as in Section \ref{subsection:dataformfg}, and given $0<s<T,$ $\mu\in\wasstwospace,$ we prove that, granted $T$ is small, there are functions $\Sigma^1:[0,T]\times\dtorus\to\dtorus,$ $\Sigma^2:[0,T]\times\dtorus\to\R^d$ that solve the infinite-dimensional Hamiltonian system \begin{equation}\label{eq:introhamsystem} \partial_t \Sigma^1 = \nabla_p H(\Sigma^1,\Sigma^2), \qquad \partial_t \Sigma^2 = -\nabla_q H(\Sigma^1,\Sigma^2) - \nabla_q F(\Sigma^1,\Sigma^1_{\#}\mu) \end{equation} with initial and terminal conditions $$ \Sigma^2(0,q) = \nabla_q g(\Sigma^1(0,q),{{}\Sigma^1(0,\cdot)}_{\#}\mu) , \qquad \Sigma^1(s,q) = q , $$ providing us with a path in $\wasstwospace$ given by $$ t\mapsto \sigma_t := {}{\Sigma^1_t}_{\#}\mu $$ and prove that there is a function $U:[0,T]\times\dtorus\to\R$ such that\footnote{We follow the convention, common in this field, of using the subindex $t$ to mean ``at time $t$'', and thus a shorthand for $(t,\ldots)$ rather than the time derivate.} \begin{equation}\label{eq:gradgeneralH} \nabla_q U(t,\Sigma^1_t) = \Sigma_t^2. \end{equation} On the other hand, we have that the velocity vector $v_t$ driving the path $\sigma_t$ satisfies $$ v(t,\Sigma^1_t) = \partial_t\Sigma_t^1,$$ and the first equation in (\ref{eq:introhamsystem}) implies \begin{equation}\label{eq:gradspecificH} \nabla_p H(q,\nabla_q U(t,\Sigma^1_t)) = \partial_t \Sigma^1_t. \end{equation} Comparing (\ref{eq:gradspecificH}) and (\ref{eq:gradgeneralH}), we see that they are the same if $H(q,p)=\frac{1}{2}|p|^2,$ which is the Hamiltonian in \cite{mfgmain}, and, indeed, in that case, $\partial_t\Sigma^1 = \Sigma^2,$ with $\nabla_q U(t,\cdot)$ coinciding with the velocity $v_t$ (\textit{a posteriori} from (\ref{eq:gradgeneralH})). Thus, the function $\Sigma^2$ is not present in \cite{mfgmain}, with $\partial_t \Sigma^1$ taking its place, while the relationship $\nabla_q U(t,\Sigma_t^1) = \partial_t \Sigma^1_t$ is obtained via the link of the MFG system with a variational problem: if $L(x,v):=\frac{1}{2}|v|^2,$ and having shown that the pair $(\sigma,v)$ is the unique minimizer\footnote{See also \cite{ghoussoub} for a recent connection between value functionals such as (\ref{eq:theminimizerof}) and Hopf-Lax formulae on the Wasserstein space.} of \begin{equation}\label{eq:theminimizerof} \mathcal{U}(s,\mu) = \inf_{\substack{(\sigma,v)}}\big\{ \int_0^s \int_{\dtorus} \big( L(q,v_t(q)) - \mathcal{F}(\sigma_t) \sigma_t(dq) \big) dt + \mathcal{G}(\sigma_0) \ \big| \ \sigma_s = \mu,\ \sigma\in AC^2(0,s;\wasstwospace) \big\} , \end{equation} where $\mathcal{F},\mathcal{G}:\wasstwospace\to\R$ are functions whose Wasserstein gradients are $F$ and $g,$ the minimality of the norm of $v_t$ then follows, leading to the symmetry of $\nabla_q v_t(q),$ which is used to establish $\nabla_q U(t,\Sigma_t^1)=\partial_t \Sigma_t^1$ and, in turn, the Hamilton-Jacobi equation in (\ref{eq:mfg1}-\ref{eq:mfg4}). In the case of the general Hamiltonian, it is no longer clear how this approach can give us (\ref{eq:gradgeneralH}). We turn, instead, to a more direct procedure (Lemma \ref{lemma:zlemma}) that also helps to shed further light on how the equations (\ref{eq:hamODEs}) are the \textit{characteristics} of (\ref{eq:mfg1}-\ref{eq:mfg4}). This optic allows us to present a sort of uniqueness counterpart (Theorem \ref{thm:uniquenessmfg}) to the existence result, namely, that if a solution $(\tilde U,\tilde \sigma)$ is in $W^{2,3;\infty}((0,T)\times\dtorus)\times AC^2(0,T;\wasstwospace),$ then it must coincide, at least for a shorter time $T$, with the pair $(U,\sigma)$ constructed from $(\Sigma^1,\Sigma^2).$ With this approach we manage to circumvent the specific potential forms for $F$ and $g$ present in \cite{mfgmain}. In Section \ref{section:regularityinmu} we work out the differentiability of $\Sigma$ in $\mu$ through the same discretization approach used in \cite{mfgmain}. This is followed by the chain rules and Lipschitz estimates of the composite functions that enter the representation formula for $u$ in (\ref{eq:specificu}). We should say that the formulas for the Wasserstein gradients have to be defined and their Lipschitz estimates proved; there is no general rigorous rule on composite functions that we can invoke. Finally, and due to the preceding remarks about $\partial_t\Sigma_t^1$ and $\Sigma_t^2,$ an extra tool (Lemma \ref{lemma:keylemma}) is needed to complete the chain rule for $u$ that is really the essence of the master equation (Theorem \ref{thm:masterequation1}).
 \section{Preliminaries}\label{section:preliminaries} For full details on the theory of optimal transport and the Wasserstein space of probability measures on the $d$-dimensional torus $\dtorus:=\R^d/\Z^d$, we refer the reader to \cite{weakkam}. In this section we set down the notation for the paper, present a few general results that will be needed and fix the class of coefficients for the MFG equations. We also fix our meaning of classical solution to the MFG system and to the master equation. \begin{itemize} 
\item The set of equivalence classes on $\R^d$ with respect to the equivalence relation: $$ x\sim y \qquad  \textrm{iff} \qquad \textrm{ there exist integers } n_1, \ldots, n_d \textrm{ such that } x^{(j)}-y^{(j)} = n_j, \ j=1,\ldots,d $$ is denoted by $\dtorus,$ where $x^{(j)},$ $y^{(j)}$ are the $j$-th coordinates of $x,$ $y.$ If $x,y\in\dtorus$ then, $$ |x-y|_{\dtorus} := \min\{ |x'-y'| \ \big| \ x, y \in\R^d, x'\sim x, y'\sim y \}.  $$ 
\item If $\mu,$ $\nu$ are Borel probability measures on $\R^d$, $\Gamma(\mu,\nu)$ denotes the set of those Borel probability measures $\gamma$ on $\R^d\times\R^d$ whose marginals are $\mu$ and $\nu$: that is, $\pi^1_{\#}\gamma=\mu$ and $\pi^2_{\#}\gamma=\nu,$ where $\pi^1,\pi^2:\R^d\times \R^d\to \R^d$ are the first and second coordinate projections, respectively, and the subindex $_{\#}$ stands for the pushforward operator. 
\item We use the standard notation $\mathscr{P}_2(\R^d)$ for the Wasserstein space of Borel probability measures on $\R^d$ whose second moments are finite, with quadratic Wasserstein distance $W_2.$ 
\item For $\mu, \nu\in\mathscr{P}_2(\R^d),$ we define \begin{equation}\label{eq:defofW} \mathscr{W}(\mu,\nu) = \bigg( \inf_{\gamma\in\Gamma(\mu,\nu)} \int_{\R^d\times\R^d} |x-y|^2_{\dtorus}\gamma(dx,dy) \bigg)^{1/2} ,  \end{equation} and let $\Gamma_0(\mu,\nu)$ denote the set of optimal transport plans $\gamma$ between $\mu$ and $\nu,$ i.e.~those for which the infimum in (\ref{eq:defofW}) is attained. With the equivalence relation \begin{equation*}
\mu \sim \nu \qquad \textrm{ iff } \qquad \int_{\R^d}\phi d\mu = \int_{\R^d} \phi d\nu \ \textrm{ for all }\  \phi\in C(\dtorus) 
\end{equation*} 
on $\mathscr{P}_2(\R^d),$ where $C(\dtorus)$ are all real-valued continuous functions $\phi$ on $\R^d$ such that $\phi(x)=\phi(x')$ whenever $x\sim x',$ it is true that $\mathscr{W}(\mu,\nu) = \mathscr{W}(\mu',\nu')$ whenever $\mu\sim\mu'$ and $\nu\sim\nu'.$ In this way, $\mathscr{W}$ in formula (\ref{eq:defofW}) is defined on the set of equivalence classes, which we henceforth denote by $\wasstwospace.$ Moreover, $\mathscr{W}$ is a metric on $\wasstwospace,$ with respect to which $\wasstwospace$ is compact.
\item By a mapping $F:\dtorus\to S,$ where $S$ is any set, we mean $F:\R^d\to S$ such that $F(x)=F(x')$ whenever $x\sim x'.$ Likewise, a mapping $\mathcal{F}:\wasstwospace\to S$ is a function $\mathcal{F}:\mathscr{P}(\R^d)\to S$ that takes constant values on the equivalence classes of $\wasstwospace.$ Furthermore, a function $F:\dtorus\to\dtorus$ is to be understood as a function $F:\R^d\to\R^d$ such that $F(x)\sim F(y)$ whenever $x\sim y.$ 
\item If $x=(x_1,\ldots,x_n)\in (\R^d)^n,$ then $\mu^x\in\mathscr{P}_2(\R^d)$ denotes the measure $\mu^x = \frac{1}{n}\sum_{j=1}^n \delta_{x_j}.$ Such measures are called averages of Dirac masses.
\item If $f, g : \R^d\to \R^d$ are Borelian, and $\mu\in\mathscr{P}_2(\R^d),$ then, estimating through $(f\times g)_{\#}\mu$ one obtains  
\begin{equation}\label{eq:remarkfg} 
W_2(f_{\#}\mu,g_{\#}\mu)\leq\|f-g\|_{L^{2}(\R^d,\mu)}.
\end{equation} 
\item Let $\mu\in\mathscr{P}_2(\R^d).$ Then $L^2(\dtorus,\mu)$ denotes the completion of $C(\dtorus)$ with respect to the $L^2(\R^d,\mu)$ norm: $L^2(\dtorus,\mu)=\overline{C(\dtorus)}^{L^2(\R^d,\mu)}.$ At the same time, we define the \textit{tangent space to $\wasstwospace$ at $\mu$}, $\mathscr{T}_{\mu}\wasstwospace,$  to be the $L^2(\R^d,\mu)$-completion of the subspace of $L^2(\dtorus,\mu)$ consisting of gradients of smooth periodic functions on $\R^d$: $ \mathscr{T}_{\mu}\wasstwospace := \overline{\nabla C^{\infty}(\dtorus;\R)}^{L^2(\R^d,\mu)}.$
Since $L^2(\dtorus,\mu)$ is a Hilbert space, if we have $\xi, \eta \in L^2(\dtorus,\mu),$ and $\eta\in\mathscr{T}_{\mu}\wasstwospace,$  then 
\begin{equation}\label{eq:projfact}
\int_{\dtorus} \xi(x)\cdot \eta(x) \mu(dx) = \int_{\dtorus} \bar\xi(x) \cdot \eta(x) \mu(dx),
\end{equation} 
where $\bar\xi$ is the projection of $\xi$ onto $\mathscr{T}_{\mu}\wasstwospace.$ 
\item \textit{Wasserstein distance between average of Dirac masses.} If $\mu, \nu \in\mathscr{P}_2(\R^d)$ are such that $\mu=\frac{1}{n}\sum_{j=1}^n \delta_{x_j}$ and $\nu=\frac{1}{n}\sum_{j=1}^n\delta_{y_j},$ where $x_j\neq x_k, y_j\neq y_k$ for $j\neq k,$ then there is a permutation $p:\{1,\ldots,n\}\to\{1,\ldots,n\}$ such that $$ \mathscr{W}^2(\mu,\nu) = \frac{1}{n} \sum_{j=1}^n |y_{p(j)}-x_j|_{\dtorus}^2 .$$ 
\item We denote by $AC^{2}(0,T;\wasstwospace)$ the set of paths $\mu:(0,T)\to\wasstwospace$ for which there exists $m\in L^2(0,T)$ such that $\mathscr{W}(\mu_{t_1},\mu_{t_2})\leq \int_{t_1}^{t_2}m(\tau) d\tau$ whenever $0<t_1\leq t_2<T.$ 
\item 
We say that a time-dependent velocity vector field $v_t:\dtorus\to\R^d$ is a velocity vector field for the absolutely continuous path $\mu_t$ if $v_t\in L^p(\dtorus,\mu)$, 
$$ 
\int_0^T\int_{\dtorus}|v_t(q)|\mu_t(dq) dt < \infty
$$
and the \textit{continuity equation} is true:
\begin{equation*}
\partial_t \mu_t + \textrm{div}(v_t\mu_t) = 0 \quad \textrm{ in } \mathcal{D}'((0,T)\times\dtorus).
\end{equation*}
\item
A path $\mu_t$ in $\wasstwospace,$ $0\leq t\leq 1,$ is said to be a constant-speed geodesic if
$$
\mathscr{W}(\mu_{t_1},\mu_{t_2}) = |t_2-t_1|\mathscr{W}(\mu_0,\mu_1), \quad t_1,t_2\in[0,1].
$$
Given a path $\mu_t$ in $\wasstwospace,$ a velocity $v_t$ for $\mu_t$ is in $L^2(\dtorus;\mu)$ but it may or may not be in $\mathscr{T}_{\mu_t}\wasstwospace.$ However, if $\mu_t$ is an $AC^2(0,T;\wasstwospace)$ path, a velocity field of minimal $L^2(\dtorus,\mu)$-norm always exists, and it belongs to $\mathscr{T}_{\mu_t}\wasstwospace.$ This is the content of \cite[Theorem 8.3.1]{gradientflows}.
\begin{remark}\label{remark:littlefact}
Let $\mu, \nu \in \wasstwospace,$ $\gamma\in\Gamma_0(\mu,\nu).$ For each $0\leq \tau\leq 1,$ let $$\mu^{\tau}:= [(1-\tau)\pi^1 + \tau\pi^2]_{\#}\gamma.$$ Let $w^{\tau},$ $0\leq \tau\leq 1,$ be the velocity vector field of minimal norm for $\mu^{\tau}.$ 
\begin{enumerate}[(i)]
\item For $0\leq \tau\leq 1,$ $\|w^{\tau}\|_{L^2(\mu^{\tau})} = \mathscr{W}(\mu,\nu).$
\item For every $f\in C(\dtorus;\R^d),$ every $0\leq \tau\leq 1,$
\begin{align*}
\int\limits_{\dtorus\times\dtorus} f((1-\tau)x+\tau y)\cdot w^{\tau}((1-\tau)x+\tau y)\gamma(dx,dy) = \int\limits_{\dtorus\times\dtorus} f((1-\tau)x+\tau y)\cdot (y-x) \gamma(dx,dy). 
\end{align*} 
\item Furthermore, fix $\tau\in (0,1),$ and let $\gamma^{\tau}\in\Gamma_0(\mu,\mu^{\tau}).$ Then, for every $f_1, f_2\in C(\dtorus;\R^d),$ 
\begin{align*}
\int_{\dtorus\times\dtorus} [  f_2(y)\cdot w^{\tau}(y) - f_1(x) \cdot w^{0}(x) ] \gamma^{\tau}(dx,dy) 
=
\int_{\dtorus\times\dtorus} [f_2(y) - f_1(x)] \cdot \frac{y-x}{\tau} \gamma^{\tau}(dx,dy).
\end{align*}  
\end{enumerate}
\end{remark}
We omit the proof of this remark, for which the reader can refer to \cite{majorga}. $\sslash$
\item
\textit{Density of average of Dirac masses in $\wasstwospace.$} Let $\mu\in\mathscr{P}_2(\R^d).$ As it is well known (see, for instance, \cite[Ex.~8.1.6]{bogachev}), the set of average of Dirac masses is dense in $\mathscr{P}_2(\R^d)$ with respect to narrow convergence. In $\wasstwospace,$ this convergence coincides with convergence in $\mathscr{W}.$ Thus, there exists a sequence $\{\mu(n)\}_1^{\infty}\subset\mathscr{P}_2(\R^d),$ with $ \mu(n) = \frac{1}{n} \sum_{j=1}^n \delta_{x_j(n)},$ an average of Dirac masses, such that $\mathscr{W}(\mu,\mu(n))\to 0$ as $n\to\infty.$ Moreover, this sequence can be chosen so that each $x_j(n)\in\textrm{supp}(\mu),$ where $\textrm{supp}(\mu)$ is the support of the measure $\mu.$    
\end{itemize}  \subsection{Assumptions for the mean-field game equations}\label{subsection:dataformfg} \begin{enumerate} \item Let $H\in C^3(\dtorus\times\R^d), $ $H=H(q,p).$ In this manuscript, $\nabla_qH(\cdot,\cdot)$ will always denote the gradient of $H$ with respect to $q$, evaluated at $(\cdot,\cdot).$ Similarly for $\nabla_pH(\cdot,\cdot),$ and higher-order derivatives. 
\item Let $F=F(q,\mu),$ $q\in\dtorus,$ $\mu\in\wasstwospace,$ be continuous in the $\mu$ variable and of class $C^3$ in $q$, and let $\kappa>0$ be a constant such that  $$ |\nabla_q F(q,\mu)|,\ |\nabla^2_{qq}F(q,\mu)|,\ |\nabla^3_{qqq}F(q,\mu)| \leq \kappa, \quad q\in\dtorus, \mu\in\wasstwospace.  $$ Suppose, further, that $\nabla_q F$ is $\kappa$-Lipschitz on $\dtorus\times\wasstwospace,$ meaning that $$ |\nabla_q F(q_1,\mu_1)-\nabla_q F(q_2,\mu_2)|\leq\kappa\sqrt{|q_1-q_2|^2+\mathscr{W}^2(\mu_1,\mu_2)}, \quad q_1, q_2 \in \dtorus, \ \mu_1, \mu_2 \in \wasstwospace .
$$  
\item Furthermore: we require that the vector field $\nabla_q F(q,\mu)$ is differentiable with respect to every $\mu,$ at every $q,$ and
$$
\nabla_{\mu}\nabla_qF(q,\mu)(x) =: \nabla^2_{\mu q}F(q,\mu)(x)
$$
is continuous in $(q,\mu,x)$ (hence, uniformly bounded).
\item 
Let $g=g(q,\mu),$ $q\in\dtorus,$ $\mu\in\wasstwospace,$ and suppose $g$ satisfies exactly the same conditions asked of $F.$ 
\end{enumerate}
We call the triple $(H,F,g)$ \textit{the coefficients} for the mean-field game equations.
\subsection{Assumptions for the master equation}\label{subsection:medata}
In addition to the previous set of conditions, here we suppose that the functions 
$$
(q,\mu)\mapsto F(q,\mu), \quad  (q,\mu)\mapsto g(q,\mu) \quad \textrm{ are twice differentiable}
$$ 
in the sense explained below in Section \ref{subsubsection:twicedifferentiability}, and that 
$$
\nabla_{\mu}F, \  \nabla^2_{q\mu}F, \  \nabla^2_{\mu\mu} F, \  \nabla^2_{x\mu}F \quad \textrm{ are continuous in all its variables}
$$ (and, therefore, uniformly bounded). We suppose an identical statement holds for $g$.

\textit{Examples.} \textit{1.} The following is the case in \cite{mfgmain}:
$$
F(q,\mu) = \int_{\dtorus}\phi(q-y)\mu(dy), \quad g(q,\mu) = U^0(q) + \int_{\dtorus}U^1(q-y)\mu(dy),
$$
where $U^0,$ $U^1,$ $\phi$ are smooth functions and $\phi, U^1$ are even.

\textit{2.} We can take
$$
F(q,\mu) = U(q) + \frac{1}{m}\int\limits_{(\dtorus)^m}\Phi(q,y_1,\ldots,y_m)\mu(dy_1)\cdots\mu(dy_m)
$$
(and similarly for $g$), where $U$ and $\Phi$ are smooth and $\Phi$ is symmetric in its $m+1$ variables; see \cite{chowgangbo}.
\subsection{Definitions of classical (strong) solutions}\label{subsection:defnofclassicalmfg} Let $T>0,$  and $F,g: \dtorus\times\wasstwospace\to\R$ be continuous; let $H:\dtorus\times\R^d\to\R$ be continuous and differentiable in $p.$ \paragraph{MFG system} Let $0<s<T,$ $\mu\in\wasstwospace.$ We say that the pair of functions $U:(0,T)\times\dtorus \to \R,$ $\sigma : (0,T)\to \wasstwospace$ is a \textit{classical solution to the first-order MFG system} (\ref{eq:mfg1}-\ref{eq:mfg4}) \textit{on $\dtorus$ with coefficients $(H,F,g)$ and parameters $s,$ $\mu$} if the following hold: \begin{itemize}\renewcommand{\labelitemi}{\tiny$\blacksquare$} \item $U\in C^1((0,T)\times\dtorus);$ \item the path $\sigma\in AC^2(0,T;\wasstwospace)$ and (\ref{eq:mfg2}) is true in the sense of distributions, i.e., for every $\varphi\in C_c^{\infty}((0,T)\times\dtorus)$: 
\begin{equation*}
\int_0^T \int_{\R^d} [ \partial_t \varphi(t,q) + \nabla \varphi(t,q)\cdot \nabla_pH(q,\nabla_q U(t,q)) ]  \sigma_t(dq)  dt = 0  ; 
\end{equation*}
\item equation (\ref{eq:mfg1}) is satisfied pointwise, along with the condition (\ref{eq:mfg3}) at time $t=0$ for $U$ and the condition (\ref{eq:mfg4}) at time $t=s$ for $\sigma$. \end{itemize} We will often refer to the function $U$ in (\ref{eq:mfg1}) as the \textit{value function}. \paragraph{Master equation} We say that the function $u:(0,T)\times\dtorus\times\wasstwospace\to\R$ is a \textit{classical solution the master equation of first-order MFGs} (\ref{eq:ME}) \textit{with coefficients $(H,F,g)$} if: \begin{itemize}\renewcommand{\labelitemi}{\tiny$\blacksquare$} \item  $u$ is differentiable in $s,$ with $\partial_s u(\cdot,\cdot,\mu)$ continuous at every $\mu\in\wasstwospace;$ \item  $u$ is differentiable in $q$, with $\nabla_qu$ continuous in all three variables; \item $u$ is differentiable in $\mu$ (see the following section), and $u$ satisfies (\ref{eq:ME}) pointwise.\end{itemize} We will refer to the function $u$ in (\ref{eq:ME}) as the \textit{full value function}.
\subsection{Differentiability in the Wasserstein space}\label{subsection:differentiabilityinwasserstein} 
Let $\mathcal{W}$ be a real-valued function on $\wasstwospace$ and let $\mu\in\mathscr{P}_2(\R^d)$ be fixed. For $\xi\in L^2(\dtorus,\mu),$ $\nu\in\mathscr{P}_2(\R^d),$ $\gamma\in\Gamma(\mu,\nu),$ define 
\begin{equation*}
e(\nu,\xi,\gamma) := \mathcal{W}(\nu) - \mathcal{W}(\mu) - \int\limits_{\R^d\times\R^d} \xi(x)\cdot (y-x)\gamma(dx,dy). 
\end{equation*} 
We have chosen to present this section with a notation similar to the one found in the paper \cite{gangtud2017}, which unifies the different notions of differentiability on $\mathscr{P}_2(\R^d)$ used in the literature. If $r>0,$ set 
\begin{equation*} 
e[\xi,r] = \sup_{\gamma\in\Gamma(\mu,\nu)}\sup_{\nu\in\mathscr{P}_2(\R^d)} \big\{ \frac{|e(\nu,\xi,\gamma)|}{\|\pi^1-\pi^2\|_{\gamma}} \ \big| \ \|\pi^1-\pi^2\|_{\gamma} \leq r \big\} 
\end{equation*} 
and 
\begin{equation*} e^0[\xi,r] = \sup_{\gamma\in\Gamma_{0}(\mu,\nu)}\sup_{\nu\in\mathscr{P}_2(\R^d)} \big\{ \frac{|e(\nu,\xi,\gamma)|}{\|\pi^1-\pi^2\|_{\gamma}} \ \big| \ \|\pi^1-\pi^2\|_{\gamma} \leq r \big\}.    
\end{equation*} 
Here $\pi^1:\dtorus\times\dtorus\to\dtorus$ denotes projection onto the first component; and $\pi^2$ onto the second.
\begin{defn}
\label{defn:defnofgradient} 
With the preceding notation, we say that $\mathcal{W}$ is differentiable at $\mu$ if 
\begin{equation}\label{eq:limhate} 
\lim_{r\to 0^+} e^0[\xi,r] = 0. 
\end{equation} 
The set of all $\xi\in L^2(\dtorus,\mu)$ for which (\ref{eq:limhate}) holds is denoted  $\partial\mathcal{W}(\mu).$ \end{defn} 
\begin{lemma}
\label{lemma:atmostonedifferential} 
If $\xi\in\partial\mathcal{W}(\mu),$ then so is its projection $\bar{\xi}$ onto $\mathscr{T}_{\mu}\mathscr{P}(\dtorus),$ which is then the unique element of minimal norm in $\partial\mathcal{W}(\mu)$ and is denoted by $$ \nabla_{\mu}\mathcal{W}(\mu).$$ We will call it the \emph{Wasserstein gradient of $\mathcal{W}$ at $\mu.$}  
\end{lemma} 
Its proof can be found in \cite{weakkam}. 
\begin{remark}
\label{remark:todefnofgradient} 
The following is an alternative characterization of a vector field $\xi\in L^2(\dtorus,\mu)$ that satisfies (\ref{eq:limhate}): \begin{equation}\label{eq:eqvwassgradient} \mathcal{W}(\nu) - \mathcal{W}(\mu) - \sup\limits_{\gamma\in\Gamma_0(\mu,\nu)} \int\limits_{\R^d\times\R^d} \xi(x) \cdot (y-x) \gamma(dx,dy) = o(\mathscr{W}(\mu,\nu)). 
\end{equation} 
Likewise, $\xi$ satisfies 
$$
\lim_{r\to 0^+}e[\xi,r]=0
$$
if and only if
\begin{equation*}
\mathcal{W}(\nu) - \mathcal{W}(\mu) - \sup\limits_{\gamma\in\Gamma(\mu,\nu)} \int\limits_{\R^d\times\R^d} \xi(x) \cdot (y-x) \gamma(dx,dy) = o(\mathscr{W}(\mu,\nu)). \quad \qquad \sslash 
\end{equation*} 
\end{remark}
The following lemma will be used in the final section. 
\begin{lemma}
\label{lemma:keylemma} 
With the foregoing notation, if $\mathcal W:\wasstwospace\to \R$ is differentiable at $\mu,$ then 
$$ 
\lim_{r\to 0^+} e[\nabla_{\mu}\mathcal{W}(\mu),r] = 0. 
$$ 
By Remark \ref{remark:todefnofgradient}, this is the same as 
$$ 
\mathcal{W}(\nu) - \mathcal{W}(\mu) - \sup\limits_{\gamma\in\Gamma(\mu,\nu)} \int\limits_{\R^d\times\R^d} \nabla_{\mu}\mathcal{W}(\mu)(x) \cdot (y-x) \gamma(dx,dy) = o(\mathscr{W}(\mu,\nu)).
$$ 
\end{lemma} 
\begin{proof} 
See the Appendix.
\end{proof}
\subsubsection{Twice differentiability}
\label{subsubsection:twicedifferentiability} 
In \cite{chowgangbo}, the notion of Hessian of a function on the Wasserstein space is defined. We will follow the same framework. Let $\rho, \epsilon$ be moduli of continuity, with $\rho$ concave. We will say that a function 
$$
V:\dtorus\times\wasstwospace\to\R $$ 
is \textit{twice differentiable} at $(q,\mu)$ if the following hold:
\begin{itemize}
\item the mapping $x\mapsto \nabla_{\mu}V(q,\nu)(x)$ exists and is differentiable for every $\nu$ in a neighbourhood of $\mu$, with its derivative denoted by $\nabla_{x\mu}^2V(q,\nu)(x)$;
\item the gradient $\nabla_q\nabla_{\mu}V(q,\mu)(x)=:\nabla^2_{q\mu}V(q,\mu)(x)$ exists;
\item there exist a Borel, bounded matrix-valued function $A_{\mu\mu}:(\dtorus)^3\to\R^{d\times d}$ such that 
\begin{align}
& \sup\limits_{\gamma\in\Gamma_0(\mu,\nu)}| \nabla_{\mu}V(\bar{q},\nu)(y) - \nabla_{\mu}V(q,\mu)(x) - \nabla^2_{q\mu}V(q,\mu)(x)(\bar{q}-q) - P_{\gamma}[\mu](q,x,y) | \notag  \\ 
\leq & \  o(|\bar{q}-q|) + \big(\mathscr{W}(\mu,\nu)+|x-y|\big)\big(\rho(\mathscr{W}(\mu,\nu))+\epsilon(|x-y|)\big), \notag
\end{align}
where
\begin{align*}
P_{\gamma}[\mu](q,x,y) = \nabla^2_{x\mu}V(q,\mu)(x)(y-x) + \int_{\dtorus\times\dtorus} A_{\mu\mu}(q,x,a)(b-a)\gamma(da,db).
\end{align*}
\end{itemize}
Without loss of generality, we may suppose that $A_{\mu\mu}(q,x,\cdot)\in\mathscr{T}_{\mu}\wasstwospace$ for all $q,x.$  We put
\begin{align*}
\nabla^2_{\mu\mu}V(q,\mu)(\cdot,\cdot) := A_{\mu\mu}(q,\cdot,\cdot), \qquad q\in\dtorus.
\end{align*}
In regard to the former definition and notation, the following fact will be useful.
\begin{prop}\label{prop:timeder}
Let $V:\dtorus\times\wasstwospace\to\R$ be twice differentiable, in the sense explained above. Let $h\mapsto q^h,$ $h\mapsto x^h,$ be differentiable paths in $\dtorus$ defined on an interval $I,$ and $\mu_h\in AC^2(I;\wasstwospace),$ with $v_h$ a continuous in $h$ velocity vector field for $\mu_h.$ 
\begin{enumerate}[(i)]
\item There exists a set $J\subset I,$ of equal measure to that of $I$, such that, if $h_0\in I,$ then the function $h\mapsto \nabla_{\mu} V(q^h,\mu_h)(x^h)$ is differentiable at $h_0$ and 
\begin{align}
 \frac{d}{dh}[\nabla_{\mu}V(q^h,\mu_h)(x^h)]&\big|_{h=h_0} = \notag \\ = & \ \nabla^2_{qh}V(q^{h_0},\mu_{h_0})(x^{h_0})(q^{h})'|_{h=h_0} + \nabla^2_{xh}V(q^{h_0}),\mu_{h_0})(x^{h_0})(x^{h})'|_{h=h_0} \notag \\ & \  + \int_{\dtorus}\nabla^2_{\mu\mu}V(q^{h_0},\mu_{h_0})(x^{h_0},r)v_{h_0}(r)\mu_{h_0}(dr). \notag 
\end{align}
\item If $\nabla^2_{q\mu}V,$ $\nabla^2_{x\mu}V,$ $\nabla^2_{\mu\mu}V$ are continuous, and the paths $h\mapsto x^{h},$ $h\mapsto q^{h}$ are in $C^1(I),$ then
$$ 
\nabla_{\mu}V(q^b,\mu_b)(x^b)-\nabla_{\mu}V(q^a,\mu_a)(x^a)=  \int_a^b \frac{d}{dh}\nabla_{\mu}V(q^h,\mu_h)(x^h)dh
$$
for any $a,b\in I.$ 
\end{enumerate}
\end{prop}
\begin{proof} 
See the Appendix.
\end{proof}
\section{Main statements}
\label{section:summaryofresults} 
We collect here the three main statements that were proved in this paper. Let $H, F, g$ be as in Section \ref{subsection:dataformfg}. \begin{stm} 
(Theorem \ref{thm:soltomfg}) If $T$ is sufficiently small, in a way that depends only on the coefficients $(H,F,g),$ then, for every $0<s<T,$ $\mu\in\wasstwospace,$ the MFG system (\ref{eq:mfg1}-\ref{eq:mfg4}) admits a classical solution $(U,\sigma)$, in the sense of Section \ref{subsection:defnofclassicalmfg}. Moreover, $(U,\sigma)\in W^{2,2;\infty}((0,T)\times\dtorus)\times AC^2(0,T;\wasstwospace).$ 
\end{stm} 
\begin{stm} 
(Theorem \ref{thm:uniquenessmfg}) If $(\tilde U,\tilde \sigma)\in W^{2,3;\infty}((0,T)\times\dtorus)\times AC^2(0,T;\wasstwospace)$ is a classical solution to the MFG system (\ref{eq:mfg1}-\ref{eq:mfg4}), then, at least during a possibly shorter interval $[0,T]$ than the one in the previous statement, the pair $(\tilde U,\tilde \sigma)$ must be the pair constructed for Theorem \ref{thm:soltomfg}.  
\end{stm} 
Additionally, let $F, g$ be as in Section \ref{subsection:medata}.
\begin{stm} 
(Theorem \ref{thm:masterequation1}) If $T$ is small enough, in a way that depends only on the coefficients $(H,F,g),$ then the master equation (\ref{eq:ME}) admits a classical solution in the sense of Section \ref{subsection:defnofclassicalmfg}. 
\end{stm}  
\section{Solving the MFG system}\label{section:ODEs} In this section, we construct a solution to the first-order MFG system (\ref{eq:mfg1}-\ref{eq:mfg4}). A fixed-point argument will give us existence and uniqueness of solutions to the \emph{characteristics} of the system. These solutions are functions $\Sigma:[0,T]\times\dtorus\to\dtorus\times\R^d$ that depend on $s$ and $\mu$ and are the backbone of our work. They incorporate enough regularity that we can construct classical solutions (in the sense defined above) to the MFG system on $\dtorus.$   
\subsection{System of equations and its solution.} 
For $T>0,$ we will denote by $\mathcal{M}$ the space of continuous functions $$ Z=(Q,P): [0,T]\times\dtorus\longrightarrow\dtorus\times\R^d, $$ endowed with the uniform norm, 
$$ 
\|Z\|_{\infty}=\max_{\substack{0\leq  t\leq      T}}|Z(t,q)|=\max\{(|Q(t,q)|^2+|P(t,q)|^2)^{1/2} \ | \ t\in[0,T],q\in\dtorus \}.
$$ That is, 
$$ 
\mathcal{M}=C([0,T]\times\dtorus;\dtorus\times\R^d).
$$ 
Similarly, let 
$$
\mathcal{M}^1 := C([0,T]\times\dtorus;\dtorus), \quad \mathcal{M}^2 := C([0,T]\times\dtorus;\R^d) .
$$ 
\begin{defn}\label{defn:defofM} 
Let $\theta\in\R^+,$ and fix $\mu\in\wasstwospace,$ $s\in[0,T].$  
\begin{enumerate} 
\item (Fixed point operator) Define the operator $\bar{\mathfrak{m}}^{s,\mu}:\mathcal{M}\to\mathcal{M},$ $\frmbar=((\frmbar^{s,\mu})^1,(\frmbar^{s,\mu})^2)$ as follows: 
\begin{align}
    \textrm{ If } \bar Z = &\ (\bar Q,\bar P) \in \mathcal M , \textrm{ then }
    \bar{\mathfrak{m}}^{s,\mu} (\bar Z) = ((\bar{\mathfrak{m}}^{s,\mu})^1(\bar Z),(\bar{\mathfrak{m}}^{s,\mu})^2(\bar Z)), \textrm{ where:} \notag
    \\ (\frmbar^{s,\mu})^1(\bar Z)(t,q) &\ = q + \int_s^t
    \nabla_pH(\bar Q_{\tau}(q),\theta\bar P_{\tau}(q))d\tau , \label{eq:mathfrak1}
    \\ (\frmbar^{s,\mu})^2(\bar Z)(t,q) &\ =
    \frac{1}{\theta}\nabla_qg(\bar Q_0(q),{{}\bar{Q}_0}_{\#}\mu)-\frac{1}{\theta}\int_0^t\nabla_qH(\bar Q_{\tau}(q),\theta\bar P_{\tau}(q))+\nabla_qF(\bar Q_{\tau}(q),{{}\bar{Q}_{\tau}}_{\#}\mu)d\tau, \label{eq:mathfrak2} \end{align} 
$0\leq t\leq T,$ $q\in\R^d.$ In equalities (\ref{eq:mathfrak1}) and (\ref{eq:mathfrak2}), $\bar Q_{\tau}(q):=\bar Q(\tau,q),$   $\bar P_{\tau}(q):=\bar P(\tau,q),$ $\tau\in[0,T], q\in\R^d.$ \item (Coefficient bounds I) For $B>0,$ let 
\begin{align*} 
\bar{l}(B) &\ := \max_{\substack{q\in \R^d, |p|\leq B, \\ \mu\in\wasstwospace}} \big\{ \sqrt{2}\ |\nabla H(q,\theta p)| + |\nabla_qF(q,\mu)| \big\} , \\ \bar{h}(B) &\ := \max_{\substack{q\in \R^d,|p|\leq B, \\ \mu\in\wasstwospace}}\big\{ \sqrt{2}|\nabla^2 H(q,\theta p)|+\sqrt{2}|\nabla^3H(q,\theta p)| + |\nabla_{qq}^2F(q,\mu)| + |\nabla_{qqq}^3F(q,\mu)| \big\} , \\ & c :=
\max\{d,\kappa\}. 
\end{align*}
Thus, for a fixed $B,$ the numbers $\bar{l}(B),$ $\bar{h}(B), c$ depend only on the coefficients ($H,$ $F,$ $g$). \end{enumerate}  \end{defn} \noindent \textbf{Notes.} (1) Since $\bar Q,$ $\bar P$ are periodic in $q$ (i.e., $q\in\dtorus$), if $q'\sim q$ then $(\bar{\mathfrak{m}}^{s,\mu})^1(\bar Z)(t,q) \sim (\bar{\mathfrak{m}}^{s,\mu})^1(\bar Z)(t,q'),$ so $(\bar{\mathfrak{m}}^{s,\mu})^1(\bar Z)(t,\cdot)$ is indeed a mapping into $\dtorus,$ in the sense explained in the Preliminaries.

(2) Both the fixed-point operator $\bar{\mathfrak{m}}^{s,\mu}$ and the coefficient bounds depend on the value of $\theta.$

(3) Throughout this text, $|\nabla
  H(q,p)|^2=$ $\sum_{j=1}^d\big|\frac{\partial H}{\partial (q^{(j)})^2}(q,p)\big|^2+\sum_{j=1}^d\big|\frac{\partial H}{\partial (p^{(j)})^2}(q,p)\big|^2 , $ and the norms of second  order derivatives are defined similarly, i.e., we are using quadratic norms. $\sslash$ \\

 Suppose that the operator $\bar{\mathfrak{m}}^{s,\mu}$ has a fixed point $(\bar Q,\bar P)$, so on the left-hand side of (\ref{eq:mathfrak1}) and (\ref{eq:mathfrak2}) we would see $\bar Q(t,q)$ and $\bar P(t,q)$ respectively. Set $Q := \bar Q$ and $P := \theta\bar P.$ Then $Z:=(Q,P)$ satisfies 
 \begin{equation}\label{eq:hamODEs}
 \qquad \left\{ 
 \begin{aligned} 
 \partial_t Q(t,q) = &\ \nabla_p H(Q(t,q),P(t,q)) \qquad \textrm{ in } [0,s]\times\dtorus \ ,  \\ \partial_t P(t,q) = &\ -\nabla_qH(Q(t,q),P(t,q)) -\nabla_qF(Q(t,q),Q(t,\cdot)_{\#}\mu) \quad \textrm{ in } [0,s]\times\dtorus \ , \\ Q(s,q) = &\ q \qquad \textrm{ on } \dtorus \ ,  \\ P(0,q) =  &\ \nabla_qg(Q(0,q),Q(0,\cdot)_{\#}\mu) \qquad \textrm{ on } \dtorus.  
 \end{aligned}    \right.   
 \end{equation} 
 We will refer to the system (\ref{eq:hamODEs}) as the \textit{Hamiltonian ODEs with parameters $s$ and $\mu.$}   \begin{defn}\label{defn:Msubzero} 
 If $T>0,$ $A_1,A_2,B,E,E_1,E_2>0,$ define 
 $$ 
 \mathcal{M}_0(A_1,A_2,B,E,E_1,E_2,T)\subset\mathcal{M}
 $$ 
 to be the subset of those $\bar Z(\cdot,\cdot)=(\bar Q(\cdot,\cdot),\bar P(\cdot,\cdot))$  such that: 
 
 (i) $\bar Z(\cdot,\cdot)$ belongs to $W^{2,2;\infty}([0,T]\times\dtorus;\dtorus\times\R^d)$;
 
 (ii) the following bounds hold: 
 \begin{equation}
 \label{eq:defofM0} 
 \left\{ \begin{aligned} & \|\partial_t\bar Q\|_{\mathcal{M}^1}\leq A_1 , \  \|\nabla_q \bar Q\|_{C([0,T]\times\dtorus;\dtorus\times\dtorus)} \leq A_1 , \ \|\nabla^2_{qq}\bar Q\|_{C([0,T]\times\dtorus;\mathbb{T}^{2d}\times\dtorus)} \leq A_1   ; \\ & \|\partial_t\bar{P}\|_{\mathcal{M}^2}\leq A_2 , \ \ \|\nabla_q \bar P\|_{C([0,T]\times\dtorus;\R^d\times\R^d)} \leq A_2 , \ \|\nabla^2_{qq} \bar P\|_{C([0,T]\times\dtorus;\mathbb{R}^{2d}\times\R^d)} \leq A_2 \     \\  &  \| \bar P \|_{\mathcal{M}^2}\leq B  ; \\ & \|\nabla_q\bar Q_0\|_{C(\dtorus;\dtorus\times\dtorus)}  
 , \|\nabla_{qq}^2\bar Q_0\|_{C(\dtorus;\mathbb{T}^{2d}\times\dtorus)}  \leq E   ;\ \end{aligned} \right.  \end{equation}

(iii) $\|\partial^2_{tt}\bar{Q}\|_{\mathcal{M}} \leq E_1, \|\partial^2_{tt}\bar{P}\|_{\mathcal{M}} \leq E_2. $ 
\end{defn}  
Here $W^{2,2;\infty}([0,T]\times\dtorus;\dtorus\times\R^d)$ is the Sobolev space of functions periodic in $q,$ taking values in $\dtorus\times\R^d,$ with essentially bounded second-order weak derivatives in $t$ and second-order weak gradients in $q.$ 
Since functions in $W^{1,1;\infty}$ are Lipschitz, $\mathcal{M}_0(A_1,A_2,B,E,E_1,E_2,T)$ is indeed a subset of $\mathcal{M}.$ The following is a standard fact, so we will omit its proof.
 \begin{prop}
 \label{prop:M0isclosed} 
 For any $A_1,A_2,B,E,E_1,E_2,T>0,$ $\mathcal{M}_0(A_1,A_2,B,E,E_1,E_2,T)$ is closed in   $\mathcal{M}.$ 
 \end{prop} 
\begin{lemma}
\label{lemma:existenceofABE} 
Let $\theta>0$ of Definition \ref{defn:defofM} be arbitrary. There exist $A_1, A_2, B, E , E_1 , E_2 > 0,$ and $T>0,$ such that $\frmbar^{s,\mu}$ maps $\mathcal{M}_0(A_1,A_2,B,E,E_1,E_2,T)$ into itself, for any $s\in(0,T)$ and $\mu\in\wasstwospace.$ The numbers $A_1, A_2, B, E, E_1, E_2, T$ depend only on the coefficients.
\end{lemma} 
\begin{proof} 
Observe that\footnote{We make a convention here and in the rest of the paper that in the application of the classical chain rule, and only if are concerned solely about estimates, juxtaposition is enough, i.e., we will not pay attention to the order of the factors or whether they are properly transposed.}  \begin{equation*} |(\frmbar^{s,\mu})^2(\bar Z)(t,q)|\leq \frac{1}{\theta}|\nabla_q g(\bar Q_0(q),{{}\bar{Q}_0}_{\#}\mu)| + t \frac{1}{\theta} \underset{0\leq\tau\leq t}{\sup}[|\nabla_q H(\bar  Q_{\tau}(q),\theta\bar P_{\tau}(q))| + |\nabla_q F(\bar Q_{\tau}(q),{{}\bar{Q}_{\tau}}_{\#}\mu ) | ], \end{equation*}  
\begin{align*} |\partial_t(\frmbar^{s,\mu})^1(\bar Z)(t,q))| \leq &\ |\nabla_p H(\bar Q_t(q),\theta\bar P_t(q))|\ , \\ |\partial_t(\frmbar^{s,\mu})^2(\bar Z)(t,q)| \leq &\  \frac{1}{\theta} |\nabla_qH(\bar Q_{t}(q),\theta \bar P_{t}(q))| + \frac{1}{\theta} |\nabla_q F(\bar Q_{t}(q),{{}\bar{Q}_{t}}_{\#}\mu) | \ ;  \end{align*} 
\begin{equation*} |\nabla_q(\frmbar^{s,\mu})^1(\bar Z)(t,q)| \leq \sqrt{d} + \int_s^t \big| \nabla^2_{qp}H(\bar Q_{\tau}(q),\theta\bar P_{\tau}(q))\nabla_q \bar Q_{\tau}(q) + \nabla^2_{pp}H(\bar Q_{\tau}(q),\theta \bar P_{\tau}(q)) \theta \nabla_q \bar P_{\tau}(q)  \big| d\tau \ ;   \end{equation*} 
\begin{align*} 
  |\nabla_q(\frmbar^{s,\mu})^2(\bar Z)(t,q&)| \\ \leq &\ \frac{1}{\theta} |\nabla^2_{qq}g(\bar Q_0(q),{{}\bar{Q}_0}_{\#}\mu)\nabla_q \bar Q_0(q)| + \frac{1}{\theta} \int_0^t |\nabla^2_{qq}F(\bar Q_{\tau}(q),{{}\bar{Q}_{\tau}}_{\#}\mu)\nabla_q\bar Q_{\tau}(q)| d\tau   \\ & \  +  \frac{1}{\theta} \int_0^t \big| \nabla^2_{qq}H(\bar Q_{\tau}(q),\theta\bar P_{\tau}(q))\nabla_q\bar Q_{\tau}(q) + \nabla^2_{pq}H(\bar Q_{\tau}(q),\theta \bar P_{\tau}(q))\theta \nabla_q \bar P_{\tau}(q)  \big| d\tau
\end{align*} 
The previous lines are inequalities for the moduli of $Q,$ $P,$ and their derivatives. Let us also compute second-order derivatives to find: 
\begin{align*}  
|\nabla^2_{qq}(\frmbar^{s,\mu})^1&(\bar Z)(t,q)| \\ \leq &  \int_s^t \big| ( \nabla^3_{qqp}H(\bar Q_{\tau}(q),\theta\bar P_{\tau}(q))\nabla_q \bar Q_{\tau}(q)+\nabla^3_{pqp}H(\bar Q_{\tau}(q),\theta\bar P_{\tau}(q))\theta\nabla_q\bar P_{\tau}(q) )\nabla_q \bar Q_{\tau}(q)   \\ & \quad  +   \nabla^2_{qp}H(\bar Q_{\tau}(q),\theta\bar P_{\tau}(q))\nabla^2_{qq}\bar Q_{\tau}(q)   \\ & \quad +  ( \nabla^3_{qpp}H(\bar Q_{\tau}(q),\theta\bar P_{\tau}(q))\nabla_q \bar Q_{\tau}(q)+\nabla^3_{ppp}H(\bar Q_{\tau}(q),\theta\bar P_{\tau}(q))\theta \nabla_q \bar P_{\tau}(q) )\theta \nabla_q \bar P_{\tau}(q)   \\  & \quad  + \nabla^2_{pp}H(\bar Q_{\tau}(q),\theta\bar P_{\tau}(q))\theta \nabla^2_{qq}\bar P_{\tau}(q)  \big|  d\tau  
\end{align*} 
for the first component of $\frmbar$, and, for the second component, 
\begin{align*} 
|\nabla^2_{qq}(\frmbar^{s,\mu})^2&(\bar Z)(t,q)|\\   \leq &  \ \frac{1}{\theta} \big| (\nabla^3_{qqq}g(\bar Q_0(q),{{}\bar{Q}_0}_{\#}\mu)\nabla_q \bar Q_0(q))\nabla_q \bar Q_0(q)  +  \nabla^2_{qq}g(\bar Q_0(q),{{}\bar{Q}_0}_{\#}\mu)\nabla^2_{qq}\bar Q_0(q) \big|   \\        & +   \frac{1}{\theta} \int_0^t \big| ( \nabla^3_{qqq}H(\bar Q_{\tau}(q),\theta \bar P_{\tau}(q))\nabla_q\bar Q_{\tau}(q)+\nabla^3_{pqq}H(\bar Q_{\tau}(q),\theta\bar P_{\tau}(q))\theta \nabla_q\bar P_{\tau}(q) )\nabla_q \bar Q_{\tau}(q)  \\ & \qquad  \qquad   + \nabla^2_{qq}H(\bar Q_{\tau}(q),\theta\bar P_{\tau}(q))\nabla^2_{qq}\bar Q_{\tau}(q)  \\ &\  \qquad \qquad + ( \nabla^3_{qpq}H(\bar Q_{\tau}(q),\theta\bar P_{\tau}(q))\nabla_q \bar Q_{\tau}(q)+\nabla^3_{ppq}H(\bar Q_{\tau}(q),\theta\bar P_{\tau}(q))\theta \nabla_q\bar P_{\tau}(q) )\theta \nabla_q \bar P_{\tau}(q)  \\ & \ \qquad \qquad + \nabla^2_{pq}H(\bar Q_{\tau}(q),\theta\bar P_{\tau}(q))\theta \nabla^2_{qq}\bar P_{\tau}(q)\big| d\tau  \\ &  + \frac{1}{\theta} \int_0^t \big| (\nabla^3_{qqq}F(\bar Q_{\tau}(q),{{}\bar{Q}_{\tau}}_{\#}\mu)\nabla_q\bar Q_{\tau}(q))\nabla_q\bar Q_{\tau}(q)  +  \nabla^2_{qq}F(\bar Q_{\tau}(q),{{}\bar{Q}_{\tau}}_{\#}\mu)\nabla^2_{qq}\bar Q_{\tau}(q)   \big|  d\tau.    
\end{align*} 
We deal with $A_1, A_2, B, E, T$ first. Let $A_1,$ $A_2,$ $B,$ $E,$ $T$ be for the moment arbitrary positive numbers. Suppose that $\bar Z\in\mathcal{M}_0,$ that is, $\bar Z=(\bar P,\bar Q)$ satisfies (\ref{eq:defofM0}). From the latter inequalities we see that: \begin{align*} 
\textrm{(a)} & \quad c/\theta  + T \bar{h}(B)/\theta \leq B \  \qquad  \textrm{ implies } \ |(\frmbar^{s,\mu})^2(\bar Z)| \leq B ;  \\  \textrm{(b1)} & \quad \bar{l}(B) \leq A_1 \  \qquad \quad  \textrm{ implies } \ |\partial_t(\frmbar^{s,\mu})^1(\bar Z)|  \leq A_1   ; \\ \textrm{(b2)} & \quad \bar{l}(B)/\theta \leq A_2 \  \qquad \quad  \textrm{ implies } \ |\partial_t(\frmbar^{s,\mu})^2(\bar Z)|  \leq A_2   ; \end{align*} \begin{align*} \textrm{(c)} & \quad c + T \bar{h}(B)(A_1 +\theta A_2)  \leq A_1 \  \qquad   \textrm{ implies } \ |\nabla_q(\frmbar^{s,\mu})^1(\bar Z)| \leq A_1 ; \\ \textrm{(d)} & \quad cE/\theta + \frac{T}{\theta}\bar{h}(B)(A_1+\theta A_2) \leq A_2 \  \qquad  \textrm{ implies } \ |\nabla_q(\frmbar^{s,\mu})^2(\bar Z)| \leq A_2 ; \\ \textrm{(e)} & \quad c + T\bar{h}(B)(A_1+\theta A_2) \leq E \  \qquad  \textrm{ implies} \ |\nabla_q(\frmbar^{s,\mu})^1(\bar Z)|\big|_{t=0} \leq E ; \end{align*} \begin{align*} \textrm{(f1)} & \quad T\bar{h}(B) (A_1+\theta A_2)(A_1+\theta A_2+1) \leq A_1 \ \quad \textrm{ implies } \ |\nabla^2_{qq}(\frmbar^{s,\mu})^1(\bar Z)| \leq A_1 ; \\ \textrm{(f2)} & \quad T\bar{h}(B) (A_1+\theta A_2)(A_1+\theta A_2+1) \leq E \ \quad \textrm{ implies } \ |\nabla^2_{qq}(\frmbar^{s,\mu})^1(\bar Z)|\big|_{t=0} \leq E ; \\  \textrm{(g)} & \quad \frac{1}{\theta }cE(E+1) + \frac{T}{\theta }\bar{h}(B)(A_1+\theta A_2)(A_1+\theta A_2+1) + \frac{T}{\theta }\bar{h}(B)A_1(A_1+1) \leq A_2 \\   & \qquad \ \textrm{ implies } \ |\nabla^2_{qq}(\frmbar^{s,\mu})^2(\bar Z)| \leq A_2  . 
\end{align*} 
We need to set $A_1, A_2, B, E, T$ so that the above inequalities hold simultaneously. First choose $B> c/\theta .$ The number $B$ now depends only on the coefficients and $\theta$ (through $c$), and thus $\bar{l}(B), \bar{h}(B)$ depend only on the coefficients and $\theta$, through $B.$ Let $T$ be small enough that $c/\theta +(T/\theta )\bar{l}(B)\leq B$ (i.e.~$T<(\theta B-c)/\bar{l}(B)$). This gives (a). Choose $E$ to be any number such that $E>c,$ and pick $A_1, A_2$ such that \begin{align} A_1 > &\ \max\{\bar{l}(B),E,c \} , \label{eq:A1}  \\ A_2 > &\ \max\{\frac{\bar{l}(B)}{\theta }, \frac{1}{\theta }cE(E+1)\}. \label{eq:A2}  \end{align} This gives (b1), (b2). Making $T$ possibly smaller by letting \begin{align} T < &\ R:=  \min\big\{ \frac{\theta B-c}{\bar{l}(B)}, \frac{E-c}{\bar{h}(B)(A_1+\theta A_2}, \notag  \\ & \quad \qquad  \frac{A_2-cE(E+1)/\theta }{(1/\theta )\bar{h}(B)[(A_1+\theta A_2)(A_1+\theta A_2+1)+A_1(A_1+1)]}, \frac{E}{\bar{h}(B)(A_1+\theta A_2)(A_1+\theta A_2+1)}   \big\} , \label{eq:defnofR} \end{align} 
we make sure that: (e) holds, and, consequently, (c) holds because $A_1 > E$; (g) holds and therefore (d) as well; and (f2) holds, hence, (f1) is true.
Finally, let us answer the existence of the constants $E_1, E_2.$ We compute $\partial^2_{tt}(\frmbar^{s,\mu})^1(\bar Z)(t,q)$ to be  
\begin{align*}
\partial^2_{tt}(\frmbar^{s,\mu})^1(\bar Z)(t,q) = \nabla^2_{pq}H(\bar{Q}_t(q),\theta\bar{P}_t(q))\partial_t\bar{Q}_t(q) + \theta\nabla^2_{pp}H(\bar{Q}_t(q),\theta\bar{P}_t(q))\partial_t\bar{P}_t(q)
\end{align*}
and, thanks to the conditions in Section \ref{subsection:dataformfg}, we get for $\partial^2_{tt}(\frmbar^{s,\mu})^2(\bar Z)(t,q)$:
\begin{align*}
\partial^2_{tt}(\frmbar^{s,\mu})^2(\bar Z)(t,q) = &\ -\frac{1}{\theta} \nabla^2_{qq}H(\bar{Q}_t(q),\theta\bar{P}_t(q))\partial_t\bar{Q}_t(q) + \nabla^2_{pq}H(\bar{Q}_t(q),\theta\bar{P}_t(q))\partial_t\bar{P}_t(q) 
\\
& \  - \frac{1}{\theta}\nabla^2_{qq}F(\bar{Q}_t(q),{{}\bar{Q}_t}_{\#}\mu)\partial_t\bar{Q}_t(q) 
- \frac{1}{\theta}\int_{\dtorus} \nabla^2_{\mu q}F(\bar{Q}_t(q),{{}\bar{Q}_t}_{\#}\mu)(\bar{Q}_t(x))\partial_t\bar{Q}_t(x)\mu(dx).
\end{align*}
Therefore, it is enough to choose $E_1, E_2$ large enough such that 
\begin{align*}
\bar{h}(B)A_1 + \theta\bar{h}(B)A_2 \leq & \ E_1/\sqrt{2} ,
\\
\frac{1}{\theta}\bar{h}(B)A_1 + \bar{h}(B)A_2 + \frac{1}{\theta}c A_1 + \frac{1}{\theta}\|\nabla^2_{\mu q}F\|_{\infty}A_1  \leq & \ E_2/\sqrt{2}.
\end{align*}
\end{proof}   
\begin{prop}
\label{prop:m0islipschitz} (Contraction property) 
Let $\theta >2\kappa.$ Then there exist positive numbers $A_1$,$A_2,$ $B,$ $E,$ $E_1,$ $E_2,$ $T$ such that for any $s\in[0,T],$ $\mu\in\wasstwospace,$ the operator $\frmbar^{s,\mu}$ maps $\mathcal{M}_0(A,B,E,E_1,E_2,T)$ into itself and is a contraction. 
\end{prop} 
\begin{proof} 
We run the previous lemma to obtain the numbers $A_1,$ $A_2,$ $B,$ $E,$ $E_1$, $E_2,$ and $T,$ and decrease $T,$ if necessary, so that \begin{equation}\label{eq:TinABElemma} T < \min\big\{ R, \frac{1-\frac{2\kappa}{\theta }}{\bar{h}(B)\sqrt{2}(1+1/\theta )+2\kappa/\theta} \big\}, \end{equation} where $R$ is the number defined in (\ref{eq:defnofR}).   Let $\bar Z=(\bar Q,\bar P), \bar Z'=(\bar Q',\bar P')\in \mathcal{M}_0.$ Let $s\in[0,T],$ $\mu\in\wasstwospace$ be arbitrary. We have, for the first component of $\frmbar^{s,\mu},$ that $$ |(\frmbar^{s,\mu})^1(\bar Z)(t,q)-(\frmbar^{s,\mu})^1(\bar Z')(t,q)| \leq |s-t|\underset{t\leq\tau\leq s}{\max}|\nabla_pH(\bar Q_{\tau}(q),\theta \bar P_{\tau}(q))-\nabla_pH(\bar Q'_{\tau}(q),\theta \bar P'_{\tau}(q))|. $$ Since $H$ is $C^2,$ we can write  $$ |\nabla_pH(\bar Q_{\tau}(q),\theta \bar P_{\tau}(q))-\nabla_pH(\bar Q'_{\tau}(q),\theta  \bar P'_{\tau}(q))| \leq M^1_{\tau,q}|\bar Z(\tau,q)-\bar Z'(\tau,q)|, $$ where $$M^1_{\tau,q} = \max_{0\leq\lambda\leq 1}|\nabla(\nabla_p H)[(1-\lambda)\bar Q_{\tau}(q)+\lambda \bar Q'_{\tau}(q) ,(1-\lambda)\theta \bar P_{\tau}(q)+\lambda \theta  \bar P'_{\tau}(q) ]| . $$ For the second component of $\frmbar$ we apply  (\ref{eq:remarkfg}) to get \begin{align*} & |(\frmbar^{s,\mu})^2(\bar Z)(t,q)-(\frmbar^{s,\mu})^2(\bar Z')(t,q)| \\ &  \leq  \frac{\kappa}{\theta }\sqrt{|\bar Q_0(q)-\bar Q_0'(q)|^2+\|\bar Q_0-\bar Q_0'\|^2_{L^2(\mu)}}\     + \frac{1}{\theta }t  \underset{0\leq\tau\leq t}{\max}\big[ |\nabla_q H(\bar Q_{\tau}(q),\theta \bar P_{\tau}(q))-\nabla_q H(\bar Q'_{\tau}(q),\theta \bar P'_{\tau}(q))|  \\ & \quad  +  \kappa\sqrt{|\bar Q_{\tau}(q)-\bar Q_{\tau}'(q)|^2+\|\bar Q_{\tau}-\bar Q_{\tau}'\|^2_{L^2(\mu)}} \ \big] \\ &\ \leq \sqrt{2}\frac{\kappa}{\theta }(1+t)\|\bar Z-\bar Z'\|_{\infty}     + \frac{t}{\theta }\underset{0\leq\tau\leq t}{\max} |\nabla_q H(\bar Q_{\tau}(q),\theta \bar P_{\tau}(q))-\nabla_q H(\bar Q'_{\tau}(q),\theta \bar P'_{\tau}(q))| , \end{align*} with $$ |\nabla_q H(\bar Q_{\tau}(q),\theta \bar P_{\tau}(q))-\nabla_q H(\bar Q'_{\tau}(q),\theta \bar P'_{\tau}(q))| \leq M^2_{\tau,q}|\bar Z(\tau,q)-\bar Z'(\tau,q)|, $$ where $$ M^2_{\tau,q} = \max_{0\leq\lambda\leq 1}|\nabla(\nabla_q H)[(1-\lambda)\bar Q_{\tau}(q)+\lambda \bar Q'_{\tau}(q) ,(1-\lambda)\theta \bar P_{\tau}(q)+\lambda \theta \bar P'_{\tau}(q) ]|. $$ But, since $\mathcal{M}_0$ is a convex subset of $\mathcal{M},$ it is true that $M^1_{\tau,q}, M^2_{\tau,q} \leq \bar{h}(B),$ $(\tau,q)\in[0,s]\times\dtorus.$  It follows that $$ |(\frmbar^{s,\mu})^1(\bar Z)(t,q)-(\frmbar^{s,\mu})^1(\bar  Z')(t, q)| \leq |s-t|\bar{h}(B)\|\bar Z - \bar Z'\|_{\infty} \ , \quad 0\leq t\leq s , \  q\in\dtorus,  $$ and  $$ |(\frmbar^{s,\mu})^2(\bar Z)(t,q)-(\frmbar^{s,\mu})^2(\bar Z')(t,q)| \leq \sqrt{2}\frac{1}{\theta }\kappa(1+t)\|\bar Z-\bar Z'\|_{\infty} + \frac{1}{\theta }t \bar{h}(B)\|\bar Z-\bar Z'\|_{\infty}.$$ Consequently, since $0\leq t,s\leq T,$ we obtain 
\begin{align} \|\frmbar^{s,\mu}(\bar Z)-\frmbar^{s,\mu}(\bar Z')\|_{\infty} &\ \leq \sqrt{2}T\bar{h}(B) + \sqrt{2}(\sqrt{2}\frac{1}{\theta }\kappa (1+T) + \frac{1}{\theta } T\bar{h}(B)) \|\bar Z-\bar Z'\|_{\infty} \notag \\ &\ = \big[\frac{2\kappa}{\theta } + T\big( \bar{h}(B)\sqrt{2}(1+\frac{1}{\theta })+\frac{2\kappa}{\theta }\big)\big] \|\bar Z-\bar Z'\|_{\infty}.  \label{eq:mislipschitz} \end{align} Due to (\ref{eq:TinABElemma}), the expression inside the square brackets in (\ref{eq:mislipschitz}) is less than $1.$  
\end{proof} 
It follows now that the operator (\ref{eq:mathfrak1}) and (\ref{eq:mathfrak2}) has a unique fixed point in $\mathcal{M}_0(A_1,A_2,B,E,E_1,E_2,T),$ where $A_1,A_2,B,E,E_1,E_2,T$ are as above. 
\begin{defn}
\label{defn:mathfrakm} 
Fix $\mu\in\wasstwospace,$ $s\in[0,T].$ 
Define the operator $\mathfrak{m}^{s,\mu}:\mathcal{M}\to\mathcal{M},$ $\frm=((\frm^{s,\mu})^1,(\frm^{s,\mu})^2)$ as follows: 
\begin{align}
    \textrm{ If } Z = &\ (Q,P) \in \mathcal M , \textrm{ then }
    \mathfrak m^{s,\mu} (Z) = ((\mathfrak m^{s,\mu})^1( Z),(\mathfrak
    m^{s,\mu})^2( Z)), \textrm{ where :} \notag
    \\ (\frm^{s,\mu})^1(Z)(t,q) &\ = q + \int_s^t
    \nabla_pH(Q_{\tau}(q),P_{\tau}(q))d\tau , \label{eq:mathfrak1again}
    \\ (\frm^{s,\mu})^2(Z)(t,q) &\ =
    \nabla_q g( Q_0(q),{{}Q_0}_{\#}\mu)-\int_0^t\nabla_q H(Q_{\tau}(q),P_{\tau}(q))+\nabla_qF(Q_{\tau}(q),{{}Q_{\tau}}_{\#}\mu)d\tau, \label{eq:mathfrak2again} 
    \end{align} 
    $0\leq t\leq T,$ $q\in\dtorus.$ 
    \end{defn} 
\begin{cor}
\label{cor:hamODEshassolution} 
\begin{enumerate}[(i)] \item For any $T>0,$ $s\in[0,T], \mu\in\wasstwospace,$ $\theta >0,$ the operator $\bar{\mathfrak{m}}$ has a unique fixed point in $\mathcal{M}_0(A_1,A_2,B,E,E_1,E_2,T)$ if, and only if, $\mathfrak{m}$ has a unique fixed point in \newline $\mathcal{M}_0(A_1,\theta A_2 ,\theta B ,E,E_1,\theta E_2,T).$ \item With $\theta >2\kappa,$ fix $s\in[0,T],$ $\mu\in\wasstwospace.$ Let $A_1,A_2,B,E,E_1,E_2,T$ be as obtained in Lemma \ref{lemma:existenceofABE}. Then $\mathfrak{m}^{s,\mu}$ maps $\mathcal{M}_0(A_1,\theta A_2 ,\theta B ,E,E_1,\theta E_2,T)$ into itself and the system (\ref{eq:hamODEs}) has a unique solution in \newline $\mathcal{M}_0(A_1,\theta A_2 ,\theta B ,E,E_1,\theta E_2,T).$ 
\end{enumerate}
\end{cor} 
\begin{proof} 
\begin{enumerate}[(i)] \item Suppose $\bar{\mathfrak{m}}^{s,\mu}$ has a unique fixed point $\bar{\Sigma}[s,\mu]=(\bar{\Sigma}^1[s,\mu],\bar{\Sigma}^2[s,\mu])$ that satisfies the bounds of Definition \ref{defn:Msubzero} with $\bar Q=\bar \Sigma^1$ and $\bar P=\bar \Sigma^2.$  Define \begin{equation}\label{eq:SigmabartoSigma} \Sigma^2[s,\mu] := \bar{\Sigma}^1[s,\mu], \quad \Sigma^2[s,\mu] := \theta \bar{\Sigma}^2[s,\mu]. \end{equation} Then it is straighforward to check that $\Sigma[s,\mu] = (\Sigma^1[s,\mu],\Sigma^2[s,\mu])$ is the unique fixed point of the operator $\mathfrak{m}^{s,\mu}$ such that the inequalities of Definition \ref{defn:Msubzero} are true for $\Sigma^1,$ $\Sigma^2$ with the new constants $\theta A_2 ,$ $\theta B,$ $\theta E_2$ in place of $A_2$ and $B,$ $E_2$ respectively, that is, such that $Q=\Sigma^1,$ $P=\Sigma^2$ satisfy 
\begin{equation}
\label{eq:defofM0theta} 
\left\{ 
\begin{aligned} 
& \|\partial_t Q\|_{\mathcal{M}^1}\leq A_1 , \  \|\nabla_q  Q\|_{C([0,T]\times\dtorus;\dtorus\times\dtorus)} \leq A_1 , \ \|\nabla^2_{qq} Q\|_{C([0,T]\times\dtorus;\mathbb{T}^{2d}\times\dtorus)} \leq A_1   ; \\ & \|\partial_t P\|_{\mathcal{M}^2}\leq \theta A_2 , \ \ \|\nabla_q  P\|_{C([0,T]\times\dtorus;\R^d\times\R^d)} \leq \theta A_2 , \ \|\nabla^2_{qq} P\|_{C([0,T]\times\dtorus;\mathbb{R}^{2d}\times\R^d)} \leq \theta A_2; \     \\  &  \| P \|_{\mathcal{M}^2}\leq \theta B  ; \\ & \|\nabla_q Q_0\|_{C(\dtorus;\dtorus\times\dtorus)}  
 , \|\nabla_{qq}^2 Q_0\|_{C(\dtorus;\mathbb{T}^{2d}\times\dtorus)}  \leq E   , \\
 & \|\partial^2_{tt} Q\|_{\mathcal{M}} \leq E_1, \|\partial^2_{tt} P\|_{\mathcal{M}} \leq \theta E_2. 
 \end{aligned}
 \right.  
 \end{equation}  
 The sufficiency part of the statement is equally easily verified. 
 \item 
 We take $Z=(Q,P)\in\mathcal{M}_0(A_1,\theta A_2 ,\theta B ,E,E_1,\theta E_2,T),$ and evaluate $\bar{\mathfrak{m}}^{s,\mu}$ at $\bar Z = (\bar Q, \bar P)$ where $\bar Q = Q$ and $\bar P = P/\theta ,$ $\bar Z\in\mathcal{M}_0(A_1,A_2,B,E,E_1,E_2,T).$ Then, by Lemma \ref{lemma:existenceofABE}, \newline $\bar{\mathfrak{m}}^{s,\mu}(\bar Z)\in \mathcal{M}_0(A_1,A_2,B,E,E_1,E_2,T).$ But $(\bar{\mathfrak{m}}^{s,\mu})^1(\bar Z)=(\mathfrak{m}^{s,\mu})^1(Z)$ and $(\bar{\mathfrak{m}}^{s,\mu})^2(\bar Z)=\frac{1}{\theta }(\mathfrak{m}^{s,\mu})^1(Z)$, thus, $\mathfrak{m}^{s,\mu}(Z) \in \mathcal{M}_0(A_1,\theta A_2 ,\theta B ,E,E_1,\theta E_2,T).$ Furthermore, Proposition \ref{prop:m0islipschitz} provides a unique fixed point $\bar\Sigma[s,\mu]$ of $\bar{\mathfrak{m}}$ in $\mathcal{M}_0(A_1,A_2,B,E,E_1,E_2,T).$ Defining $\Sigma[s,\mu]$ as in (\ref{eq:SigmabartoSigma}), then, by (i), $\Sigma[s,\mu]$ is the unique fixed point of the operator $\mathfrak{m}^{s,\mu}$ on $\mathcal{M}_0(A_1,\theta A_2 ,\theta B ,E,E_1,\theta E_2,T)$, so it is the unique solution to (\ref{eq:hamODEs}) in $\mathcal{M}_0(A_1,\theta A_2 ,\theta B ,E,E_1,\theta E_2,T)$. 
\end{enumerate}
\end{proof}  
We stress that, as long as $T$ is as in (\ref{eq:TinABElemma}), we have a solution $\Sigma[s,\mu]$ to (\ref{eq:hamODEs}) for any $\mu\in\wasstwospace,$ and $0\leq s\leq T,$  and moreover, its $Q$ and $P$ components $\Sigma^1[s,\mu]$ and $\Sigma^2[s,\mu]$ satisfy the bounds of Definition \ref{defn:Msubzero} with $\theta A_2,$ and $\theta B,$ $\theta E_2$ in place of $A,$ $B,$ $\theta E_2$ respectively, \emph{independently of} $\mu,$ with solutions being continuous and differentiable in $t$ and $q$.
\begin{center}
\emph{\textbf{$\Sigma[s,\mu]$ will always denote such solutions to (\ref{eq:hamODEs}) as found in Corollary \ref{cor:hamODEshassolution}}.}
\end{center}
\begin{remark}
\label{remark:remarkaboutconstantT} 
The preceding proofs make it clear that $T$ can be assumed to be smaller if necessary at each following step, without affecting the validity of the previous statements. We choose to refer back to this remark in later stages instead of imposing tighter bounds on $T$ than (\ref{eq:TinABElemma}) above that would make their purpose unclear at first reading. We may sometimes just say ``$T$ is small'', having this remark in mind. $ \sslash $
\end{remark} 
\subsection{First regularity properties}\label{subsubsection:firstregularityproperties}
\begin{lemma}
\label{lemma:continuityinsandmu} 
For any fixed $Z=(Q,P)\in\mathcal M,$ $t\in[0,T],$ $q\in\R^d,$ and $\mu\in\wasstwospace,$ the function $$ s\mapsto \mathfrak{m}^{s,\mu}(Z)(t,q) $$ is continuous. Likewise, for any fixed $Z\in\mathcal M,$ $t\in[0,T],$ $q\in\R^d,$ and $s\in[0,T],$ the function 
$$ 
\mu\mapsto \mathfrak{m}^{s,\mu}(Z)(t,q) 
$$ 
is continuous. 
\end{lemma} 
\begin{proof} 
Continuity of $ s\mapsto \mathfrak{m}^{s,\mu}(Z)(t,q)$ for fixed $t,q,\mu$ is immediate from Definition \ref{defn:mathfrakm} (formulas (\ref{eq:mathfrak1again}) and (\ref{eq:mathfrak2again})). Looking at the same definition for the continuity with respect to $\mu,$ note that for each $\tau,$ $0\leq \tau\leq T,$ the function $q\mapsto Q_{\tau}(q)$ is Lipschitz. This implies (see, e.g., \cite[Remark 3.3]{mfgmain}) that the function $\mu\mapsto {{}Q_{\tau}}_{\#}\mu$ is Lipschitz from $\wasstwospace$ into itself, with the same constant. Furthermore, since the mapping $(\tau,q)\mapsto Q(\tau,q)$ is Lipschitz, the Lipschitz constants of the functions $q\mapsto Q_{\tau}(q)$  are bounded with respect to $\tau,$ $0\leq \tau \leq T.$ These facts, combined with the Lipschitz continuity of $\nabla_q F$ and $\nabla_q g,$ show that 
\begin{equation*}
\mu_n\to \mu \textrm{ in } \wasstwospace \quad \Longrightarrow \quad \mathfrak m^{s,\mu_n}(Z)(t,q) \to \mathfrak m^{s,\mu}(Z)(t,q)   
\end{equation*} 
for all $Z\in \mathcal{M}_0,$ $t,s\in [0,T],$ $q\in\R^d.$ 
\end{proof} 
We will need the continuity and differentiability of the fixed point $\Sigma[s,\mu]$  with respect to $s,$ and its continuity with respect to $\mu.$   This is addressed in Lemmas \ref{lemma:continuityofSigmains} and \ref{lemma:continuityofSigmainmu} below. Before that, let us name the coefficient bounds that will appear in the calculations.
\begin{defn}
\label{defn:databoundsii} 
(Coefficient bounds II) For $B>0,$ let \begin{align*} l(B) &\ := \max_{\substack{q\in \R^d, |p|\leq B, \\ \mu\in\wasstwospace}} \sqrt{2}\ |\nabla H(q,p)| + |\nabla_qF(q,\mu)| , \\ h(B) &\ := \max_{\substack{q\in \R^d,|p|\leq B, \\ \mu\in\wasstwospace}}\sqrt{2}|\nabla^2 H(q,p)|+\sqrt{2}|\nabla^3H(q,p)| + |\nabla_{qq}^2F(q,\mu)| + |\nabla_{qqq}^3F(q,\mu)|. \end{align*}  \end{defn} Unlike the coefficient bounds $\bar{l}(B)$ and $\bar{h}(B),$ here $l(B)$ and $h(B)$ are independent of the number $\theta .$ However, if, say, $(Q,P)\in\mathcal{M}_0(A_1,\theta A_2 ,\theta B ,E,E_1,\theta E_2,T),$ then $|\nabla_p H(Q(t,q),P(t,q))| \leq l(\theta B ).$    
\begin{defn}
\label{defn:Mzerostar} 
For any $D=(D_1,D_2),$ $D_1,D_2>0,$ define:

(i) 
$$ 
\mathcal M_{0,D}^*(A_1,\theta A_2,\theta B ,E,E_1,\theta E_2,T)\subset W^{1,2,2;\infty}([0,T]\times[0,T]\times\dtorus;\dtorus\times\R^d) 
$$ 
as the subset of those $Z(\cdot;\cdot,\cdot)$ such that, for each $s\in[0,T],$
$$
Z(s;\cdot,\cdot)\in\mathcal M_0(A_1,\theta A_2,\theta B ,E,E_1,\theta E_2,T),
$$ 
and 
\begin{equation}
\label{eq:partialsZ} 
\| \partial_s Q (s;\cdot,\cdot) \|_{\infty} \leq D_1 , \quad \| \partial_s P(s;\cdot,\cdot) \|_{\infty} \leq D_2
\end{equation} 
wherever $\partial_sQ,$ $\partial_s P$ are defined;

(ii) 
$$ 
\mathcal Q_{0,D_1}^*(A_1,\theta A_2,\theta B ,E,E_1,\theta E_2,T)  \subset W^{1,2,2;\infty}([0,T]\times[0,T]\times\dtorus;\dtorus)
$$
as the subset of those $Q(\cdot;\cdot,\cdot)$ such that, for each $s\in[0,T],$ 
$$
(Q(s;\cdot,\cdot),0) \in \mathcal M_0^*(A_1,\theta A_2 ,\theta B ,E,E_1,\theta E_2,T),
$$ 
and 
\begin{equation*}
\| \partial_s Q (s;\cdot,\cdot) \|_{\infty} \leq D_1 
\end{equation*} 
wherever $\partial_s Q$ is defined.  
\end{defn} 
By an argument similar to that of Proposition \ref{prop:M0isclosed}, the sets $\mathcal M_0^*$ and $\mathcal Q_0^*$ just defined are closed subsets for the uniform convergence of $C([0,T]\times[0,T]\times\dtorus;\dtorus\times\R^d)$ and $C([0,T]\times[0,T]\times\dtorus;\dtorus)$ respectively.  
\begin{lemma}
\label{lemma:continuityofSigmains} 
For fixed $\mu\in\wasstwospace,$ using the same notation of Definition \ref{defn:Mzerostar}, and $\theta >2\kappa$: \begin{enumerate}[(i)] 
\item 
There exists a pair of positive constants $D=(D_1,D_2)$ such that, if \newline $Z\in\mathcal{M}_{0,D}^*(A_1,\theta A_2 ,\theta B ,E,E_1,\theta E_2,T)$, then the function $$ (s,t,q)\mapsto \mathfrak{m}^{s,\mu}(Z(s;\cdot,\cdot))(t,q) $$ belongs to $\mathcal{M}_{0,D}^*(A_1,\theta A_2 ,\theta B ,E,E_1,\theta E_2,T)$ for any $\mu\in\wasstwospace$.  
\item 
The mapping $$ s\mapsto \Sigma[s,\mu](t,q)$$ is differentiable in $s$ $a.e.$~on the interval $0<s<T,$ for every $\mu\in\wasstwospace,$ $t\in[0,T],$ and $q\in\dtorus,$ and it satisfies
$$ 
\| \partial_s \Sigma^1[s,\mu](\cdot,\cdot)\|_{C([0,T]\times\dtorus;\R^d)} \leq  D_1 , \quad \| \partial_s \Sigma^2[s,\mu](\cdot,\cdot)\|_{C([0,T]\times\dtorus;\R^d)} \leq  D_2 
$$ 
for a.e.~$s\in[0,T], \mu\in\wasstwospace .$ 
\end{enumerate} 
\end{lemma} 
\begin{proof} 
Given a function $Z\in\mathcal{M}_{0,D}^*(A_1,\theta A_2 ,\theta B ,E,E_1,\theta E_2,T),$ define $\mathfrak{m}^{\mu}(Z)\in C([0,T]\times[0,T]\times\dtorus;\dtorus\times\R^d)$ to be the first function displayed in the statement. 

(i) Let us show that $\mathfrak{m}^{\mu}(Z)\in\mathcal M_{0,D}^*(A_1,\theta A_2 ,\theta B ,E,E_1,\theta E_2,T)$ for an appropriate $D.$ Indeed, that $\mathfrak{m}^{\mu}(Z)$ is continuous is evident. For a.e.~$s\in(0,T),$
\begin{align*} 
\partial_s Q'(s;t,q) =  -\nabla_pH(Q(s;s,q),P(s;s,q))  +  \int_s^t  &\ [ \nabla^2_{qp}H(Q(s;\tau,q),P(s;\tau,q))\partial_s Q(s;\tau,q) 
\\ 
&\ + \nabla^2_{pp}H(Q(s;\tau,q),P(s;\tau,q)) \partial_s P(s;\tau,q) ] d\tau , 
\end{align*} 
so 
\begin{equation*}
\| \partial_s Q'(s;\cdot,\cdot)\|_{\infty} \leq l(\theta B ) + Th(\theta B )(D_1+D_2) . 
\end{equation*} 
Put $Q(s;\tau,\cdot)_{\#}\mu=: \sigma^s_{\tau},$ $0\leq \tau\leq T.$ Then, for a.e.~$s\in(0,T),$  \begin{align*} \partial_s P'(s;t,q) = &\ (\partial_s)(g(Q(s;0,q),\sigma^s_0))   \\ -  \int_0^t &\ [ \nabla^2_{qq}H(Q(s;\tau,q),P(s;\tau,q))\partial_s Q(s;\tau,q) \\ &\ \qquad + \nabla^2_{pq}H(Q(s;\tau,q),P(s;\tau,q)) \partial_s P(s;\tau,q)  +(\partial_s)(g(Q(s;\tau,q),\sigma^s_\tau))  ] d\tau .  \end{align*} By the joint Lipschitz constants of $g$ and $F$ being bounded by $\kappa,$ we obtain 
\begin{equation*}
\| \partial_s P'(s;\cdot,\cdot)\|_{\infty} \leq 2\kappa D_1 + T(2\kappa D_1 + h(\theta B )(D_1+D_2)).  \end{equation*} 
Thus, if $D_1>l(\theta B ),$ and $D_2>2\kappa D_1,$ we refer to Remark \ref{remark:remarkaboutconstantT} and assume that 
$$ 
T < \min\big\{\frac{D_1-l(\theta B )}{h(\theta B )(D_1+D_2)}, \frac{D_2-2\kappa D_1}{2\kappa D_1+h(\theta B )(D_1+D_2)}  \big\}
$$ 
to obtain $\| \partial_s Q'(s;\cdot,\cdot)\|_{\infty} \leq D_1$ and $\| \partial_s P'(s;\cdot,\cdot)\|_{\infty} \leq D_2.$ 

The fact that $\mathfrak{m}^{\mu}(Z)(s;\cdot,\cdot)\in\mathcal M_0(A_1,\theta A_2 ,\theta B ,E,E_1,\theta E_2,T) $ follows from Corollary \ref{cor:hamODEshassolution}. 

(ii) Let $Z^0\in\mathcal M_{0,D}(A_1,\theta A_2 ,\theta B ,E,E_1,\theta E_2,T)$ be arbitrary, with $D$ from part (i). Define inductively $Z^k=\mathfrak{m}^{\mu}(Z^{k-1}),$ $k=1,\ldots$ Then $\{Z^k\}_{k=0}^{\infty}$ is a sequence in $\mathcal M_{0,D}^*(\cdots),$ and for each $s\in[0,T],$ $Z^k(s;\cdot,\cdot) = \mathfrak{m}^{s,\mu}(Z^{k-1}(s;\cdot,\cdot)),$ so for each fixed $s\in[0,T],$ 
$$ 
Z^k(s;\cdot,\cdot) \longrightarrow \Sigma[s,\mu](\cdot,\cdot) \quad \textrm{uniformly in } \mathcal M_0(A_1,\theta A_2 ,\theta B ,E,E_1,\theta E_2,T) ,
$$ 
by the fixed point theorem. We thus have pointwise convergence of $Z^k(\cdot;\cdot,\cdot)$ to $Z[\cdot,\mu](\cdot,\cdot).$ We call on the equicontinuity, uniform boundedness of the sequence and the periodicity of the functions (i.e., they are defined on $[0,T]\times[0,T]\times\dtorus$) to conclude that this convergence is actually uniform. The closedness of the subspace $\mathcal M_{0,D}^*(\cdots)$ with respect to uniform convergence now ensures that $\Sigma[\cdot,\mu](\cdot,\cdot)$ belongs to this subspace, so it is differentiable with respect to $s$ for a.e.~$s\in[0,T]$ and satisfies (\ref{eq:partialsZ}).  
\end{proof} 
In the following, the constants $D_1,D_2$ will always be as in Lemma \ref{lemma:continuityofSigmains}. The next remark will not be used before Section \ref{section:regularityinmu}. 
\begin{remark}
\label{remark:Dbar} 
If $Z=(Q,P)$ and $\bar Z=(\bar Q, \bar P)$ are related as in the proof of Corollary \ref{cor:hamODEshassolution}(ii), that is, $Q = \bar Q,$ $P = \theta \bar P,$ then $Z\in\mathcal{M}^*_{0,D}(A_1,\theta A_2,\theta B,E,E_1,\theta E_2,T)$ if, and only if, $\bar Z \in \mathcal{M}^*_{0,\bar D}(A_1,A_2,B,E,E_1,E_2,T),$ where 
\begin{equation*}
\bar D = (\bar D_1, \bar D_2), \quad \bar D_1 := D_1, \ \bar D_2 := D_2 / \theta . \qquad \qquad \sslash \end{equation*}     
\end{remark}
\begin{lemma}
\label{lemma:continuityofSigmainmu} 
Let $\theta > 2\kappa.$ The mapping
\begin{align*}
[0,T]\times[0,T]\times\dtorus\times\wasstwospace &\ \longrightarrow [0,T]\times[0,T]\times\dtorus\times\wasstwospace
\\ 
(t,s,q,\mu) &\ \longmapsto (t,s,\Sigma^1[s,\mu](t,q),\mu)
\end{align*}
is continuous, and for any fixed $\mu\in\wasstwospace,$ $\Sigma^1[\cdot,\mu](\cdot,\cdot)$ is a $C^1$ diffeomorphism.  
\end{lemma}  
\begin{proof} 
Let $\{\mu_k\}_{k=1}^{\infty}$ be a sequence in $\wasstwospace$ converging to $\mu\in\wasstwospace.$ Consider the sequence $\{\Sigma^1[\cdot,\mu_k](\cdot,\cdot)\}_1^{\infty}$ and an arbitrary subsequence $\{\Sigma^1[\cdot,\mu_{k_j}](\cdot,\cdot)\}_{j=1}^{\infty}.$ Being in $\mathcal Q_{0,D}^*(A_1,\theta A_2 ,\theta B ,E,E_1,\theta E_2,T)$, the latter is equicontinuous and uniformly bounded, so there is a sub-subsequence, which we still index with $j$, converging to $S$ for some $S\in\mathcal Q_{0,D}^*(\cdots)$ as $j\to\infty.$ For each $s\in[0,T],$ on one hand, $\Sigma^1[s,\mu^{k_j}](\cdot,\cdot) \underset{j\to\infty}{\longrightarrow} S(s;\cdot,\cdot).$ On the other, since, by Corollary \ref{cor:hamODEshassolution} and Lemma \ref{lemma:continuityinsandmu}, the mapping $(s,Z,\mu)\mapsto \mathfrak{m}^{s,\mu}(Z)$ is a continuous mapping of $[0,T]\times\mathcal M_0(A_1,\theta A_2 ,\theta B ,E,E_1,\theta E_2,T)\times\wasstwospace$ into $\mathcal M_0(A_1,\theta A_2 ,\theta B ,E,E_1,\theta E_2,T)$ (because the Lipschitz constant of $\mathfrak{m}^{s,\mu}$ is independent of $s$ and $\mu$), we have $$ \Sigma^1[s,\mu^{k_j}](\cdot,\cdot) = (\mathfrak{m}^1)^{s,\mu^{k_j}}(\Sigma^1[s,\mu^{k_j}](\cdot,\cdot)) \underset{j\to\infty}{\longrightarrow} (\mathfrak{m}^1)^{s,\mu}(S(s;\cdot,\cdot)).$$ Therefore, for each $s\in[0,T],$ $S(s;\cdot,\cdot) = (\mathfrak m^1)^{s,\mu}(S(s;\cdot,\cdot)),$ that is, $S(s;\cdot,\cdot)$ is a fixed point of $(\mathfrak{m}^1)^{s,\mu},$ so, by uniqueness, $S(s;\cdot,\cdot)=\Sigma^1[s,\mu](\cdot,\cdot).$ Thus, every subsequence of $\{\Sigma^1[\cdot,\mu_k](\cdot,\cdot)\}_1^{\infty}$ has a subsequence that converges to $\Sigma^1[\cdot,\mu](\cdot,\cdot) \in \mathcal Q_{0,D}^*(A_1,\theta A_2 ,\theta B ,E,E_1,\theta E_2,T).$ Hence, $$  \Sigma^1[\cdot,\mu_k](\cdot,\cdot) \longrightarrow \Sigma^1[\cdot,\mu](\cdot,\cdot) \qquad \textrm{uniformly,} $$ which implies the claimed continuity. 

The second assertion of the lemma will be an immediate consequence of Lemma \ref{lemma:invertibilityofsigma}.       \end{proof}    
\begin{lemma}
\label{lemma:invertibilityofsigma} Let $\theta >2\kappa,$ $0\leq s\leq T,$ $\mu\in\wasstwospace.$  The mapping $q\mapsto \Sigma^1[s,\mu](t,q)$ is a $C^1$ diffeomorphism, for $0\leq t \leq T,$ with \begin{equation}\label{eq:boundsfornablaqsigma} \frac{1}{2} < \det \nabla_q\Sigma^1[s,\mu](t,q), \qquad |(\nabla_q\Sigma^1[s,\mu](t,q))^{-1}| < 4(1+\sqrt{d})^{d-1},   \end{equation} provided $T$ is sufficiently small.  \end{lemma} \begin{proof} We know the mapping is already $C^1$ because $W^{2;\infty}$ mappings are continuously differentiable. To prove invertibility, put $\Theta(t,q):= \Sigma^1(t,q)-q.$ Computing $\nabla_q\Theta(t,q),$ we have 
\begin{equation}\label{eq:neededjustbelow} |\nabla_q\Theta(t,q)|\leq |s-t|(A_1+\theta A_2 )h(\theta B )  \end{equation} 
because $\Sigma\in\mathcal{M}_0(A_1,\theta A_2 ,\theta B ,E,E_1,\theta E_2,T).$ By Remark \ref{remark:remarkaboutconstantT}, this means the function $q\mapsto \Theta(t,q)$ has Lipschitz constant strictly less than $1.$ Therefore, the function $q\mapsto q + \Theta(t,q) = \Sigma^1(t,q)$ is injective, for $0\leq t\leq T.$

  To prove that $q\mapsto\Sigma^1(t,q)$ is onto, note that $\sup_{q\in\R^d}|\Sigma^1(t,q)-q| \leq T l(\theta B ) < 2Tl(\theta B ),$ for $0\leq t\leq T.$ Let $y$ be a point in the ball of radius $R-2Tl(\theta B )$ in $\R^d$ centered at the origin, where $R>1>2Tl(\theta B ).$ Then for all $q$ on the \emph{boundary} of $B_R(0)$  ---the ball of radius $R$ in $\R^d$ centered at the origin--- we have: $\Sigma^1(t,q)\neq y,$ for $0\leq t\leq T.$ Therefore $$ f(t) := \textrm{deg}(\Sigma^1_t,B_R(0),y) , \quad 0\leq t\leq T ,  $$ the topological degree of $\Sigma_t^1$ is well defined at $y\in B_{R-2Tl(\theta B )}(0).$ This counts the number of ``signed'' solutions (see, e.g., \cite{degreetheory}) $x$ in $B_R(0)$ of the equation $\Sigma_t^1(x) = y.$ Since $f$ is a continuous function taking on integer values only, we conclude that $f(t)=f(s)=1.$ This means that the range of $\Sigma_t^1$ includes $B_{R-2Tl(\theta B )}(0).$ Since $R>1$ is arbitrary, we conclude that the range of $\Sigma_t^1$ is $\R^d.$ \begin{align*} \textrm{\emph{We will}} &\textrm{\emph{ denote the inverse of }}  \Sigma^1 \textrm{ by }  X, \textrm{\emph{ so }} \\ & X[s,\mu](t,q)= [\Sigma^1[s,\mu](t)]^{-1}(q)  ,\end{align*} for $0<s\leq T ,$ $0\leq t\leq T,$ $q\in\dtorus,$ $\mu\in\wasstwospace.$ Next, note that for $0\leq t\leq T,$ $q\in\dtorus,$ $|\nabla_q\Sigma^1_t(q)-I_d|\leq T(A_1+\theta A_2 )h(\theta B ) <1,$ since $T$ is small, where $I_d$ is the $d\times d$ matrix with $1$'s in the diagonal and $0$'s everywhere else. This implies that $\nabla_q\Sigma^1_t(q)$ is an invertible matrix, for $0\leq t\leq T,$ $q\in\dtorus$. By the inverse function theorem, $X_t$ is differentiable. Moreover, since \begin{equation}\label{eq:inverseformula} \nabla_q X_t(q)=[\nabla_q\Sigma^1_t(X_t(q))]^{-1}, \end{equation} and $q\mapsto \nabla_q \Sigma^1_t(q)$ is continuous, continuity of matrix inversion gives that the mapping $q\mapsto \nabla_q X_t(q)$ is continuous; this means that $X$ is $C^1$ in $q.$ 

To show (\ref{eq:boundsfornablaqsigma}), we may use the fact that\footnote{See, for instance, \cite[Remark 3.11]{mfgmain}.} the determinant function $\det:\R^{d^2}\to\R$ has derivative $\nabla \det$ satisfying $$ |\nabla \det (\xi)| \leq 2|\xi|^{d-1} , \quad \xi\in\R^{d^2}  $$ and the inverse matrix formula $$ \xi^{-1} = \frac{1}{\det\xi}(\nabla \det\xi)^t , $$ where the superscript $t$ denotes transposition.  By the mean-value theorem, there is $\tau\in[0,1]$ such that $\det(I_d+\frac{s}{T}\nabla_q\Theta)-\det I_d=\nabla\det(I_d+\tau\frac{s}{T}\nabla_q\Theta)\cdot \frac{s}{T}\nabla_q\Theta,$ where $\nabla\Theta$ abbreviates $\nabla\Theta(t,q)$ at an arbitrary $(t,q)\in[0,T]\times\dtorus.$ Hence, by the aforementioned fact, \begin{align*} |\det(I_d+\frac{s}{T}\nabla_q\Theta)-\det I_d| &\ \leq \frac{s}{T}|\nabla\det(I_d+\tau\frac{s}{T}\nabla_q\Theta)||\nabla_q\Theta|  \leq \frac{s}{T} 2 |I_d + \tau\frac{s}{T}\nabla_q\Theta|^{d-1}||\nabla_q\Theta| \\ &\ \leq 2(\sqrt{d}+|\nabla_q\Theta|)^{d-1}|\nabla_q\Theta| \leq 2(1+\sqrt{d})^{d-1}T(A_1+A_2)h(\theta B ) \\ &\ < \frac{1}{2},  \end{align*} by (\ref{eq:neededjustbelow}) and because $T$ is small enough that $$ T < \frac{1}{4(1+\sqrt{d})^{d-1}(A_1+\theta A_2 )h(\theta B )}.$$ Since $I_d+\nabla_q\Theta(t,q)=\nabla_q\Sigma^1(t,q)$ and $\det I_d=1,$ we obtain the first inequality in (\ref{eq:boundsfornablaqsigma}). Using the inverse matrix formula and the inequality $|\nabla\det(\xi)|\leq 2|\xi|^{d-1}$ once more, we have \begin{align*} |\nabla_q X(t,q)| = &\ |(I_d+\nabla_q\Theta)^{-1}| = \big|(\det(I_d+\nabla_q\Theta))^{-1}[\nabla \det(I_d+\nabla_q\Theta)]^t\big| \\ \leq &\ \frac{2}{\det(I_d+\nabla_q\Theta)}|I_d+\nabla_q\Theta|^{d-1} \leq  \frac{2}{\det(I_d+\nabla_q\Theta)}(1+\sqrt{d})^{d-1} \\ < &\ 4 (1+\sqrt{d})^{d-1},    \end{align*} and since this holds for any $t\in[0,T]$ and $q\in\dtorus,$ we have obtained the second inequality in (\ref{eq:boundsfornablaqsigma}). 

\textbf{Bound on $\nabla_q X$.} Due to the formula $\nabla_q X_t(q)=(\nabla_q\Sigma^1_t)^{-1}\circ X_t(q),$ i.e., formula (\ref{eq:inverseformula}), the second inequality in (\ref{eq:boundsfornablaqsigma}) implies \begin{equation}
\label{eq:boundsfornablaqsigma1} 
\|\nabla_q X[s,\mu](t,\cdot)\|_{C(\dtorus;\dtorus\times\dtorus)} < 4(1+\sqrt{d})^{d-1}.  
\end{equation}     
\end{proof} 
\begin{defn}
\label{defn:sigmaandvvfield}  
Let $\theta >2\kappa.$ Given $\mu \in  \mathscr P(\mathbb T^d)$, $s\in[0,T],$ and $\Sigma$ the unique fixed point of $\mathfrak{m}^{s,\mu}$ in \newline $\mathcal{M}_0(A_1,\theta A_2 ,\theta B ,E,E_1,\theta E_2,T),$ set 
\begin{equation}
\label{eq:settingsigma}  
\sigma_t = \sigma(t) := \Sigma^1[s,\mu](t,\cdot)_{\#}\mu, \quad v[s,\mu](t,q) := \partial_t\Sigma^1_t[s,\mu] \circ X_t[s,\mu](q),  
\end{equation} 
for $0\leq t\leq T.$ 
\end{defn} 
It should be kept in mind that the path $\sigma_t$ depends on $s$ and $\mu.$ Also, the arguments $s,\mu$ may often be omitted in the notation for $v,$ as has been done for $\Sigma.$ 
\begin{prop}\label{prop:hmvisavelocity} 
The path $\sigma$ belongs to $AC^{2}(0,T;\wasstwospace)$ and $v$ is a velocity associated to $\sigma,$ that is, $\partial_t \sigma + \textrm{div}(\sigma v) =0$ in the distribution sense, with $v_t=v(t,\cdot)\in L^{2}(\dtorus,\sigma_t),$ $0< t < T.$ 
\end{prop} 
\begin{proof} 
The proof is simple and follows by direct estimation and calculation, so we omit it.
\end{proof}
Note that the definition of the field $v_t$ means that the mappings $t\mapsto \Sigma^1(t,q)$ are the flow lines of $v_t.$   
\subsection{Value function $U$ and characteristics}\label{subsection:thecharacteristicsmethod} In what follows, the hypotheses of Corollary \ref{cor:hamODEshassolution} will be in force, with $\Sigma$ denoting the solution to (\ref{eq:hamODEs}). The following statement stems from the fact that $\Sigma^1[s,\mu](\cdot,\cdot)\in W^{2,2;\infty}((0,T)\times\dtorus;\dtorus);$ we omit its proof. \begin{prop}\label{prop:XisC1int} For every $s\in[0,T],$ $\mu\in\wasstwospace,$ the function $X[s,\mu](\cdot,\cdot)$ is  in $W^{2,2;\infty}((0,T)\times\dtorus;\dtorus).$  \end{prop} Define now \begin{align} \mathcal{V}[s,\mu](t,q) :=  \Sigma^2_t[s,\mu]\circ X_t[s,\mu](q)  =  \Sigma^2_t[s,\mu]\circ (\Sigma_t^1[s,\mu])^{-1}(q)  , \label{eq:mathcalV} \end{align}for $s,t\in[0,T], \mu\in\wasstwospace, q\in\dtorus.$ Alternatively, we may write $\mathcal{V}_t[s,\mu](q).$ We can now proceed to solve the MFG system. \begin{lemma}\label{lemma:zlemma} Let T be small according to Remark \ref{remark:remarkaboutconstantT}, $s\in[0,T],$ $\mu\in\wasstwospace,$ and, as in Definition \ref{defn:sigmaandvvfield}: $\sigma_t = \Sigma^1[s,\mu](t,\cdot)_{\#}\mu.$ For each $q\in\dtorus,$ let $$ U(t,q) = z(t,X[s,\mu](t,q)), \quad t\in[0,T] , $$ where $z(\cdot,q)$ satisfies \begin{align} \partial_t z(t,q) = &\ \Sigma^2[s,\mu](t,q)\cdot \nabla_p H(\Sigma^1[s,\mu](t,q),\Sigma^2[s,\mu](t,q)) \label{eq:zode}  \\  &\ -  H(\Sigma^1[s,\mu](t,q),\Sigma^2[s,\mu](t,q)) - F(\Sigma^1[s,\mu](t,q),\sigma_t) \quad \textrm{ in } (0,T), \notag \\ z(0,q) = &\ g(\Sigma^1[s,\mu](0,q),\sigma_0).      \end{align} Then $U\in C^1((0,T)\times\dtorus)$ and solves the Hamilton-Jacobi equation of the mean-field game system: \begin{align}\partial_t U(t,x) + H(x,\nabla_q U(t,x)) + F(x,\sigma_t) = &\  0 \quad \textrm{ in } (0,T)\times \dtorus, \label{eq:HJEx} \tag{\ref{eq:mfg1}} \\ U(0,\cdot) = &\ g(\cdot,\sigma_0). \tag{\ref{eq:mfg3}} \label{eq:HJExinitial}   \end{align}  \end{lemma} \begin{proof} Since $s$ and $\mu$ are fixed, we abbreviate $\Sigma[s,\mu](t,q) = \Sigma_t(q).$ Observe that the right-hand side of (\ref{eq:zode}) is $C^1$ in $q,$ $C^0$ in $t,$ so $z$ is $C^1$ in $q,$ $C^1$ in $t.$ Therefore, $U$ is $C^1$ in both variables $t$ and $q,$ because of Proposition \ref{prop:XisC1int}. Moreover, since $\partial_t z(t,q)$ is $C^1$ in $q,$ then (e.g.~see \cite[Thm.~9.41]{babyrudin}) $\nabla^2_{tq} z$ exists and is equal to $\nabla^2_{qt} z.$ Thus, the calculations below are legitimate. We have $U(t,\Sigma_t^1(q)) = z(t,q),$ so $\partial_t z(t,q) = \partial_t (U(t,\Sigma_t^1(q))$ and  \begin{align} \partial_t z(t,q) = &\  \partial_t (U(t,\Sigma_t^1(q))) = \partial_t U(t,\Sigma_t^1(q)) + \nabla_q U(t,\Sigma_t^1(q)) \cdot \partial_t \Sigma^1_t(q) \notag \\ = &\ \partial_t U(t,\Sigma_t^1(q)) + \nabla_q  U(t,\Sigma_t^1(q)) \cdot \nabla_pH(\Sigma_t^1(q),\Sigma_t^2(q)) \label{eq:totaltimeU} \end{align} Now, \textit{if} \begin{equation}\label{eq:importantif} \nabla_q U(t,\Sigma_t^1(q)) = \Sigma_t^2(q), \quad t \in (0,T), q\in \dtorus, \end{equation} then, comparing (\ref{eq:totaltimeU}) and (\ref{eq:zode}), we get $$ \partial_t U(t,\Sigma_t^1(q)) = -H(\Sigma_t^1(q),\Sigma_t^2(q)) - F(\Sigma_t^1(q),\sigma_t), $$ and the change of variable $x=\Sigma_t^1(q)$ then yields (\ref{eq:HJEx}), (\ref{eq:HJExinitial}). We set out to prove (\ref{eq:importantif}) now. Let $$ r_i(t) := \frac{\partial z(t,q)}{\partial q^{(i)} } - \sum_{j=1}^d (\Sigma_t^2)^{(j)}(q)\frac{\partial x^{(j)}}{\partial q^{(i)}} , $$ $i=1,\ldots,d,$ where $x=\Sigma^1_t(q).$  We know that $r_i(0) = 0,$ from the initial conditions, and \begin{align*} \dot{r}_i(q) = \frac{\partial^2 z(t,q)}{\partial t\partial q^{(i)}} - \sum_{j=1}^d \partial_t (\Sigma_t^2)^{(j)}(q) \frac{\partial}{\partial q^{(i)}}\frac{\partial}{\partial t}x^{(j)}  =: a - b .\end{align*} Using the first line, $\partial_t \Sigma^1_t(q) = \nabla_pH(\Sigma^1_t(q),\Sigma^2_t(q)),$ in (\ref{eq:hamODEs}), we have \begin{align*} a =&\ \partial_{q^{(i)}}\partial_t z(t,q) \\   = &\ \sum_{k=1}^d \big[ \sum_{k=1}^d \partial_{q^{(i)}}(\Sigma^2_t)^{(k)}(q)\partial_t (\Sigma^1)^{(k)}_t(q) +  (\Sigma_t^2)^{(k)}(q)\partial^2_{t,q^{(i)}}(\Sigma^1_t)^{(k)}(q) \big] \\ &\ -   \sum_{l=1}^d \big[ \partial_{q^{(l)}}H(\Sigma^1_t(q),\Sigma^2_t(q))\partial_{q^{(i)}} (\Sigma_t^1)^{(l)}(q) + \partial_{p^{(l)}}H(\Sigma^1_t(q),\Sigma^2_t(q))\partial_{q^{(i)}} (\Sigma_t^2)^{(l)}(q)   \big ] \\  &\ - \sum_{l=1}^d \partial_{q^{(l)}}F(\Sigma^1_t(q),\sigma_t)\partial_{q^{(i)}}(\Sigma_t^1)^{(l)}(q) ,     \end{align*} and, by the second line in (\ref{eq:hamODEs}), namely, $\partial_t\Sigma^2_t(q) = -\nabla_qH(\Sigma_t^1(q),\Sigma_t^2(q)) - \nabla_qF(\Sigma_t^1(q),\sigma_t),$ this simplifies to \begin{align*} a = \sum_{k=1}^d \big[ (\Sigma_t^2)^{(k)}(q)\partial^2_{tq^{(i)}}(\Sigma_t^1)^{(k)}(q) + \partial_t(\Sigma^2_t)^{(k)}(q)\partial_{q^{(i)}}(\Sigma_t^1)^{(k)}(q)  \big].     \end{align*} As for $b,$ \begin{align*} b = &\ \partial_t \big( \sum_{j=1}^d(\Sigma_t^2)^{(j)}(q)\partial_{q^{(i)}}(\Sigma_t^1)^{(j)}(q)\big) \\ = &\ \sum_{j=1}^d \big[ \partial_t(\Sigma_t^2)^{(j)}(q)\partial_{q^{(i)}}(\Sigma_t^1)^{(j)}(q) + (\Sigma_t^2)^{(j)}(q)\partial^2_{tq^{(i)}}(\Sigma_t^1)^{(j)}(q)  \big] ,   \end{align*} so $a=b.$ Therefore $\dot{r}_i(t)\equiv 0,$ and $r_i(t) \equiv 0$ on $(0,T],$ by the uniqueness of (\ref{eq:zode}) $[0,T].$  Now we differentiate $U$, keeping in mind that $U(t,x) = z(t,q)$; using the fact that $r_i(t) = 0,$ $0\leq t\leq T,$ we have \begin{align*} \partial_{x^{(i)}}U(t,x) = \sum_{j=1}^d\frac{\partial z(t,q)}{\partial q^{(j)}}\frac{\partial q^{(j)}}{\partial x^{(i)}}  = \sum_{j=1}^d  \sum_{k=1}^d (\Sigma_t^2)^{(k)}(q)\frac{\partial x^{(k)}}{\partial q^{(i)}} \frac{\partial q^{(j)}}{\partial x^{(i)}} = \sum_{k=1}^d (\Sigma_t^2)^{(k)}(q)\frac{\partial x^{(k)}}{\partial x^{(i)}} = (\Sigma^2_t)^{(i)}(q) ,      \end{align*} for $i=1,\ldots,d$ and $0\leq t \leq T.$ This proves (\ref{eq:importantif}), completing the proof of the lemma. \end{proof} \begin{cor}\label{cor:cortozlemma} By (\ref{eq:importantif}) in the proof of the preceding lemma, we have \begin{equation}\label{eq:mathcalvisgradient} \nabla_q U(t,q) = \mathcal{V}[s,\mu](t,q), \qquad t\in(0,T), q\in\dtorus.  \end{equation}   \end{cor} The dependence of $U$ on the parameters $s$ and $\mu,$ made clear by its definition, should not be forgotten. This corollary means that $ \mathcal{V}[s,\mu](t,\cdot)$ is the $C^1$ gradient of a function. Thus, \begin{equation*} \nabla_q\mathcal{V}[s,\mu](t,q) \textrm{ \emph{is symmetric, for every }} t\in (0,T), q\in\dtorus. \end{equation*} To conclude our statement about the MFG system, observe that \begin{equation}\label{eq:observethat} v_t(q) =  \partial_t\Sigma_t^1((\Sigma_t^1)^{-1}(q)) = \nabla_pH(q,\mathcal V(t,q)) =  \nabla_pH(q,\nabla_q U(t,q)). \end{equation} Hence, combining with Proposition \ref{prop:hmvisavelocity}, we have the following. \begin{thm}\label{thm:soltomfg} (Existence of solution to the MFG system) Let $\mu\in\wasstwospace,$ $T$ be in accordance with Remark \ref{remark:remarkaboutconstantT} and Proposition \ref{prop:m0islipschitz}, $0 < s < T,$ and let $\sigma_t,$ $v_t$ be as in (\ref{eq:settingsigma}), where $(\Sigma^1[s,\mu],\Sigma^2[s,\mu])$ is the unique solution to (\ref{eq:hamODEs}) with parameters $s,\mu$. Then the pair $(U,\sigma),$ where $U$ is as in Lemma \ref{lemma:zlemma}, is a classical solution to the mean-field game system (\ref{eq:mfg1}-\ref{eq:mfg4}) in the sense explained in Section \ref{subsection:defnofclassicalmfg}.    \end{thm} Note that, by Proposition \ref{prop:XisC1int}, the function $U$ in the pair $(U,\sigma)$ constructed above is in $W^{2,2;\infty}((0,T)\times\dtorus)\times AC^2(0,T;\wasstwospace).$ The following is, in a sense, a consistency, or restricted uniqueness, complement to the latter theorem.  \begin{thm}\label{thm:uniquenessmfg} (The case of $W^{2,3;\infty}((0,T)\times\dtorus)\times AC^2(0,T;\wasstwospace)$ solutions to the MFG system) Let $(\tilde U,\tilde \sigma)$ be a classical solution to the MFG system (\ref{eq:mfg1}-\ref{eq:mfg4}), in the sense explained in Section \ref{subsection:defnofclassicalmfg}, such that $U\in W^{2,3;\infty}((0,T)\times\dtorus).$ Then $(\tilde{U},\tilde{\sigma})=(U,\sigma),$ where $(U,\sigma)$ is the pair constructed for Theorem \ref{thm:soltomfg}.    \end{thm} \begin{proof} Let $(\tilde{U},\tilde{\sigma})$ be a solution to the system (\ref{eq:mfg1}-\ref{eq:mfg4}) with parameters $s\in(0,T)$ and $\mu\in\wasstwospace$, according to the definition of Section \ref{subsection:defnofclassicalmfg}, and suppose, moreover, that $\tilde{U}$ is $W^{3;\infty}$ in $q.$  

1. We will prove that the characteristics of the MFG system satisfied by $(\tilde U,\tilde \sigma)$ must solve the Hamiltonian system (\ref{eq:hamODEs}). Set $$ \tilde v_t(q) = \tilde v(t,q) := \nabla_p H(q,\nabla_q \tilde{U}(t,q)),$$  $0\leq t\leq T,$ $q\in\dtorus.$  Since $\tilde{U}\in W^{2,3;\infty}((0,T)\times\dtorus)$ and $H\in C^3,$ we have that $\textrm{Lip}(\tilde v_t,K) + \sup_{q\in K}|\tilde v_t(q)|$ is bounded on $[0,T],$ where $K$ is any compact subset of $\R^d$ and $\textrm{Lip}(\tilde v_t,K)$ is the Lipschitz constant of $\tilde v_t|_{K}.$ By elementary ODE theory (see, e.g., \cite[Lemma 8.1.4]{gradientflows}), if $q\in\R^d,$ the ODE \begin{equation}\label{eq:ODEQ} \tilde \Sigma^1(s,q) = q, \quad \frac{\partial}{\partial t}\tilde \Sigma^1(t,q) = \tilde v_t(\tilde\Sigma^1(t,q)) \end{equation} has a unique maximal solution in a neighborhood $I(q,s)\subset (0,T)$ of $s,$ but, since $\tilde \Sigma^1(t,\cdot)$ is periodic, and therefore bounded for every $t\in(0,T),$ then $I(q,s) = (0,T).$ Clearly, the path $t\mapsto \tilde\Sigma^1(t,\cdot)_{\#}\mu$ solves the continuity equation with velocity $\tilde v_t.$ We can apply Proposition 8.1.7 of \cite{gradientflows} to conclude that \begin{equation}\label{eq:tildesigmatildesigma1} \tilde{\sigma}_t = \tilde \Sigma^1(t,\cdot)_{\#}\mu,\end{equation} $0\leq t\leq T.$
Let \begin{equation}\label{eq:ODEQ1} \tilde{\mathcal{V}}(t,q) := \nabla_q \tilde{U}(t,q) \ , \qquad \tilde\Sigma^2(t,q) := \tilde{\mathcal{V}}(t,\tilde \Sigma^1(t,q)),  \end{equation} $0\leq t\leq T,$ $q\in\dtorus.$ Then \begin{equation}\label{eq:ODEQ2} \partial_t \tilde\Sigma^1(t,q)  = \tilde v_t(\tilde\Sigma^1(t,q)) = \nabla_p H(\tilde\Sigma^1(t,q),\tilde\Sigma^2(t,q)) , \end{equation} which is the first equation in (\ref{eq:hamODEs}). To obtain the second one, observe that $$ \partial_t \tilde\Sigma^2(t,q) = \partial_t\tilde{\mathcal{V}}(t,\tilde\Sigma^1(q,t)) + \nabla_q\tilde{\mathcal{V}}(t,\tilde\Sigma^1(t,q))\partial_t \tilde\Sigma^1(t,q) = \nabla_{tq}^2 \tilde{U}(t,\tilde\Sigma^1(t,q)) + \nabla^2_{qq}\tilde{U}(t,\tilde\Sigma^1(t,q)) \partial_t \tilde\Sigma^1(t,q), $$ while, differentiating the Hamilton-Jacobi equation with respect to $q$ gives $$ \nabla_{qt}^2 \tilde{U}(t,q) + \nabla_q H(q,\nabla_q \tilde{U}(t,q)) + \nabla_p H(q,\nabla_q \tilde{U}(t,q))\nabla^2_{qq} \tilde{U}(t,q) + \nabla_q F(q,\tilde{\sigma}_t) = 0 ,$$ which, evaluating at $\tilde\Sigma^1(t,q)$ in place of $q,$ and using (\ref{eq:ODEQ1}), (\ref{eq:ODEQ2}), serves to simplify the former equality to $$ \partial_t \tilde\Sigma^2(t,q) = -\nabla_q H(\tilde\Sigma^1(t,q),\tilde\Sigma^2(t,q)) - \nabla_q F(\tilde\Sigma^1(t,q),\tilde{\sigma}_t),   $$ which is the second equation in (\ref{eq:hamODEs}). The condition $\tilde\Sigma^2(0,q) = \nabla_q g(\tilde\Sigma^1(0,q),\tilde\Sigma^1(0,\cdot)_{\#}\mu)$ follows readily from (\ref{eq:ODEQ1}) and (\ref{eq:HJExinitial}).

2. We prove that, for a possibly smaller $T,$ the solutions $(\tilde\Sigma^1,\tilde\Sigma^2)$ to (\ref{eq:hamODEs}) from the previous paragraph belongs to $\mathcal{M}_0(A_1,\theta A_2,\theta B,E,E_1,\theta E_2,T),$ i.e., they satisfy the bounds (\ref{eq:defofM0theta}) and (iii) of Definition \ref{defn:Msubzero}.  If $T$ is small enough, since $|\tilde\Sigma^2(0,q)|=|\nabla_q g(\tilde\Sigma^1_0(q),\tilde{\sigma}_0)|\leq\kappa,$ continuity implies $|\tilde\Sigma^2(t,q)|\leq \kappa+\eps,$ $0\leq t\leq T,$ $q\in\dtorus$ for an $\eps>0$ such that $$ |\tilde\Sigma^2(t,q)| \leq  \theta \frac{\kappa}{\theta} + \eps \leq  \theta \frac{1}{\theta}\max\{d,\kappa\} + \eps = \theta c/\theta + \eps \leq \theta B,  $$ because $B$ in the proof of Lemma \ref{lemma:existenceofABE} was chosen as $B > c/\theta$ (in these lines we are referring back to the proof of Lemma \ref{lemma:existenceofABE}, in particular (\ref{eq:A1}), (\ref{eq:A2}) and the paragraph preceding those inequalities). This is the third line in (\ref{eq:defofM0theta}). Since $\|\tilde\Sigma^2\|_{\mathcal{M}^2}\leq \theta B,$ we have $|\partial_t \tilde\Sigma^1(t,q)|=|\nabla_p H(\tilde\Sigma^1(t,q),\tilde\Sigma^2(t,q))| \leq \bar{l}(B)$ (see Definition \ref{defn:defofM}, ``Coefficient bounds I'') for small enough $T$ and all $q,$ and since $A_1$ in Lemma \ref{lemma:existenceofABE} was chosen to be larger than $\bar l(B),$ we obtain the bound for $\|\partial_t \tilde\Sigma^1\|$ in (\ref{eq:defofM0theta}). The one for $\|\partial_t \tilde\Sigma^2\|_{\mathcal{M}^2}$ goes in a similar way, because $A_2$ in Lemma \ref{lemma:existenceofABE} was chosen larger than $\bar l(B)/\theta.$ 
The bounds for the second-order time derivatives are dealt with in a similar way, keeping in mind the way $E_1$ and $E_2$ were chosen in the proof of Lemma \ref{lemma:existenceofABE}.

From (\ref{eq:ODEQ}), $\tilde\Sigma^1(t,q) = q + \int_s^t v_{\tau}(\tilde\Sigma^1(\tau,q)) d\tau,$ which makes it clear that, upon taking the gradient in $q,$ if $T$ is small enough, the norm of $\nabla_q \tilde\Sigma^1(t,q)$ will be only slightly larger than $\sqrt{d},$ making it less than $A_1,$ because of (\ref{eq:A1}). The bound for $\|\nabla^2_{qq}\tilde\Sigma^1\|$, due to $\nabla^2_{qq}(q) \equiv 0,$ can actually be made arbitrarily small by choosing $T$ small enough. To address $\nabla_q \tilde\Sigma^2$ and $\nabla_{qq}^2 \tilde\Sigma^2,$ since $$ \tilde{U}(t,q) = g(q,\tilde{\sigma}_0) + \int_0^t[H(q,\nabla_q \tilde{U}(\tau,q))+F(q,\tilde{\sigma}_{\tau})]d\tau,$$ and $\nabla_q \tilde\Sigma^2(t,q) = \nabla_{qq}^2\tilde{U}(t,\tilde\Sigma^1(t,q))\nabla_q\tilde\Sigma^1(t,q),$ the norm of $\nabla_q \tilde\Sigma^2$ is the product of a number slightly larger than $\kappa$ and one slightly larger than $\sqrt{d},$ for small times $T.$ But the constant $E$ in the proof of Lemma \ref{lemma:existenceofABE} is larger than $c=\max\{d,\kappa\},$ and $A_2 > \frac{1}{\theta}cE(E+1).$ This ensures that $\|\nabla_{q}\tilde\Sigma^2 \| \leq \theta A_2.$ Next, given that $$ \nabla_{qq}^2\tilde\Sigma^2 = \nabla^3_{qqq}\tilde{U}\nabla_q \tilde\Sigma^1\nabla_q \tilde\Sigma^1 +\nabla_{qq}^2\tilde{U}\nabla_{qq}^2 \tilde\Sigma^1 $$ (because $\tilde{U}$ is $W^{3;\infty}$ in $q$), and $|\nabla_{qq}^2 \tilde\Sigma^1|$, as already mentioned, can be made as small as needed by reducing $T$, the same argument shows that $\|\nabla_{qq}^2 \tilde\Sigma^2 \|$ is also no greater than $\theta A_2,$ since the norm of $ \nabla^3_{qqq}\tilde{U}\nabla_q \tilde\Sigma^1\nabla_q \tilde\Sigma^1$ is the product of a number slightly larger than $\kappa$ and one slightly than $d.$  Finally, $\|\nabla_q\tilde\Sigma^1_0\|, \|\nabla^2_{qq}\tilde\Sigma^1_0\| \leq E$ follow from $E$ having been picked larger than $d$ (and, therefore, than $\sqrt{d}$), and taking $T$ smaller if necessary.

Thus, the mapping $(\tilde\Sigma^1,\tilde\Sigma^2)$ constructed in (\ref{eq:ODEQ}) and (\ref{eq:ODEQ1}) from $(\tilde{U},\tilde{\sigma})$ coincides with the unique solution $(\Sigma^1[s,\mu],\Sigma^2[s,\mu])$ of (\ref{eq:hamODEs}) in $\mathcal{M}_0(A_1,\theta A_2,\theta B,E,E_1,\theta E_2,T)$ during a possibly shorter interval $[0,T].$ Consequently, by (\ref{eq:tildesigmatildesigma1}), we further have that $ \tilde{\sigma} =  \sigma.$ Also, now that $\tilde{\mathcal{V}}(t,q) = \mathcal{V}(t,q),$ $0\leq t \leq T, $  we get \begin{align*} \tilde U(t,q) = g(x,\tilde\sigma_0) -\int_0^t [H(q,\tilde{\mathcal{V}}(\tau,q)-F(x,\tilde\sigma_{\tau})]d\tau  = g(x,\sigma_0)-\int_0^t [H(x,\mathcal{V}(\tau,q)) - F(q,\sigma_{\tau})]d\tau = U(t,q) ,  \end{align*} for any $t\in[0,T].$ Thus, $(\tilde U,\tilde \sigma)=(U,\sigma)$ on the possibly smaller interval $[0,T].$ 

        \end{proof}  \subsection{The full value function $u(s,q,\mu)$} In this section we begin our study of the dependence of our solution $U$ to (\ref{eq:HJEx}) on the parameter $\mu.$ First, we present a list of facts that will be used in this and following sections.  \begin{prop}\label{prop:compositions1}  Let $0\leq s, t_0\leq T,$  $\mu\in\wasstwospace.$ Set 
$$ 
\sigma_{t_0} = \Sigma_{t_0}^1[s,\mu]_{\#}\mu. 
$$ 
Then, 
\begin{enumerate}[(i)] 
\item For every $0\leq t\leq T$: 
\begin{equation}\label{eq:ifwereunique} 
\Sigma_t[t_0,\sigma_{t_0}]\circ \Sigma_{t_0}^1[s,\mu] = \Sigma_t[s,\mu]. 
\end{equation} 
\item  For every $0\leq t\leq T$: 
\begin{align}
\Sigma^1_t[t_0,\Sigma^1_{t_0}[t,\mu]_{\#}\mu] &\ \quad \textrm{and} \quad \Sigma^1_{t_0}[t,\mu] \ \quad \textrm{are inverses of each other , } \label{eq:3.16ii} \\ v_t[s,\mu] &\ = v_t[t_0,\sigma_{t_0}] \ ,\label{eq:3.16iii} \\  \Sigma^2_t[t_0,\sigma_{t_0}]\circ \Sigma_{t_0}^1[s,\mu] &\ = \Sigma^2_t[s,\mu]   \ , \label{eq:3.16iiiNEW} \\  \partial_s \Sigma^1_t[s,\mu]     &\ = -\nabla_q \Sigma^1_t[s,\mu] v_s[t,\sigma_t].   \label{eq:3.16iv}       \end{align} 
\item If $0\leq \tau, t\leq T,$ then 
\begin{equation}
\label{eq:mathcalVttau1} 
\Sigma_{\tau}^2[t,\sigma_t]\circ (\Sigma_{\tau}^1[t,\sigma_t])^{-1} = \Sigma_{\tau}^2[s,\mu]\circ \Sigma^1_{\tau}[s,\mu]^{-1} . 
\end{equation}  
\end{enumerate}
\end{prop} 
\begin{proof} (i)  Let \begin{align*} Q(t,q) &\ = (\Sigma_t^1[s,\mu]\circ \Sigma_{t_0}^1[s,\mu]^{-1})(q) , \\  P(t,q) &\ = (\Sigma_t^2[s,\mu]\circ \Sigma_{t_0}^1[s,\mu]^{-1})(q) , \end{align*} for $0\leq t \leq T,$ $q\in\dtorus.$ By differentiating, and noting that $Q(0,\cdot)_{\#}\sigma_{t_0} = \Sigma_0^1[s,\mu]_{\#}\mu,$ one verifies that $Q$ and $P$ defined this way satisfy the Hamiltonian ODEs (\ref{eq:hamODEs})  with $$ s=t_0, \quad \mu = \sigma_{t_0}.$$  Since solutions to (\ref{eq:hamODEs}) are unique, we conclude that $$ (Q_t,P_t) = (\Sigma^1_t[t_0,\sigma_{t_0}],\Sigma^2_t[t_0,\sigma_{t_0}]) ,  $$ yielding (\ref{eq:ifwereunique}).

    (ii) Fact (\ref{eq:3.16ii}) follows readily from (i), by setting $t=s.$ For (\ref{eq:3.16iii}), see \cite[p. 6593]{mfgmain}. Formula (\ref{eq:3.16iiiNEW}) is just the second component of (\ref{eq:ifwereunique}). 

    By (\ref{eq:3.16ii}), with $t=s,$ $t_0=t$, we have $$ id = \Sigma^1_t[s,\mu]\circ \Sigma_s^1[t,\sigma_t] = \Sigma^1[s,\mu](t,\Sigma^1_s[t,\sigma_t]). $$ By Lemma \ref{lemma:continuityofSigmains}, we can differentiate both sides with respect to $s$: \begin{align*} 0 = \partial_s\Sigma^1[s,\mu](t,\Sigma_s^1[t,\sigma_t](q))+\nabla_q\Sigma^1[s,\mu](t,\Sigma_s^1[t,\sigma_t](q))\partial_s \Sigma_s^1[t,\sigma_t](q), \quad q\in \dtorus. \end{align*} Substituting $\Sigma_s^1[t,\sigma_t](q)$ for $q,$ we get \begin{align*} 0 = &\ \partial_s\Sigma^1[s,\mu](t,q) + \nabla_q\Sigma^1[s,\mu](t,q) \partial_s \Sigma^1_s[t,\sigma_t](\Sigma_s^1[t,\sigma_t]^{-1}) \\ = &\ \partial_s \Sigma^1[s,\mu](t,q) + \nabla_q\Sigma^1[s,\mu](t,q) v_s[t,\sigma_t],     \end{align*} which gives (\ref{eq:3.16iv}). 
    
    (iii) For (\ref{eq:mathcalVttau1}), simply use (\ref{eq:3.16iiiNEW}) with $\tau$ in place of $t$ and $t$ in place of $t_0$: \begin{align*} \Sigma_{\tau}^2[t,\sigma_t] \circ (\Sigma^1_{\tau}[t,\sigma_t])^{-1} &\ =  \Sigma_{\tau}^2[s,\mu]\circ \Sigma_t[s,\mu]^{-1}\circ \Sigma^1_{\tau}[t,\sigma_t]^{-1} \\ &\ = \Sigma_{\tau}^2\circ(\Sigma_{\tau}^1[t,\sigma_t]\circ \Sigma_t[s,\mu])^{-1} = \Sigma^2_{\tau}[s,\mu]\circ\Sigma^1_{\tau}[s,\mu]^{-1}.    \end{align*} 
    \end{proof} 
Given $s\in [0,T],$ $q\in\dtorus,$ $\mu\in\wasstwospace,$ define 
\begin{equation}
\label{eq:defofu} 
u(s,q,\mu) = g(q,\Sigma^1[s,\mu](0,\cdot)_{\#}\mu) - \int_0^s \big[ H(q,\mathcal{V}[s,\mu](\tau,q))+F(q,\Sigma^1[s,\mu](\tau,\cdot)_{\#}\mu \big] d\tau, 
\end{equation} 
and, as before, $\sigma_t = \Sigma^1_t[s,\mu]_{\#}\mu,$ $0\leq t\leq T.$ Note that (\ref{eq:mathcalVttau1}) reads now as 
\begin{equation*}
\mathcal V_{\tau}[t,\sigma_t] = \mathcal V_{\tau}[s,\mu]. 
\end{equation*} 
This, coupled with the fact that, by (\ref{eq:ifwereunique}) of Proposition \ref{prop:compositions1}, $\Sigma_{\tau}^1[t,\sigma_t]\circ \Sigma_t^1[s,\mu] = \Sigma_{\tau}^1[s,\mu], $ $0\leq \tau\leq T,$ gives us \begin{align} u(t,q,\sigma_t) &\ = g(q,\Sigma_0^1[t,\sigma_t]_{\#}\sigma_t) - \int_0^t \big[  H(q,\mathcal{V}[t,\sigma_t](\tau,q)) + F(q,\Sigma^1[t,\sigma_t](\tau,\cdot)_{\#}\mu) \notag \big] d\tau  \\ &\ = g(q,\sigma_0) - \int_0^t \big[ H(q,\mathcal V[s,\mu](\tau,q)) + F(q,\Sigma_{\tau}^1[s,\mu]_{\#}\mu)   \big] d\tau  \label{eq:deffofu}.  \end{align} Since $$ \mathcal{V}[s,\mu](t,\Sigma^1[s,\mu](t,q)) = \Sigma^2[s,\mu](t,q), $$ and $\Sigma^2$ satisfies the second of the Hamiltonian ODEs (\ref{eq:hamODEs}), it follows, by taking the total time derivative of $ \mathcal{V}[s,\mu](t,\Sigma^1[s,\mu](t,q)),$ and then changing variable from $\Sigma^1(t,q)$ to $q,$ that $\mathcal{V}(t,q)=\mathcal{V}[s,\mu](t,q)$ satisfies the equation \begin{align} \partial_t \mathcal{V}(t,q) + \nabla_q\mathcal{V}(t,q) \nabla_pH(q,\mathcal{V}(t,q)) = &\ -  \nabla_qH(q,\mathcal{V}(t,q)) - \nabla_qF(q,{{}\Sigma^1_t}_{\#}\mu), \label{eq:mathcalVsatisfies1} \\ \mathcal{V}(0,q) = \nabla_qg(q,\sigma_0). \label{eq:mathcalVsatisfies2}  
\end{align} 
If we differentiate $u(t,q,\sigma_t)$ with respect to $q$ in (\ref{eq:deffofu}), and use (\ref{eq:mathcalVsatisfies1}) and (\ref{eq:mathcalVsatisfies2}), we get 
\begin{align*} 
\nabla_qu(t,q,\sigma_t) =  \nabla_qg(q,\sigma_0)  +  \int_0^t &\big[ \partial_t\mathcal{V}[s,\mu](\tau,q) + \nabla_q\mathcal{V}[s,\mu](\tau,q)\nabla_pH(q,\mathcal{V}[s,\mu](\tau,q)) \\ & -  \nabla_pH(q,\mathcal{V}[s,\mu](\tau,q)\nabla_q\mathcal{V}[s,\mu](\tau,q)      \big] \ d\tau .  
\end{align*} 
Since $\nabla_q\mathcal{V}_{\tau}$ is a symmetric matrix for $\tau\in[0,T],$ only the term $\partial_t\mathcal{V}$ survives in the integral. Hence 
\begin{equation}
\label{eq:nablaquequalsmathcalV} 
\nabla_qu(t,q,\sigma_t) = \mathcal{V}[s,\mu](t,q) , \quad 0\leq t \leq T, \ q \in \dtorus.   
\end{equation}  
Differentiating now with respect to $t$ in (\ref{eq:deffofu}), and substituting (\ref{eq:nablaquequalsmathcalV}) into it, we conclude that
\begin{equation}
\label{eq:totalofu} 
\partial_t (u(t,q,\sigma_t)) + H(q,\nabla_qu(t,q,\sigma_t)) + F(q,\sigma_t) = 0 . 
\end{equation} 
Thus, we have shown: 
\begin{lemma}
\label{lemma:dependsonbeingsymmetric} 
For $s\in[0,T],$ $\mu\in\wasstwospace,$ we have: \begin{enumerate}[(i)] \item For any $(t,q)\in[0,T]\times\dtorus,$ 
$$ 
u(0,\cdot,\mu) = g(\cdot,\mu)  \ , \quad \nabla_qu(t,q,\sigma_t) = \mathcal{V}[s,\mu](t,q). 
$$ 
\item The function $t\mapsto u(t,q,\sigma_t)$ is continuously differentiable and $$ \partial_t (u(t,q,\sigma_t)) + H(q,\nabla_qu(t,q,\sigma_t)) + F(q,\sigma_t) = 0 \ ,  \quad(t,q)\in[0,T]\times\dtorus. 
$$ 
\end{enumerate}   
\end{lemma}    
\section{Regularity of $\Sigma[s,\cdot](t,q)$}\label{section:regularityinmu} 
Besides the conditions of Section \ref{subsection:dataformfg}, we now activate the conditions of Section \ref{subsection:medata} for the remainder of the paper. 
\subsection{The discretized map $M$}\label{subsection:discretizedM} 
For the remainder of the paper, let \begin{equation}\label{eq:kgreaterthan} \theta  > \max\{ 1, 5\sqrt{2}\kappa\}, \end{equation} and $A_1,$ $A_2,$ $B,$ $E,$ $T,$ $D$ be as in Proposition \ref{prop:m0islipschitz} and Corollary \ref{cor:hamODEshassolution}, with $T$ being subject to Remark \ref{remark:remarkaboutconstantT}. The functions $\Sigma[s,\mu]$, $s\in[0,T],\mu\in\wasstwospace,$ as before, denote the fixed points of the operators $\mathfrak{m}^{s,\mu},$ while $\bar\Sigma[s,\mu]=(\Sigma^1[s,\mu],\frac{1}{\theta }\Sigma^2[s,\mu])=:(\bar\Sigma^1[s,\mu],\bar\Sigma^2[s,\mu])$ are the fixed points of the operators $\bar{\mathfrak{m}}^{s,\mu}$; recall (\ref{eq:SigmabartoSigma}) from the proof of Corollary \ref{cor:hamODEshassolution}. 

Let $M=(M_1,M_2)$ be the map $\Sigma=(\Sigma^1,\Sigma^2)$ restricted to average of Dirac masses. Namely: 
\begin{defn}
\label{defn:mapM} 
For any $s,t\in[0,T],$ $q\in\dtorus,$ $x\in(\dtorus)^n,$ let 
\begin{align*} 
M=(M_1,M_2): [0,T]\times[0,T]\times\dtorus\times(\dtorus)^n &\ \longrightarrow \dtorus\times\R^d \\ (t,s,q,x) &\ \longmapsto \Sigma[s,\mu^x](t,q) \\ & \ \qquad = (\Sigma^1[s,\mu^x](t,q),\Sigma^2[s,\mu^x](t,q)) . 
\end{align*} 
\end{defn} 
\noindent \textbf{Note.} The domain of the mapping $M$ depends on $n\in Z^+.$ 
$\sslash$ 
\begin{defn}
\label{defn:Mksequence} 
(i) For $n\in\N,$ Let $\bar M^k,$ $k=0,1,\ldots$ be the sequence of $\dtorus\times\R^d$-valued functions on $[0,T]\times[0,T]\times\dtorus\times(\dtorus)^n$ defined by 
$$ 
\bar M^0 \equiv (q,0), \quad  \bar M^{k+1}  = \bar{\mathfrak{m}}^{s,\mu^x}(\bar M^k(\cdot,s,\cdot,x)), 
$$ 
where $\mu^x = \frac{1}{n}\sum_{j=1}^n \delta_{x_j},$ $x=(x_1,\ldots,x_n)\in(\dtorus)^n.$ 
That is, 
\begin{equation*}
\bar  M_1^{k+1}(t,s,q,x) =  q+\int_s^t \nabla_pH(\bar M_1^k(\tau,s,q,x),\theta  \bar M_2^k(\tau,s,q,x)) d\tau  \end{equation*}
and
\begin{align}\bar M_2^{k+1}(t,s,q,x) = &\  \frac{1}{\theta} \nabla_q g_n(\bar{M}^k_1(0,s,q,x),\bar{M}^k_1(0,s,x_1,x),\cdots,\bar{M}^k_1(0,s,x_n,x)) \notag \\   & -  \frac{1}{\theta} \int_0^t [ \nabla_p H(\bar{M}^k_1(\tau,s,q,x),\theta\bar{M}^k_2(\tau,s,q,x)) \notag \\ & \quad \qquad + \nabla_q F_n(\bar{M}^k_1(t,s,q,x),\bar{M}^k_1(t,s,q,x_1),\cdots,\bar{M}^k_1(t,s,q,x_n)) ]  d\tau, \notag
\end{align} 
where 
$$ 
F_n(q,x) :=  F(q,\mu^x) , \quad g_n(q,x) := g(q,\mu^{x}).
$$(ii) Let $M^k, k=0,1,\ldots$ be the the sequence of $\dtorus\times\R^d$-valued functions on $[0,T]\times[0,T]\times\dtorus\times(\dtorus)^n$ defined by 
$$ 
M^0 \equiv (q,0), \quad  M^{k+1}  = \mathfrak{m}^{s,\mu^x}(M^k(\cdot,s,\cdot,x)). 
$$   
\end{defn} 
\begin{remark}
\label{remark:MandMbar} 
It follows that, for every $k=0,1,\ldots$ $$ M_1^{k}(t,s,q,x) = \bar M_1^k(t,s,q,x) ; \quad M_2^k(t,s,q,x) = \theta \bar M_2^k(t,s,q,x) , $$ $t, s \in [0,T], q\in\dtorus,$ $x\in(\dtorus)^n.$ $\sslash$ 
\end{remark} 
The objective now is to obtain estimates on the derivatives of $M^{k}$ with respect to $x.$ Using the definition of Wasserstein gradient directly, one verifies the formulas
\begin{equation}
\label{eq:firstwassfla}
\nabla^2_{x_i q}F_n(q,x) = \frac{1}{n}\nabla_{\mu}\nabla_q F(q,\mu^x)(x_i), \qquad \nabla^2_{x_i q}g_n(q,x) = \frac{1}{n}\nabla_{\mu}\nabla_q g(q,\mu^x)(x_i). 
\end{equation} 
The previous is a matrix equality: the entries on the left hand side are $\partial_{x_i^{(j)}}\partial_{q^{(l)}}F_n(q,x)$ and those on the right are $\frac{1}{n}\nabla_{\mu_j}\partial_{q^{(l)}}F(q,\mu^x)(x_i),$ for $j,l=1,\ldots,d,$ where $\nabla_{\mu_{j}}F(q,\mu)$ denotes the $j$-th component of the Wasserstein gradient of $\partial_{q^{(l)}}F(q,\mu^x)$ at $x_i.$ Furthermore, since $\nabla_q F$ and $\nabla_q g$ are twice differentiable in the measure variable, we also know that 
\begin{equation}
\label{eq:secondwassfla} 
\nabla^2_{x_jx_i}\nabla_q F_n(q,x) = \frac{1}{n^2} \nabla^2_{\mu\mu}\nabla_q F(q,\mu^x)(x_i,x_j),   \qquad \nabla^2_{x_jx_i}\nabla_q g_n(q,x) = \frac{1}{n^2} \nabla^2_{\mu\mu}\nabla_q g(q,\mu^x)(x_i,x_j). 
\end{equation} 
We begin with the $x_j$-derivative of the $(k+1)$-th iteration ($j=1,\ldots,n$), and in this and subsequent calculations, we will include in the arguments of the functions only the variables that are relevant to them.
\begin{align} 
\nabla_{x_j}\bar{M}^{k+1}_1(t,s,q,x) = \int_s^t [ \nabla_{qp}^2 H(\bar{M}^k_1,\theta\bar{M}^k_2)\nabla_{x_j}\bar{M}^k_1 + \theta\nabla_{pp}^2H(\bar{M}^k_1,\theta\bar{M}^k_2)\nabla_{x_j}\bar{M}^k_2] d\tau  \label{eq:nablaxjM1k}   ,  
\end{align}  
\begin{align} 
\nabla_{x_j} \bar{M}^{k+1}_2(t,s,q,x) = &\ \frac{1}{\theta} \nabla^2_{qq}g_n(\bar{M}^k_1(q,x),\bar{M}^k_1(x_1,x),\ldots,\bar{M}^k_1(x_n,x))\nabla_{x_j}\bar{M}^k_1(q,x) \notag \\ &\ +\frac{1}{\theta}\nabla_{x_j q}^2 g_n(\bar{M}^k_1(q,x),\bar{M}^k_1(x_1,x),\ldots,\bar{M}^k_1(x_n,x))[\nabla_q \bar{M}^k_1(x_j,x) + \nabla_{x_j}\bar{M}^k_1(x_j,x)] \notag \\ &\ +  \frac{1}{\theta}\sum_{i\neq j} \nabla_{x_i q}^2 g_n(\bar{M}^k_1(q,x),\bar{M}^k_1(x_1,x),\ldots,\bar{M}^k_1(x_n,x))\nabla_{x_j}\bar{M}^k_1(x_i,x) \notag \\ &\ - \frac{1}{\theta}\int_0^t [ \nabla_{qq}^2 H(\bar{M}^k_1,\theta\bar{M}^k_2)\nabla_{x_j}\bar{M}^k_1 + \theta\nabla_{pq}^2H(\bar{M}^k_1,\theta\bar{M}^k_2)\nabla_{x_j}\bar{M}^k_2] d\tau \notag \\ &\ - \frac{1}{\theta}\int_0^t [ \nabla^2_{qq}F_n(\bar{M}^k_1(q,x),\bar{M}^k_1(x_1,x),\ldots,\bar{M}^k_1(x_n,x))\nabla_{x_j}\bar{M}^k_1(q,x) \notag \\ &\ \qquad +\nabla_{x_j q}^2 F_n(\bar{M}^k_1(q,x),\bar{M}^k_1(x_1,x),\ldots,\bar{M}^k_1(x_n,x))[\nabla_q \bar{M}^k_1(x_j,x) + \nabla_{x_j}\bar{M}^k_1(x_j,x)] \notag \\ &\ \qquad +  \sum_{i\neq k} \nabla_{x_i q}^2 F_n(\bar{M}^k_1(q,x),\bar{M}^k_1(x_1,x),\ldots,\bar{M}^k_1(x_n,x))\nabla_{x_j}\bar{M}^k_1(x_i,x) ]  d\tau. \label{eq:nablaxjM2k} 
\end{align} 
For the next lemma, we remind the reader of Remark \ref{remark:Dbar}. 
\begin{lemma}
\label{lemma:nablaxjMk} 
Fix $n\in\Z^+.$ Using the terminology of Definition \ref{defn:Mksequence}, the following hold: \begin{enumerate}[(i)] \item For each $k=0,1,\ldots$~and each $x\in(\dtorus)^n,$ 
$$
M^k(\cdot,\cdot,\cdot,x)\in \mathcal M_{0,D}^*(A_1,\theta A_2,\theta B,E,E_1,\theta E_2,T).
$$  
\item There is a constant $C>0,$ independent of $k$ such that for any $j=1,\ldots,n $: 
\begin{equation}
\label{eq:ineqsnablaxjMk} 
\| \nabla_{x_j}M^k \|_{\infty} \leq \frac{C}{n}.    \end{equation} 
\end{enumerate}  
\end{lemma} 
\begin{proof} 
(i) Clearly, $\bar{M}^0(\cdot,\cdot,\cdot,x)\in \mathcal{M}_{0,\bar D}^*(A_1,A_2,B,E,E_1,E_2,T),$ and by Lemma \ref{lemma:continuityofSigmains}, each \newline $\bar{M}^k(\cdot,\cdot,\cdot,x)$ $\in$ $\mathcal{M}_{0,\bar D}^*(A_1,A_2,B,E,E_1,E_2,T).$ By Remark \ref{remark:Dbar}, each $M^k(\cdot,\cdot,\cdot,x)\in\mathcal{M}_{0,D}^*(A_1,\theta A_2,\theta B,E,E_1,\theta E_2,T).$ 

(ii) Recall formulas (\ref{eq:firstwassfla}), (\ref{eq:secondwassfla})---since $\nabla_q F,$ $\nabla_q g$ and their first and second order Wasserstein gradients are uniformly bounded (conditions imposed in Section \ref{subsection:medata}), we get 
\begin{align*} 
\| \nabla_{x_j} \bar{M}^{k+1} \|_{\infty} \leq & \big(\frac{\sqrt{2}\kappa}{\theta}E + \frac{\sqrt{2}\kappa}{\theta}TA_1\big)\frac{1}{n}  \\ & + \|\nabla_{x_j} \bar M\|_{\infty} \big( \frac{\sqrt{2}\kappa}{\theta} + \frac{\sqrt{2}\kappa}{\theta}\frac{1}{n} + \frac{\sqrt{2}\kappa}{\theta}\frac{n-1}{n} + \\ & \hspace{3cm} T\big[ 2\sqrt{2}(1+\theta)\bar{h}(B) + \frac{\sqrt{2}\kappa}{\theta} + \frac{\sqrt{2}\kappa}{\theta}\frac{1}{n} + \frac{\sqrt{2}\kappa}{\theta}\frac{n-1}{n} \big] \big) \\ \leq & \frac{\sqrt{2}\kappa}{\theta}(E+TA_1)\frac{1}{n}  +  \|\nabla_{x_j} \bar M^k\|_{\infty} \big( 3\frac{\sqrt{2}\kappa}{\theta} + T(2\sqrt{2}(1+\theta)\bar{h}(B) + 3\frac{\sqrt{2}\kappa}{\theta})\big) .  
\end{align*} 
 By (\ref{eq:kgreaterthan}), we can invoke Remark \ref{remark:remarkaboutconstantT} to obtain that the latter is an expression of the form 
 $$  \|\nabla_{x_j}\bar M^{k+1}\|_{\infty} \leq \frac{a}{n} + b\| \nabla_{x_j}\bar M^k \|_{\infty} $$ with positive constants $a, b$ in which $ b<1.$ This holds for every $k=0,1,\ldots$ Applying this inequality recursively (see \cite[Remark 8.1]{mfgmain}), we find a constant $C>0$ such that $\|\nabla_{x_j}\bar M^k\|_{\infty}\leq C/n, $ for every $k\in\Z^+.$ Thus, by Remark \ref{remark:MandMbar}, 
 \begin{equation}
 \label{eq:MbarMCn} 
 (\|\nabla_{x_j}M_1^k\|_{\infty}^2 + \|\frac{1}{\theta}\nabla_{x_j}M_2^k\|_{\infty}^2)^{1/2} = (\|\nabla_{x_j}\bar M_1^k\|_{\infty}^2 + \|\nabla_{x_j}\bar M_2^k\|_{\infty}^2)^{1/2} \leq \frac{C}{n}.
 \end{equation} 
 Multiplying by $\theta$ on both sides we get, since $\theta >1,$ that $  (\|\nabla_{x_j}M_1^k\|^2 + \|\nabla_{x_j}M_2^k\|^2)^{1/2} $ $\leq \frac{\theta C}{n},$ which is (\ref{eq:ineqsnablaxjMk}) for a larger constant $C.$ \end{proof}
\subsubsection{Regularity of $M$ in $x$, $q$ and $s$}
\label{subsubsection:regofMinxandq} 
\begin{cor}
\label{cor:firstcortonablaxjMx} 
The sequence $\{M^k\}_1^{\infty}$ of Definition \ref{defn:Mksequence}(ii) converges uniformly to the function $M$ of Definition \ref{defn:mapM}, with 
$$
M(\cdot,\cdot,\cdot,x) \in \mathcal{M}_{0,D}^*(A_1,\theta A_2,\theta B,E,E_1,\theta E_2,T) 
$$ 
for every $x\in(\dtorus)^n$ and $M(t,s,q,x)= M(t,s,q,\bar{x})$ whenever $\bar{x}$ is a permutation of $x.$ Moreover, there is a constant $C > 0$ such that 
\begin{align} 
\|\nabla^2_{qx_j}M\|_{\infty} &\ \leq \frac{C}{n}, \quad j=1,\ldots,n \ ,  \label{eq:firstcortonablaxjMxone}  \\ \|\nabla^2_{x_{i}x_j}M\|_{\infty} &\ \leq \frac{C}{n^2}, \quad i\neq j, \ i,j\in\{1,\ldots,n\} \ ,  \label{eq:firstcortonablaxjMxtwo}   \\  \|\nabla^2_{x_jx_j}M\|_{\infty} &\ \leq \frac{C}{n}, \quad j=1,\ldots,n \ , \label{eq:firstcortonablaxjMxthree} \\
 \|\nabla^2_{sx_j} M\|_{\infty} &\ \leq \frac{C}{n}, \quad j=1,\ldots,n. \label{eq:nablasxjM}  
 \end{align}
\textit{Note.}~Since on $W^{2;\infty}(\dtorus\times\dtorus)$ the mixed partial derivatives $\nabla^2_{x_jx_i}$ and $\nabla^2_{x_i x_j}$ are equal, estimate (\ref{eq:firstcortonablaxjMxone}) holds for $ \nabla^2_{x_j q}M$ too. \end{cor} \begin{proof} The proof of Lemma \ref{lemma:continuityofSigmains} shows that for every $x\in(\dtorus)^n,$ $M^k(\cdot,\cdot,\cdot,x)$ converges uniformly to $\Sigma[\cdot,\mu^x](\cdot,\cdot).$ Formula (\ref{eq:ineqsnablaxjMk}) means that the sequence $M^k$ is equicontinuous, and uniformly bounded, on $[0,T]\times[0,T]\times\dtorus\times(\dtorus)^n.$ It follows, by Ascoli's theorem, that the convergence of $M^k$ to $M$ is uniform, and $M$ also satisfies (\ref{eq:ineqsnablaxjMk}) of Lemma \ref{lemma:nablaxjMk}; to be more precise: \begin{equation}\label{eq:ineqnablaxjM} \| \nabla_{x_j}M \|_{\infty} \leq \frac{C}{n}.   \end{equation} If $\bar x$ is a permutation of $x$ then $\mu^x = \mu^{\bar{x}}$ and so $M(t,s,q,x) = M(t,s,q,\bar{x}).$ The three estimates above will be true of the limit function $M$ if they hold for every $M^k$. Like before, we are unable to obtain them for $M^k$ directly due to the size of the constant $\kappa,$ so we again do it for $\bar M^k$ first. 
Differentiating (\ref{eq:nablaxjM1k}) and (\ref{eq:nablaxjM2k}) with respect to $q,$ we get
\begin{align*} 
\nabla^2_{qx_j}\bar{M}^{k+1}_1(t,s,q,x)  = \int_s^t & [  (\nabla^3_{qqp}H(\bar{M}^k_1,\theta\bar{M}^k_2)\nabla_q \bar{M}^k_1 + \theta\nabla^3_{pqp}H(\bar{M}^k_1,\theta\bar{M}^k_2)\nabla_q \bar{M}^k_2)\nabla_{x_j}\bar{M}^k_1 \\  & + \nabla^2_{qp}H(\bar{M}^k_1,\theta\bar{M}^k_2)\nabla^2_{qx_j}\bar{M}^k_1  \\  &+\theta(\nabla^3_{qpp}H(\bar{M}^k_1,\theta\bar{M}^k_2)\nabla_q \bar{M}^k_1 + \theta\nabla^3_{ppp}H(\bar{M}^k_1,\theta\bar{M}^k_2)\nabla_q \bar{M}^k_2)\nabla_{x_j}\bar{M}^k_2 \\  & + \theta\nabla^2_{pp}H(\bar{M}^k_1,\bar{M}^k_2)\nabla^2_{qx_j}\bar{M}^k_2] d\tau,   
\end{align*} 
and\footnote{For the second order gradients of $\bar{K}_2,$ we will only write the part corresponding to $g_n,$ knowing that the one corresponding to $F_n$ has the exactly the same structure, and the expression coming from $H$ is the same as the one just displayed, except that the last subindex in $\nabla^3_{\square \square \square} $ is $q.$}
\begin{align*} 
& \qquad \qquad \qquad  \nabla^2_{q x_j} \bar{M}^{k+1}_2(t,s,q,x) = \\  & \ =   \frac{1}{\theta}\nabla^3_{qqq}g_n(\bar{M}^k_1(q,x),\bar{M}^k_1(x_1,x),\ldots,\bar{M}^k_1(x_n,x))\nabla_q \bar{M}^k_1(q,x) \nabla_{x_j} \bar{M}^k_1(q,x)  \\ & + \frac{1}{\theta}\nabla^2_{qq}g_n(\bar{M}^k_1(q,x),\bar{M}^k_1(x_1,x),\ldots,\bar{M}^k_1(x_n,x))\nabla^2_{x_j q}\bar{M}^k_1(q,x) \\ & + \frac{1}{\theta}\sum_{l\neq k} \big[ \nabla^3_{qx_lq}g_n(\bar{M}^k_1(q,x),\bar{M}^k_1(x_1,x),\ldots,\bar{M}^k_1(x_n,x))\nabla_q \bar{M}^k_1(q,x) \nabla_{x_j}\bar{M}^k_1(x_l,x) \\ &\ \qquad \qquad + \nabla^2_{x_l q}g_n(M_1(q,x),M_1(x_1,x),\ldots,M_1(x_n,x))\nabla^2_{qx_j}M_1(x_l,x) \big] \\ & + \frac{1}{\theta}\nabla^3_{qx_jq}g_n(\bar{M}^k_1(q,x),\bar{M}^k_1(x_1,x),\ldots,\bar{M}^k_1(x_n,x))\nabla_q \bar{M}^k_1(q,x)(\nabla_q\bar{M}^k_1(x_j,x)+\nabla_{x_j}\bar{M}^k_1(x_j,x)) \\ & \  +\frac{1}{\theta}\nabla^2_{x_jq}g_n(\bar{M}^k_1(q,x),\bar{M}^k_1(x_1,x),\ldots,\bar{M}^k_1(x_n,x))[\nabla^2_{qq}\bar M^k_1(x_j,x)+\nabla^2_{qx_j}\bar M^{k}_1(x_j,x)] \\ & - \frac{1}{\theta} \int_0^t \cdots \ d\tau.        
\end{align*} 
Using (\ref{eq:ineqsnablaxjMk}) and the bounds on the coefficients, we get: 
\begin{align*} 
\|\nabla_{qx_j}^2\bar{M}^{k+1}\|_{\infty} \leq  & \ 2 \sqrt{2}T(1+\theta)\bar{h}(B)(A_1+\theta A_2)\frac{C}{n} + \sqrt{2}\frac{\kappa}{\theta}\frac{1}{n}(E(1+C)+\frac{n-1}{n}E\frac{C}{n} + E(E+\frac{C}{n})) \\ & + \sqrt{2}\frac{\kappa}{\theta}T\frac{1}{n}(A_1(1+C)+\frac{n-1}{n}A_1\frac{C}{n} + A_1(A_1+\frac{C}{n})) \\   & \ + \|\nabla_{qx_j}^2\bar{M}^k\|_{\infty}\big(\sqrt{2}T\frac{1+\theta}{\theta}\bar{h}(B) + \frac{\sqrt{2}\kappa}{\theta}(\frac{1}{n} + 1 + \frac{n-1}{n}) + T\frac{\sqrt{2}\kappa}{\theta n}   \big)  , 
\end{align*} 
which is an inequality of the form 
$$ 
\|\nabla^2_{qx_j}\bar M^{k+1}\|_{\infty} \leq \frac{a}{n} + b  \|\nabla^2_{qx_j}\bar M^{k}\|_{\infty} 
$$ 
for constants $a,$ $b$ with $b<1,$ because $\theta > 4$ and $T$ is small. By induction again, increasing $C$ and switching back to $M^k$ in similar fashion to (\ref{eq:MbarMCn}), we obtain (\ref{eq:firstcortonablaxjMxone}) in the limit as $k\to\infty.$

We do not present here the full calculations (the reader may refer to \cite{majorga}), similar to the preceding one, that lead to (\ref{eq:firstcortonablaxjMxtwo}) and (\ref{eq:firstcortonablaxjMxthree}), but we provide a short description:

\textit{Case $i\neq j$:} 
The portion involving $g$ has the same bound as that involving $F,$ except the latter will be multiplied by $T$ in the estimation, so we can focus our attention on the terms coming from $g.$ We see that $1/n^2$ factors out from the terms that do not involve $\|\nabla^2_{x_jx_i}\bar{M}^k\|_{\infty},$ while the term multiplying $\|\nabla^2_{x_jx_i}\bar{M}^k\|_{\infty}$ is bounded by
$$ \frac{\kappa}{\theta} + \frac{(n-2)\kappa}{\theta n } +\frac{\kappa}{\theta n} + \frac{\kappa}{\theta n} + T(\cdots).
$$ 
Thus, because of the lower bound imposed on $\theta,$ we get that $$ \| \nabla^2_{x_jx_i}\bar M^{k+1}\|_{\infty} \leq  \frac{a}{n^2} + b\|\nabla_{x_jx_i}^2\bar{M}^{k+1}\|_{\infty},$$ $a,$ $b,$ with $b<1.$ The conclusion (\ref{eq:firstcortonablaxjMxtwo}) follows.

\textit{Case $i=j:$} This time only $1/n$, not $1/n^2,$ factors out from the terms not involving $\| \nabla^2_{x_jx_j} \bar{M}^k\|,$ and the term multiplying $\| \nabla^2_{x_jx_j} \bar{M}^k\|$ is 
$$ 
\frac{\kappa}{\theta} + (n-1)\frac{\kappa}{\theta n} + \frac{\kappa}{\theta n} + T(\cdots) .
$$
It then follows that
$$  
\|\nabla^2_{x_jx_j}\bar M^{k+1}\|_{\infty} \leq \frac{a}{n} + b  \|\nabla^2_{x_jx_j}\bar M^{k}\|_{\infty}
$$ 
with $b<1,$ for all $k\in\Z^+.$  Therefore (\ref{eq:firstcortonablaxjMxthree}) holds. Going through the same process once more, this time with the operator $\nabla_{sx_j}^2,$ ends up in
$$ 
\| \nabla_{sx_j}^2 \bar M^{k+1}\|_{\infty} \leq  a/n + b\|\nabla^2_{sx_j}\bar M^k\|_{\infty},
$$
with $b<1,$ from which the estimate (\ref{eq:nablasxjM}) follows.
\end{proof}
\subsubsection{First-order Taylor estimate} Our next step is to get an appropriate bound on the remainder of a first-order Taylor approximation of $M(\cdot,s,\cdot,\cdot)$ around $(t,s,q,x).$ Since 
\begin{align*} \partial_t M_1(t,s,q,x) = &\  \nabla_p H(M_1(t,s,q,x),M_2(t,s,q,x)) , 
\\ 
\partial_t M_2(t,s,q,x) = &\ - \nabla_q H(M_1(t,s,q,x),M_2(t,s,q,x)) -  \nabla_q F_n(M_1(t,s,q,x),M_1(t,s,q,x_1),\cdots,M_1(t,s,q,x_n)), 
\end{align*} 
differentiating once more with respect to $t$ and knowing that $M(\cdot,\cdot,\cdot,x)\in\mathcal{M}_{0,D}^*(A_1,\theta A_2,\theta B,E,E_1,\theta E_2,T)$ for every $x\in(\dtorus)^n,$ we obtain \begin{equation}\label{eq:threemore} \|\partial_{tt}^2 M\|_{\infty}, \ \|\nabla_{qt}^2 M\|_{\infty} \  \leq \ 2\sqrt{2}(h(\theta B)(A_1+\theta A_2)+\kappa A_1), \quad \|\nabla^2_{x_jt} M\|_{\infty} \leq \sqrt{2}(4h(\theta B)\frac{C}{n}+\frac{\kappa}{n}(A_1+2C)),   \end{equation} for $j=1,\ldots,n.$   Equipped now with estimates on all the second order derivatives of $M(\cdot,s,\cdot,\cdot),$ we proceed to obtain the Taylor estimate in the following corollary. \begin{cor}\label{cor:taylor1} Let $M = M(t,s,q,x),$ $M'=M(t',s,q',x'),$ with the notation of Corollary \ref{defn:Mksequence}. There is a constant $C > 0$ such that  \begin{align*}  & \ |M' - M - \partial_t M (t'-t) - \nabla_qM \cdot (q'-q) -\nabla_x M \cdot (x'-x)  | \\ \leq &\ C(|t'-t|^2 + |q-q'|^2 + |x-x'|^2).    \end{align*} The constant $C$ does not depend on $n.$  \end{cor} \noindent \textbf{Note.} The norm in the left-hand side of the latter inequality is the Euclidean norm on $\R^{2d}.$ $\sslash$    \begin{proof} Let $i\in\{1,2\},$ $j\in\{1,\ldots,d\}.$ Denoting $t'-t=\Delta t,$ $q'-q=\Delta q,$ $x_l'-x_l=\Delta x_l,$ and $|\Delta x| = (\sum_{l=1}^n|\Delta x_l|^2)^{1/2},$ the mean-value theorem implies that \begin{align*} & |M_i'^{(j)}-M_i^{(j)} -\partial_t M_i^{(j)}(t'-t) - \nabla_q M_i^{(j)}\cdot (q'-q) -\nabla_x M_i^{(j)}(x'-x)|     \\ & \quad \leq  \|\partial_{tt}^2M^{(j)}_i\|_{\infty}|\Delta t|^2 + 2 \| \nabla_{tq}^2 M^{(j)}_i \|_{\infty} |\Delta t| |\Delta q| + 2 \sum_{l=1}^n \| \nabla^2_{tx_l}M^{(j)}_i\|_{\infty} |\Delta t||\Delta x_l|  +  \|\nabla^2_{qq}M^{(j)}_i\|_{\infty}|\Delta q|^2  \\  &\ \qquad + 2 \sum_{l=1}^n \| \nabla^2_{qx_l}M^{(j)}_i \|_{\infty} |\Delta q||\Delta x_l| + \sum_{\substack{l,m=1\\l\neq m}}^n\| \nabla^2_{x_lx_m}M^{(l)}_i\|_{\infty} |\Delta x_l||\Delta x_m| + \sum_{l=1}^n \| \nabla^2_{x_lx_l}M^{(j)}_i\|_{\infty} |\Delta x_l|^2  .     \end{align*} Therefore, bringing in the estimates obtained in the foregoing paragraphs, we get \begin{align*} & \qquad  |M' - M - \partial_t M (t'-t) - \nabla_qM \cdot (q'-q) -\nabla_x M \cdot (x'-x)| \\ \leq &\ \|\partial_{tt}^2M\|_{\infty}|\Delta t|^2 + 2 \| \nabla_{tq}^2 M \|_{\infty} |\Delta t| |\Delta q| + 2 \sum_{l=1}^n \| \nabla^2_{tx_l}M\|_{\infty} |\Delta t||\Delta x_l|  +  \|\nabla^2_{qq}M\|_{\infty}|\Delta q|^2  \\  &\ + 2 \sum_{l=1}^n \| \nabla^2_{qx_l}M \|_{\infty} |\Delta q||\Delta x_l| + \sum_{\substack{l,m=1\\l\neq m}}^n\| \nabla^2_{x_lx_m}M^{(l)}_i\|_{\infty} |\Delta x_l||\Delta x_m| + \sum_{l=1}^n \| \nabla^2_{x_lx_l}M\|_{\infty} |\Delta x_l|^2  \end{align*} \begin{align*}  \leq &\  2\sqrt{2}(A_1+A_2)(\kappa + h(\theta B))(|\Delta t|^2 + 2|\Delta t||\Delta q|) + 4\sqrt{2}\sqrt{2}n\frac{C}{n}(2 h(\theta B)+\kappa)|\Delta x||\Delta t| \\ &\ \ + \sqrt{2}(A_1+A_2)|\Delta q|^2  + 2 n \frac{C}{n} \sqrt{2}|\Delta x||\Delta q| + n(n-1) \frac{C}{n^2} \sqrt{2}\sqrt{2}|\Delta x|^2 + n \frac{C}{n}|\Delta x|^2 .    \end{align*}  Thus, there is a larger constant, still denoted by $C,$ and not depending on $n,$ such that inequality in the corollary's statement holds.  
\end{proof} 
Given $x, x'\in(\dtorus)^n,$ we can reorder and shift the coordinates of $x'=(x'_1,\ldots,x'_n)$ so that $|x-x'|^2 = \mathscr{W}^2(\mu^x,\mu^{x'}).$ Thus, the inequality of Corollary \ref{cor:taylor1} reads 
\begin{align}  
& \ |M' - M - \partial_t M (t'-t) - \nabla_qM \cdot (q'-q) -\nabla_x M \cdot (x'-x)  | \notag \\ &\ \quad \leq C(|t'-t|^2 + |q-q'|^2 + \mathscr{W}^2(\mu^x,\mu^{x'})). \label{eq:taylor1} 
\end{align} 
\subsection{Regularity of the inverse of $M$}\label{subsection:regularityinversemaster} 
Let us define 
\begin{align} 
N : [0,T]\times[0,T]\times\dtorus\times(\dtorus)^n &\ \longrightarrow \dtorus \notag \\ (t,s,q,x) &\ \longmapsto X[s,\mu^x](t,q)  ; \label{eq:capitaln} 
\end{align} 
recall that $X[s,\mu](t,\cdot)$ is the inverse of $\Sigma^1[s,\mu](t,\cdot).$ 
The function $N$ takes values in $\dtorus,$ so it has only one component, unlike $M=(M_1,M_2).$ Thus, $M_1$ and $N$ are related by 
\begin{equation} 
M^1(t,s,N(t,s,q,x),x) = q , \qquad t,s\in[0,T],q\in\dtorus,x\in(\dtorus)^n. \notag 
\end{equation} 
We are going to derive now the Lipschitz property of $X[s,\cdot](\cdot,\cdot)$ before addressing the full regularity of $\Sigma[s,\cdot](\cdot,\cdot)$. Recall estimate (\ref{eq:boundsfornablaqsigma1}): \begin{equation*} \| \nabla_q X[s,\mu](t,\cdot) \|_{\infty} < 4(1+\sqrt{d})^{d-1}, \qquad s,t\in[0,T],\mu\in\wasstwospace.  \end{equation*}  
Differentiating the identity $ q \equiv X[s,\mu](t,\Sigma^1[s,\mu](t,q))$ with respect to $t,$ we have
\begin{equation*}
0 = \partial_tX[s,\mu](t,\Sigma^1[s,\mu](t,q)) + \nabla_q X[s,\mu](t,\Sigma^1[s,\mu](t,q))\partial_t\Sigma^1[s,\mu](t,q) , 
\end{equation*} 
from which \begin{equation}\label{eq:partialtX} \partial_t X[s,\mu](t,q) = - \nabla_q X[s,\mu](t,q) v[s,\mu](t,q),\end{equation} at any $s,t\in[0,T], q\in\dtorus, \mu\in\wasstwospace.$ Therefore \begin{equation}\label{eq:tinverse} \| \partial_t X_t[s,\mu] \|_{\infty} \leq \|\nabla_qX_t[s,\mu]\|_{\infty} \|v_t[s,\mu]\|_{\infty} \leq 4A_1(1+\sqrt{d})^{d-1}.  \end{equation} 
%
For the regularity with respect to $x=(x_1,\ldots,x_n)\in(\dtorus)^n,$ we use the identity $$ M_1(t,s,N(t,s,q,x),x) \equiv q , $$ which holds by definition. Taking the derivative with respect to $x_j,$ $j=1,\ldots,n,$ gives \begin{align*} - \nabla_{x_j}N(t,s,q,x) = [\nabla_q M_1(t,s,N(t,s,q,x),x)]^{-1}\nabla_{x_j}M_1(t,s,q,x) = \nabla_q N(t,s,q,x) \nabla_{x_j}M_1(t,s,q,x).   \end{align*} Thus, $\|\nabla_{x_j} N \|_{\infty} \leq 4(1+\sqrt{d})^{d-1} \frac{C}{n},$ which, increasing the value of $C,$ gives  \begin{equation}\label{eq:nablaxjN} \|\nabla_{x_j} N \|_{\infty} \leq  \frac{C}{n}. 
\end{equation}   
\begin{cor}\label{cor:taylor2} 
Let $t,t',s\in[0,T],$ $q',q\in\dtorus,$ $\mu,\nu\in\wasstwospace.$ Then there is a constant $C>0$ such that \begin{align*} |X[s,\nu](t',q')-X[s,\mu](t,q)| \leq C (|t-t'|+|q'-q|+\mathscr{W}(\mu,\nu)).   \end{align*}  \end{cor} \begin{proof}  Let $x,x'\in(\dtorus)^n,$ and $N = N(t,s,q,x),$ $N'=N(t',s,q',x'),$ where $N$ is defined in (\ref{eq:capitaln}). By the bounds (\ref{eq:tinverse}), (\ref{eq:boundsfornablaqsigma1}), (\ref{eq:nablaxjN}), and relabeling the sequence $x_1,\ldots,x_n$ and shifting the points so that $\mathscr{W}^2(\mu^x,\mu^{x'})=\sum_{j=1}^n|x_j-x_j'|^2,$ \begin{align*}  |N(t',s,q',x')-N(t,s,q,x)| \leq &\ \|\partial_tN\|_{\infty}|t'-t| + \|\nabla_q N\|_{\infty}|q'-q| + \sum_{j=1}^n \|\nabla_{x_j}N\|_{\infty}|x_j'-x_j| \\ \leq &\ 4A_1(1+\sqrt{d})^{d-1} |t'-t| + 4(1+\sqrt{d})^{d-1})|q'-q| + \sum_{j=1}^n \frac{C}{n}|x_j'-x_j| ;    \end{align*} therefore, since $\sum\frac{C}{n}|x_j'-x_j|\leq C(\sum 1/n)^{1/2}(\sum |x_j'-x_j|^2/n)^{1/2},$ we get, by increasing $C,$  \begin{align*}  |N' - N | \leq &\  C(|t'-t| + |q-q'| + \mathscr{W}(\mu^x,\mu^{x'})).    \end{align*} The constant $C$ does not depend on $n.$ Applying the last fact in the list of Section \ref{section:preliminaries}, we now extend this to the arbitrary measure case: let $\mu,\nu\in\wasstwospace,$ and $\{x(n)\}_{n=1}^{\infty}, \{x'(n)\}_{n=1}^{\infty}$, with $x(n),x'(n)\in(\dtorus)^n,$ sequences such that $$ \lim_{n\to\infty}\mathscr{W}(\mu^{x(n)},\mu) = 0, \qquad \lim_{n\to\infty}\mathscr{W}(\mu^{x'(n)},\nu) = 0. $$ Since, by definition, $N(t,s,q,x) = X[s,\mu^x](t,q),$ the latter estimate means $$| X[s,\mu^{x'(n)}](t',q') - X[s,\mu^{x(n)}](t,q) | \leq  C(|t'-t| + |q-q'| + \mathscr{W}(\mu^{x(n)},\mu^{x'(n)})), $$ for every $n\in\Z^n.$ Letting $n\to\infty,$ the continuity of $X$ in all its variables finalizes the proof.\end{proof} The regularity of $\Sigma$ in all its variables simultaneously is obtained in the next paragraphs from the foregoing properties of its discretized version. 
\subsection{Regularity properties of $\Sigma$ and composite functions}
\label{subsection:regpropsmastermap} 
We will follow the extension method of \cite{mfgmain} to get the regularity of $\Sigma$ in the measure variable.   The idea, roughly speaking, begins with introducing a Lipschitz extension of the derivative $\nabla_{x_j}M_1,$ $j=1,\ldots,n,$ that is defined at every measure $\mu,$ by way of a Moreau-Yosida type of extension, which becomes closer to $n\nabla_{x_j}M_1$ the larger $n$ ---the number of particles--- is. When the $n$-particle ordered sets $x^n=(x_1^n,\ldots,x_n^n)$ are chosen in such a way that $\delta^{x^n}\to \mu,$ the extension just mentioned will reveal itself as the Wasserstein gradient in the first-order Taylor approximation derived in the preceding paragraphs; recall (\ref{eq:eqvwassgradient}).

For fixed $n\in\Z^+,$ let \begin{align*} \mathcal{B} := &\ [0,T]\times[0,T]\times\dtorus\times\{(y_j,\mu^y) \ | \ y=(y_1,\ldots,y_n)\in(\dtorus)^n, j\in\{1,\ldots,n\} \} \\ \subset &\ [0,T] \times [0,T]\times \dtorus \times [(\dtorus)\times\wasstwospace].    \end{align*} A typical element of $\mathcal B$ is thus $ (t,s,q,(y_j,\mu^{y})) $ where $y$ is any $n$-particle ordered set $(y_1,\ldots,y_n)\in(\dtorus)^n$ and $y_j$ is any of its component particles. If $m\in\Z^+$ and $f:\mathcal{B}\to\R^m$ is a continuous function, let $$ \|f\|_{\mathcal{B}} := \sup\{ |f(t,s,q,(x_j,\mu^x))| \ \big| \ t,s\in[0,T],q\in\dtorus,x\in(\dtorus)^n, j\in\{1,\ldots,n\} \}. $$  For any continuous function $f=(f^{(1)},\ldots,f^{(m)}):\mathcal B\to\R^m$ such that \begin{equation}\label{eq:conditiononf} |f(t,s,q,(y_j,\mu^y)) - f(t,s,q,(x_i,\mu^x)) | \leq C (|x_i-y_j|_{\dtorus} + \mathscr{W}(\mu^x,\mu^y) + \frac{1}{n}),  \end{equation} where $t,s\in[0,T], q\in\dtorus, x,y\in(\dtorus)^n, i,j\in\{1,\ldots,n\},$ define \begin{align*} g^{(l)}(t,s,q,z,\mu) := &\ \inf\big\{f^{(l)}(t,s,q,(y_j,\mu^y))+ C (|z-y_j|_{\dtorus} + \mathscr{W}(\mu,\mu^y)) \ \big| \ y\in(\dtorus)^n, j\in\{1,\ldots,n\} \big\} , \\ l=1,\ldots,m , \\ g := &\ (g^{(1)},\ldots,g^{(m)}) ,   \end{align*} at any fixed $z\in\dtorus, \mu\in\wasstwospace.$ The function $g$ is thus an extension of $f$ from $\mathcal{B}$ to the full space $[0,T]^2\times\dtorus\times[\dtorus\times\wasstwospace].$ The following is \cite[Lemma 8.10]{mfgmain}. \begin{prop}\label{prop:extensionlemma} Suppose that (\ref{eq:conditiononf}) holds, and for any $x\in(\dtorus)^n,$ $j\in\{1,\ldots,n\},$ $f(\cdot,\cdot,\cdot,(x_j,\mu^x))$ is $C$-Lipschitz. Then \begin{enumerate}[(i)] \item $g$ is $\sqrt{3}C$-Lipschitz, \item $\|g|_{\mathcal{B}}-f\|_{\mathcal{B}}\leq C/n.$ \end{enumerate}  \end{prop} As in \cite{mfgmain}, we set, for $s,t\in[0,T],$ $q\in\dtorus,$ $x=(x_1,\ldots,x_n)\in(\dtorus)^n,$ $j=1,\ldots,n,$ \begin{equation}\label{eq:zetan} \zeta^n(t,s,q,(x_j,\mu^x)) = n\nabla_{x_j}M(t,s,q,x). \end{equation} The periodicity of $M$ in $q$ and $x$ ensures that $\zeta^n$ is well defined on $\mathcal{B}.$ \begin{cor}\label{cor:chin} (Extension of $\zeta^n$) For each $n\in\Z^+,$ there is a function $$ \chi^n : [0,T]\times[0,T]\times\dtorus\times\dtorus\times\wasstwospace \to \R^{d^2}\times\R^{d^2} $$ such that $ \chi^n |_{\mathcal{B}} = \zeta^n$ and, with a larger value of $C$ than before, \begin{enumerate}[(i)] \item $\chi^n$ is $C$-Lipschitz, \item $\|\chi^n|_{\mathcal{B}} - \zeta^n\|_{\mathcal{B}} \leq \frac{C}{n}.$   \end{enumerate}   \end{cor} \begin{proof} We check that $f=\zeta^n$ satisfies the conditions of Proposition \ref{prop:extensionlemma}.  The Lipschitz property in $t$ and $q$ follows from (\ref{eq:threemore}), while, in $s,$ from (\ref{eq:nablasxjM}). 
Hence, to obtain the Corollary, it is enough to prove that the condition (\ref{eq:conditiononf}) is satisfied by $f=\zeta^n.$ Fix then $x,y\in(\dtorus)^n,$ $i,j\in\{1,\ldots,n\},$ $s,t\in[0,T],$ $q\in\dtorus.$ Since the order in which we take the $n$ particles $x_1,\ldots,x_n$, which make up $x\in(\dtorus)^n,$ does not change $M(\cdots,x),$ and $\nabla_{x_j}M(\cdots,x)$ is periodic in $x,$ it can be assumed that: \begin{align*} \sum_{k\neq i,j}|x_k-y_k|^2 \leq \mathscr{W}^2(\mu^x,\mu^y) , \qquad |x_j-y_i| = |x_j-y_i|_{\dtorus}, \quad |x_i-y_j| = |x_i-y_j|_{\dtorus} ,  \\ \nabla_{x_j} M(t,s,q,y) = \nabla_{x_1}M(t,s,q,\bar{y}), \quad \nabla_{x_i}M(t,s,q,x) = \nabla_{x_1}M(t,s,q,\bar{x}) ,    \end{align*} where $\bar{y}$ denotes the result of shifting $y_j$ and $y_i$ to the first and second positions, respectively, in the $n$-uple $y,$ and $\bar{x}$ denotes the result of shifting $x_i$ and $x_j$ to the first and second positions, respectively, in the $n$-uple $x.$ Suppose, too, without loss of generality, that $i<j.$ In view of these simplifications, 
\begin{align*} 
|\nabla_{x_j}M(t,s,q,y)-&\nabla_{x_i}M(t,s,q,x)|\\  \leq &\  \|\nabla^2_{x_1,x_1} M\|_{\infty} |y_j-x_i| + \|\nabla^2_{x_1,x_2}M \|_{\infty}|y_i-x_j| +\sum_{k=1}^{i-1}\|\nabla^2_{x_1,x_{k+2}}M\|_{\infty}|y_k-x_k| \\  &\  \ + \sum_{k=i+1}^{j-1}\|\nabla^2_{x_1,x_{k+1}}M\|_{\infty}|y_k-x_k| + \sum_{k=j+1}^n\|\nabla^2_{x_1,x_k}M\|_{\infty}|y_k-x_k|.  
\end{align*} 
Therefore, by the bounds of Corollary \ref{cor:firstcortonablaxjMx}, \begin{align*} |\nabla_{x_j}M(t,s,q,y)-\nabla_{x_i}M(t,s,q,x)|  \leq &\ \frac{C}{n} |y_j-x_i| + \frac{C}{n^2}|y_i-x_j| + \frac{C}{n^2}  \sum_{k\neq i,j}|y_k-x_k| \\  \leq &\  \frac{C}{n} |y_j-x_i|_{\dtorus} + \frac{C}{n^2}|y_i-x_j|_{\dtorus} + \frac{C\sqrt{n}}{n^2}\sqrt{\sum_{k\neq i,j}|y_k-x_k|^2} \\ \leq &\ \frac{C}{n}(|y_j-x_i|_{\dtorus} + \frac{\sqrt{d}}{2n} + \mathscr{W}(\mu^x,\mu^y)), \end{align*} where $\sqrt{d}/2$ is the diamater of $\dtorus.$  Thus, 
\begin{equation*}
|n\nabla_{x_j}M(t,s,q,y) - n\nabla_{x_i}M(t,s,q,x)| \leq \sqrt{d} C (|y_j-x_i|_{\dtorus} + \mathscr{W}(\mu^x,\mu^y) + \frac{1}{n}),   
\end{equation*} 
which proves property (\ref{eq:conditiononf}) for $f=\zeta^n,$ since $i$ and $j$ were arbitrary.     
\end{proof} 
\begin{lemma}
\label{lemma:mastermapisdiffble} 
For every $s\in[0,T],$ the $\dtorus\times\R^d$-valued map $\Sigma[s,\cdot](\cdot,\cdot)$ is differentiable on $\wasstwospace\times[0,T]\times\dtorus,$ that is: there exists a mapping \begin{align*} \bar\nabla_{\mu}\Sigma:[0,T]\times[0,T]\times\dtorus\times\dtorus\times\wasstwospace &\ \longrightarrow \ \R^{d^2}\times\R^{d^2} \\ (t,s,q,x,\mu) &\ \longmapsto \bar\nabla_{\mu}\Sigma[s,\mu](t,q,x)  \end{align*} such that, for every $s,t,t'\in[0,T],$ $q,q'\in\dtorus,$ $\mu,\nu\in\wasstwospace,$ $\gamma\in\Gamma_0(\mu,\nu)$, \begin{align}\big|&\Sigma[s,\nu](t',q')-\Sigma[s,\mu](t,q) - \partial_t\Sigma[s,\mu](t,q)(t'-t) - \nabla_q\Sigma[s,\mu](t,q)\cdot (q'-q)\notag \\ &\ \quad  - \underset{\dtorus\times\dtorus}{\int} \bar\nabla_\mu\Sigma[s,\mu](t,q,x)\cdot (y-x)\gamma(dx,dy)   \big|  \notag \\  &\leq C (|t'-t|^2 + |q'-q|^2 + \mathscr{W}^2(\mu,\nu)) .   \label{eq:munugammazero} \end{align} Moreover, the mapping $\bar\nabla_{\mu}\Sigma$ is Lipschitz.     \end{lemma} \begin{proof} Let $\mu, \nu\in\wasstwospace,$ and let $\gamma\in\Gamma_0(\mu,\nu).$ Appealing to the last fact in the fist list of the Preliminaries, there is a sequence $\{\gamma(n)\}_{n=1}^{\infty},$ converging narrowly to $\gamma$ in $\mathcal{P}(\dtorus\times\dtorus),$ such that $$ \gamma(n) = \frac{1}{n}\sum_{j=1}^n \delta_{(x_j(n),y_j(n))} , $$ and for each $j\in\{1,\ldots,n\},$ $(x_j(n),y_j(n))$ belongs to the support of $\gamma.$ Due to this latter fact (see, for instance, \cite[Theorem 6.1.4]{gradientflows}), for each $n\in\Z^+,$ the sequence $\{(x_j(n),y_j(n))\}_{j=1}^n$ is $|\cdot|_{\dtorus}$-monotone, and, therefore, $$ \gamma(n) \in \Gamma_0(\mu^{x(n)},\mu^{y(n)}) , \  \qquad n\in\Z^+.$$ It is also true that 
$$ 
\lim_{n\to\infty} \mathscr{W}(\mu,\mu^{x(n)}) = 0, \qquad \lim_{n\to\infty} \mathscr{W}(\nu,\mu^{y(n)}) = 0. 
$$ 
Let $\zeta^n$ be defined as in (\ref{eq:zetan}), so that, as a consequence, for each $x(n)$ of our sequence, \begin{align*} \zeta^n(t,s,q,(x_j(n),\mu^{x(n)})) =  n\nabla_{x_j(n)}M(t,s,q,x(n)) =  n\nabla_{x_j(n)}\Sigma[s,\mu^{x(n)}](t,q)  , \end{align*} $j\in\{1,\ldots,n\}.$   Recall the second-order estimate (\ref{eq:taylor1}), now with $x=x(n)$ and $x'=y(n)$: \begin{align*} \big|&\Sigma[s,\mu^{y(n)}](t,q)-\Sigma[s,\mu^{x(n)}](t',q') - \partial_t\Sigma[s,\mu^{x(n)}](t,q)(t'-t) - \nabla_q\Sigma[s,\mu^{x(n)}](t,q)\cdot (q'-q)  \\ &\ \quad  - \sum_{j=1}^n \nabla_{x_j(n)}\Sigma[s,\mu^{x(n)}](t,q)\cdot (y_j(n)-x_j(n)) \big|\\ &\leq C (|t'-t|^2 + |q-q'|^2 + \mathscr{W}^2(\mu^{x(n)},\mu^{y(n)})).   \end{align*} Since 
$$ \frac{1}{n} \sum_{j=1}^n \zeta^n(t,s,q,(x_j(n),\mu^{x(n)})) \cdot (y_j(n)-x_j(n)) = \underset{\dtorus\times\dtorus}{\int} \zeta^n(t,s,q,(x,\mu^{x(n)}))\cdot (y-x)\gamma(n)(dx,dy) ,  $$ the latter inequality is the same as \begin{align*} \big|&\Sigma[s,\mu^{y(n)}](t',q')-\Sigma[s,\mu^{x(n)}](t,q) - \partial_t\Sigma[s,\mu^{x(n)}](t,q)(t'-t) - \nabla_q\Sigma[s,\mu^{x(n)}](t,q)\cdot (q'-q)  \\ &\ \quad  - \underset{\dtorus\times\dtorus}{\int} \zeta^n(t,s,q,(x,\mu^{x(n)}))\cdot (y-x)\gamma(n)(dx,dy)   \big|\\ &\leq C (|t'-t|^2 + |q-q'|^2 + \mathscr{W}^2(\mu^{x(n)},\mu^{y(n)})). \end{align*} Denote by $\chi^n$ the extension of $\zeta^n$ furnished by Corollary \ref{cor:chin}. The same inequality holds if we substitute $\chi^n$ for $\zeta^n$ in the previous inequality, because these functions coincide on the set $\mathcal B,$ which includes the support of $\gamma(n).$ But, since we will pass to the limit, we rather write \begin{align} \big|&\Sigma[s,\mu^{y(n)}](t,q)-\Sigma[s,\mu^{x(n)}](t,q) - \partial_t\Sigma[s,\mu^{x(n)}](t,q)(t'-t) - \nabla_q\Sigma[s,\mu^{x(n)}](t,q)\cdot (q'-q)\notag   \\ &\ \quad  - \underset{\dtorus\times\dtorus}{\int} \chi^n(t,s,q,(x,\mu^{x(n)}))\cdot (y-x)\gamma(n)(dx,dy)   \big| \notag  \\ &\leq C (|t'-t|^2 + |q-q'|^2 + \mathscr{W}^2(\mu^{x(n)},\mu^{y(n)})) \notag  \\ &\ \quad + \big|\underset{\dtorus\times\dtorus}{\int} [ \zeta^n(t,s,q,(x,\mu^{x(n)})) - \chi^n(t,s,q,(x,\mu^{x(n)})) ]\cdot (y-x)\gamma(n)(dx,dy) \big| \notag \\ &\leq C (|t'-t|^2 + |q-q'|^2 + \mathscr{W}^2(\mu^{x(n)},\mu^{y(n)})) + \frac{C}{n} \mathscr{W}(\mu^{x(n)},\mu^{y(n)}), \label{eq:longineq}   \end{align} by Corollary \ref{cor:chin}(ii) and the fact $\gamma(n)\in\Gamma_0(\mu^{x(n)},\nu^{y(n)}).$ Now, by Corollary \ref{cor:chin}(i), each $\chi^n,$ $n=1,\ldots$ is $C$-Lipschitz on the bounded domain $[0,T]^2\times(\dtorus)^2\times\wasstwospace.$ The functions $\chi^n$ are also pointwise uniformly bounded, because of (\ref{eq:ineqnablaxjM}). Thus, the sequence $\{\chi^n\}_{n=1}^{\infty}$ is equicontinuous and pointwise uniformly bounded, so a subsequence of it converges to a $C$-Lipschitz mapping, which we define as the mapping $\bar \nabla_{\mu}\Sigma$ introduced in the statement of this lemma. Passing to the limit as $n\to\infty$ in (\ref{eq:longineq}), we prove (\ref{eq:munugammazero}), in particular, that $\Sigma$ is differentiable in the $\mu$ variable.
\end{proof} We will denote by $\bar\nabla_{\mu}\Sigma^1$ the first component (the $\dtorus$-valued part) of $\bar\nabla_{\mu}\Sigma,$ and $\bar\nabla_{\mu}\Sigma^2$ the second component ($\R^d$-valued) of $\bar\nabla_{\mu}\Sigma.$ 

Next we prove the analogue of Lemma \ref{lemma:mastermapisdiffble} for $X=(\Sigma^1)^{-1}.$ 
\begin{defn}
\label{defn:nablabarmuX} 
For $t,s\in[0,T],$ $\mu\in\wasstwospace,$ $q,x\in\dtorus,$ put 
$$ 
\bar\nabla_{\mu}X[s,\mu](t,q,x) := -\nabla_qX[s,\mu](t,q)\bar\nabla_{\mu}\Sigma^1[s,\mu](t,X[s,\mu](t,q),x) .  
$$   
\end{defn} 
Before stating and proving the lemma, we recall formula (\ref{eq:partialtX}): \begin{align*}   \partial_t X[s,\mu](t,q) = -\nabla_qX[s,\mu](t,q)\partial_t\Sigma^1[s,\mu](t,\cdot)\circ X[s,\mu](t,q).  \end{align*} \begin{lemma}\label{lemma:inversemastermapisdiffble} For every $s\in[0,T],$ the $\R^d$-valued map $X[s,\cdot](\cdot,\cdot)$ is differentiable on $\wasstwospace\times[0,T]\times\dtorus,$ i.e., there is a constant $C>0$ such that for every $s,t,t'\in[0,T],$ $q,q'\in\dtorus,$ $\mu,\nu\in\wasstwospace,$ $\gamma\in\Gamma_0(\mu,\nu),$ \begin{align} \big|& X[s,\nu](t',q')-X[s,\mu](t,q) - \partial_tX[s,\mu](t,q)(t'-t) - \nabla_qX[s,\mu](t,q)\cdot (q'-q)\notag \\ &\ \quad  - \underset{\dtorus\times\dtorus}{\int} \bar\nabla_\mu X[s,\mu](t,q,x)\cdot (y-x)\gamma(dx,dy)   \big|  \notag \\  &\leq C (|t'-t|^2 + |q-q'|^2 + \mathscr{W}^2(\mu,\nu)) ,   \label{eq:munugammazero1} \end{align} where $\partial_t X,$ $\nabla_q X,$ and the mapping $\bar\nabla_{\mu}X$ of Definition \ref{defn:nablabarmuX}, are continuous.     
\end{lemma} \begin{proof} The continuity of $\bar\nabla_{\mu}X$ is immediate from its definition, and the continuity of $\nabla_q X$ and $\partial_t X$ has been known since Lemma \ref{lemma:invertibilityofsigma} and formula (\ref{eq:partialtX}) respectively. Let us put \begin{equation}\label{eq:qtilde} \tilde q = X[s,\mu](t,q), \quad \tilde q' = X[s,\nu](t',q').  \end{equation} We write out the expression on the left hand side of (\ref{eq:munugammazero1}) and factor out $\nabla_q X[s,\mu](t,q)$, while also using the fact that $\nabla_qX[s,\mu](t,q)$ and $\nabla_q\Sigma^1[s,\mu](t,\tilde q)$ are inverses of one another: \begin{align*} &\ \bigg| \nabla_qX[s,\mu](t,q)  \bigg[ \nabla_q\Sigma^1[s,\mu](t,\tilde q)(\tilde q' -\tilde q) +    \partial_t \Sigma^1[s,\mu](t,\tilde q)(t'-t) \\ &\ - (\Sigma^1[s,\nu](t',\tilde q') - \Sigma^1[s,\mu](t,\tilde q)) + \int_{\dtorus\times\dtorus} \bar\nabla_{\mu}\Sigma^1[s,\mu](t,\tilde q,x)\cdot (y-x)\gamma(dx,dy) \bigg] \bigg| \\  = &  \ \bigg| -\nabla_q X[s,\mu](t,q) \bigg[ \Sigma^1[s,\nu](t',\tilde q') - \Sigma^1[s,\mu](t,\tilde q) -   \partial_t \Sigma^1[s,\mu](t,\tilde q)(t'-t) \\ &\  - \nabla_q\Sigma^1[s,\mu](t,\tilde q)(\tilde q' -\tilde q)  - \int_{\dtorus\times\dtorus}  \bar\nabla_{\mu}\Sigma^1[s,\mu](t,\tilde q,x)\cdot (y-x)\gamma(dx,dy) \bigg]  \bigg| \\ \leq\  & \ 4(1+\sqrt{d})^{d-1} C(|t'-t|^2 + |\tilde q'-\tilde q|^2 + \mathscr{W}^2(\mu,\nu)) \\  = & \   4(1+\sqrt{d})^{d-1}C(|t'-t|^2 + |X[s,\nu](t',\tilde q')-X[s,\mu](t,q)|^2 + \mathscr{W}^2(\mu,\nu)).   \end{align*} But, by Corollary \ref{cor:taylor2}, the term $|X[s,\nu](t',q')-X[s,\mu](t,q)|$ is bounded by $C(|t'-t|+|q'-q|+\mathscr{W}(\mu,\nu))$ for some $C>0.$ Inserting this bound into the last expression, after expanding and raising the value of $C,$ one obtains (\ref{eq:munugammazero1}).     \end{proof} 
\paragraph{Regularity of $\mathcal V$}
Let us look back at the definition of $\mathcal{V},$ given by (\ref{eq:mathcalV}). Set now \begin{equation}\label{eq:defofnablamumathcalV} \bar\nabla_{\mu}\mathcal V[s,\mu](t,q,x) := \bar\nabla_{\mu}\Sigma^2[s,\mu](t,X[s,\mu](t,q),x) + \nabla_q\Sigma^2[s,\mu](t,X[s,\mu](t,q))\bar\nabla_{\mu}X[s,\mu](t,q,x),      \end{equation} for $s,t\in[0,T],$ $q,x\in\dtorus,$ $\mu\in\wasstwospace.$ \begin{lemma}\label{lemma:mathcalVisdiffble} For every $s\in[0,T],$ the $\R^d$-valued map $\mathcal V[s,\cdot](\cdot,\cdot)$ is differentiable on $\wasstwospace\times[0,T]\times\dtorus,$ i.e., there is a constant $C>0$ such that for every $s,t,t'\in[0,T],$ $q,q'\in\dtorus,$ $\mu,\nu\in\wasstwospace,$ $\gamma\in\Gamma_0(\mu,\nu),$ \begin{align} \big|& \mathcal V[s,\nu](t',q')-\mathcal V[s,\mu](t,q) - \partial_t\mathcal V[s,\mu](t,q)(t'-t) - \nabla_q\mathcal V[s,\mu](t,q)\cdot (q'-q)\notag \\ &\ \quad  - \underset{\dtorus\times\dtorus}{\int} \bar\nabla_\mu \mathcal V[s,\mu](t,q,x)\cdot (y-x)\gamma(dx,dy)   \big|  \notag \\  &\leq C (|t'-t|^2 + |q-q'|^2 + \mathscr{W}^2(\mu,\nu)) ,   \label{eq:munugammazero2} \end{align} where the mapping $\bar\nabla_{\mu}\mathcal V,$ defined by (\ref{eq:defofnablamumathcalV}), and $\partial_t\mathcal V,$ $\nabla_q\mathcal V,$ are continuous. \end{lemma} \begin{proof}  We know that $$ \nabla_q\mathcal V_t[s,\mu] = \nabla_q\Sigma^2_t[s,\mu]\nabla_qX_t[s,\mu], \quad \partial_t\mathcal V_t[s,\mu] = \partial_t\Sigma^2_t[s,\mu] + \nabla_q\Sigma^2_t[s,\mu]\partial_tX_t[s,\mu]. $$ Therefore, the continuity of the functions stated in the lemma follows from that of $\nabla_q\Sigma,$ $\nabla_q X,$ $\partial_t \Sigma,$ $\partial_t X,$ and the continuity, proved above, of $\bar\nabla_{\mu}\Sigma$ and $\bar\nabla_{\mu}X.$   Keeping the notation (\ref{eq:qtilde}), we first write down the expression to estimate, i.e.~the left hand side of (\ref{eq:munugammazero2}), and factor out $\nabla_q \Sigma^2[s,\mu](t,\tilde q)$:    \begin{align*} &\  \big| \Sigma^2[s,\nu](t',\tilde q')-\Sigma^2[s,\mu](t,\tilde q) - \nabla_q\Sigma^2[s,\mu](t,\tilde q)\nabla_q X[s,\mu](t,q)(q'-q) \\ &\ \quad - (\partial_t\Sigma^2[s,\mu](t,\tilde q) +\nabla_q\Sigma^2[s,\mu](t,\tilde q) \partial_t X[s,\mu](t,q))(t'-t) \\ &\ \quad - \int_{\dtorus\times\dtorus}\big[ 
\bar\nabla_{\mu}\Sigma^2[s,\mu](t,\tilde q,x) +  \nabla_q\Sigma^2[s,\mu](t,\tilde q)\bar\nabla_{\mu}X[s,\mu](t,q,x)   
\big] \cdot 
(y-x) \gamma(dx,dy) \big| \\ = &\  \big| \Sigma^2[s,\nu](t',\tilde q')-\Sigma^2[s,\mu](t,\tilde q) - \nabla_q\Sigma^2[s,\mu](t,\tilde q) \ \big(  \\ &\ \qquad \nabla_q X[s,\mu](t,q)(q'-q) + \partial_t X[s,\mu](t,q)(t'-t) + \int_{\dtorus\times\dtorus} \bar\nabla_{\mu}X[s,\mu](t,q)\cdot (y-x)\gamma(dx,dy) \ \big) \\ &\ - \partial_t\Sigma^2[s,\mu](t,\tilde q)(t'-t) - \int_{\dtorus\times\dtorus} \bar\nabla_{\mu}\Sigma^2[s,\mu](t,\tilde q,x)\cdot (y-x)\gamma(dx,dy) \big|. \end{align*} Inside the large round brackets we add and substract $\tilde q' - \tilde q = X_{t'}[s,\nu](q') - X_t[s,\mu](q),$ and apply (\ref{eq:munugammazero1}), to obtain that the latter expression is no greater than \begin{align*} &\ \big| \Sigma^2_{t'}[s,\nu](\tilde q') - \Sigma^2_t[s,\mu](\tilde q) - \nabla_q\Sigma^2_t[s,\mu](\tilde q)(\tilde q - q)  - \partial_t\Sigma^2_t[s,\mu](\tilde q)(t'-t) \\ &\ - \int_{\dtorus\times\dtorus}\bar\nabla_{\mu}\Sigma^2_t[s,\mu](\tilde q,x)\cdot (y-x)\gamma(dx,dy)\big| + |\nabla_q\Sigma_t^2[s,\mu](\tilde q)| C( |q'-q|^2 + |t'-t|^2 +\mathscr{W}^2(\mu,\nu)),  \end{align*} which, in turn, by (\ref{eq:munugammazero1}), is bounded above by \begin{align*} &\ C (|X[s,\nu](t',q') - X[s,\mu](t,q)|^2 + |t'-t|^2 + |q'-q|^2 + \mathscr{W}^2(\mu,\nu)) \\ &\ +  \theta A_2 C( |q'-q|^2 + |t'-t|^2 +\mathscr{W}^2(\mu,\nu)), \end{align*} and, after using Corollary \ref{cor:taylor2} again, simplifying and increasing the value of $C,$ inequality (\ref{eq:munugammazero2}) is obtained.          \end{proof} 
\paragraph{Regularity of $H(q,\mathcal{V})$} 
We finish this section by following the previous results with the regularity of what turned out to be the second function that appears in the MFG equation (\ref{eq:mfg1}), due to (\ref{eq:mathcalvisgradient}). 
Set \begin{equation}\label{eq:defofnablamuH} \bar\nabla_{\mu}H(q,\mathcal V[s,\mu](t,q))(x) := \nabla_pH(q,\mathcal{V}[s,\mu](t,q))\bar\nabla_{\mu}\mathcal{V}[s,\mu](t,q,x),    \end{equation} for $s,t\in[0,T],$ $q,x\in\dtorus,$ $\mu\in\wasstwospace.$ \begin{lemma}\label{lemma:Hisdiffble} For every $s\in[0,T],$ the $\R$-valued map $H(\cdot,\mathcal V[s,\cdot](\cdot,\cdot))$ is differentiable on $\wasstwospace\times[0,T]\times\dtorus,$ and there is a constant $C>0$ such that for every $s,t,t'\in[0,T],$ $q,q'\in\dtorus,$ $\mu,\nu\in\wasstwospace,$ $\gamma\in\Gamma_0(\mu,\nu),$ \begin{align} \big|& H(q',\mathcal V[s,\nu](t',q'))-H(q,\mathcal V[s,\mu](t,q)) - (\partial_t)(H(q,\mathcal V[s,\mu](t,q))(t'-t) \notag \\  &\ \quad - (\nabla_q)(H(q,\mathcal V[s,\mu](t,q)))\cdot (q'-q)  - \underset{\dtorus\times\dtorus}{\int} \bar\nabla_\mu H(q,\mathcal V[s,\mu](t,q))(x)\cdot (y-x)\gamma(dx,dy)   \big|  \notag \\  &\leq C (|t'-t|^2 + |q-q'|^2 + \mathscr{W}^2(\mu,\nu)) ,   \label{eq:munugammazero3} \end{align} where the mapping $\bar\nabla_{\mu} H,$ defined by (\ref{eq:defofnablamuH}), is continuous.   \end{lemma} \begin{proof} Let us abbreviate $\mathcal{V} = \mathcal{V}_t[s,\mu](q)$ and $\mathcal{V}'=\mathcal{V}_{t'}[s,\nu](q')$. Since $$(\partial_t)(H(q,\mathcal{V})) = \nabla_p H(q,\mathcal{V})\partial_t\mathcal{V}, \quad (\nabla_q)(H(q,\mathcal{V})) = \nabla_qH(q,\mathcal{V}) + \nabla_pH(q,\mathcal{V})\nabla_q\mathcal{V},$$ the left hand side of (\ref{eq:munugammazero3}) is, after factoring out $\nabla_pH(q,\mathcal{V}),$  \begin{align*} & \ \big| H(q',\mathcal{V}') - H(q,\mathcal{V})  -\nabla_pH(q,\mathcal V)\big[ -(\mathcal V' - \mathcal V) + \nabla_q\mathcal{V}(q'-q) + \partial_t\mathcal{V}(t'-t) \\ &\ \hspace{5.7cm} + \int_{\dtorus\times\dtorus}\bar\nabla_{\mu}\mathcal{V}_t[s,\mu](q)\cdot(y-x)\gamma(dx,dy)    \big] \\ &\ \ -\nabla_qH(q,\mathcal{V})(q'-q) -\nabla_pH(q,\mathcal{V}')(\mathcal{V}-\mathcal{V}')\big| \\ \leq &\ |   H(q',\mathcal{V}') - H(q,\mathcal{V}) - \nabla_qH(q,\mathcal{V})(q'-q) -\nabla_pH(q,\mathcal{V})(\mathcal{V}'-\mathcal{V}) | \\ &\ + |\nabla_pH(q,\mathcal{V})|C(|t'-t|^2+|q'-q|^2+\mathscr{W}^2(\mu,\nu)) .       \end{align*} Remember now that $|\Sigma^2|\leq \theta B$ (see Corollary \ref{cor:hamODEshassolution}(ii)) at any $t,q,s,\mu ;$ recall Definition \ref{defn:databoundsii}. Therefore, the right-hand side of this inequality is bounded by \begin{align*} h(\theta B)(|q'-q|^2 + |\mathcal{V}'-\mathcal{V}|^2) + l(\theta B)C(|t'-t|^2+|q'-q|^2+\mathscr{W}^2(\mu,\nu)). \end{align*} To deal with the term $|\mathcal{V}'-\mathcal{V}|^2,$ note that Corollary \ref{cor:taylor2} is also valid for $\Sigma$ in place of $X,$ following a similar argument. With the notation (\ref{eq:qtilde}), $$ |\mathcal{V}' - \mathcal{V}| = |\Sigma^2[s,\nu](t',\tilde q')-\Sigma^2[s,\mu](t,\tilde q)| \leq C (|t'-t| + |\tilde q'-\tilde q| + \mathscr{W}(\mu,\nu)). $$ Applying Corollary \ref{cor:taylor2}, and raising the value of $C,$ we get $$|\mathcal{V}'-\mathcal{V}|\leq C  (|t'-t| + |q'-q| + \mathscr{W}(\mu,\nu)).$$ Substituting this into the bounding expression, and simplifying, one arrives at (\ref{eq:munugammazero3}), for some larger value of $C.$ The continuity of $\bar\nabla_{\mu}H$ in all its variables is clear from the definition.
\end{proof}
\section{Solution to the master equation}
\label{section:solutiontothemasterequation} 
Let us recall the definition of the function $u$ from Section \ref{subsection:thecharacteristicsmethod}:  \begin{equation}
\tag{\ref{eq:defofu}} 
u(s,q,\mu) = g(q,\Sigma^1[s,\mu](0,\cdot)_{\#}\mu) - \int_0^s \big[ H(q,\mathcal{V}[s,\mu](\tau,q))+F(q,\Sigma^1[s,\mu](\tau,\cdot)_{\#}\mu \big] d\tau. 
\label{eq:specificu} 
\end{equation} 
\subsection{Pathwise gradients of the couplings}\label{subsection:gradientsofFandg}
For any $s, t\in [0,T),$ $q\in\dtorus,$ $\mu\in\wasstwospace,$ define
\begin{align*} \mathcal{N}^F_t[s,\mu](q)(x) := &\ - \nabla_{\mu}F(q,\sigma_t)(\Sigma^1_t[s,\mu](x))  \nabla_q \Sigma^1_t[s,\mu](x) \\ & + \int_{\dtorus}\nabla_{\mu}F(q,\sigma_t)(\Sigma^1_t[s,\mu](r)) \bar\nabla_{\mu}\Sigma^1_t[s,\mu](r)(x)\mu(dr),
\end{align*}
and 
\begin{align*} \mathcal{N}^g_t[s,\mu](q)(x) := &\ - \nabla_{\mu}g(q,\sigma_t)(\Sigma^1_t[s,\mu](x))  \nabla_q \Sigma^1[s,\mu](t,x) \\ & + \int_{\dtorus}\nabla_{\mu}g(q,\sigma_t)(\Sigma^1_t[s,\mu](r)) \bar\nabla_{\mu}\Sigma^1_t[s,\mu](r)(x)\mu(dr).
\end{align*}
Likewise,
\begin{align*}
(\partial_t)(F(q,\Sigma^1_t[s,\mu]_{\#}\mu)): = \nabla_{\mu}F(q,\sigma_t)(\Sigma^1_t[s,\mu](x))\cdot \partial_t \Sigma^1[s,\mu](t,x)\mu(dx),
\end{align*}
with an analogous definition for $(\partial_t)(g(q,\Sigma^1_t[s,\mu]_{\#}\mu)).$
In preparation for Lemma \ref{lemma:diffofF} below, we are going to adopt this notation: for $t,t'\in (0,T),$ $q,q'\in\dtorus,$ $\mu,\nu\in\wasstwospace,$ and $0\leq \tau \leq 1,$ and $\gamma\in\Gamma(\mu,\nu),$
\begin{equation}\label{eq:interpols}
\left\{
\begin{aligned}
t^{\tau} = (1-\tau)t+ \tau t', \quad q^{\tau} = & (1-\tau)q+\tau q', \quad \mu^{\tau} = ((1-\tau)\pi^1+\tau\pi^2)_{\#}\gamma, \\
\sigma'_{t'}=\Sigma^1[s,\nu](t',q'), & \quad \sigma_t = \Sigma^1[s,\mu](t,q), \\ 
\sigma^{\tau} = & \Sigma^1_{t^{\tau}}[s,\mu^{\tau}]_{\#}\mu^{\tau}.
\end{aligned}
\right.
\end{equation}
Recall that, by definition, $X = (\Sigma^1)^{-1}.$ In this context, we are going to need the following:
\begin{prop}\label{prop:vvfinterpol}
Let $\mu,\nu\in\wasstwospace,$ $q,q'\in\dtorus,$ $t,t'\in (0,T),$ and let the notation (\ref{eq:interpols}) be in force. Let $w^{\tau}$ be the velocity vector field of the geodesic $\mu^{\tau}.$ \begin{enumerate}[(i)] \item The vector field
\begin{align}
v^{\tau}(y) =\  &(t'-t)v_{t^{\tau}}[s,\mu^{\tau}](y)  +  \int_{\dtorus} \bar\nabla_{\mu}\Sigma^1_{t^{\tau}}[s,\mu^{\tau}](X_{t^{\tau}}[s,\mu^{\tau}](y))(x)w^{\tau}(x)\mu^{\tau}(dx) \notag \\ & - \nabla_q\Sigma^1_{t^{\tau}}[s,\mu^{\tau}](X_{t^{\tau}}[s,\mu^{\tau}](y))w^{\tau}(X_{t^{\tau}}[s,\mu^{\tau}](y)) , \quad y \in \dtorus, \label{eq:vtau}
\end{align} 
is a velocity vector field for the path $\sigma^{\tau}.$ 
\item 
$$ 
\|v^{\tau}\|_{L^2(\sigma^{\tau})}  \leq C (|t'-t| + \mathscr{W}(\mu,\nu)), \quad 0\leq \tau \leq 1
$$
for some constant $C.$ 
\end{enumerate}
\end{prop}
\begin{proof}
\textit{(i)} Fix any $\gamma\in\Gamma_0(\mu,\nu)$ and for $0\leq \tau\leq 1,$ put $\mu^{\tau}$ as in (\ref{eq:interpols}). By definition, $v^{\tau}$ must satisfy 
\begin{equation}\label{eq:leftright}
\frac{d}{d\tau} \int_{\dtorus} \varphi(y) \sigma^{\tau}(dy) = \int_{\dtorus} \nabla\varphi(y) \cdot v^{\tau}(y) \sigma^{\tau} (dy)
\end{equation}
for every $\varphi\in C^{\infty}(\dtorus)$ and for $\mathscr{L}^1$-a.e.~$t\in (0,1).$ 
Computing the derivative on the left hand side,
\begin{align*} 
&\ \qquad \frac{d}{d\tau} \int_{\dtorus} \varphi(\Sigma^1[s,\mu^{\tau}](t^{\tau},y))\mu^{\tau}(dy) =  \\ & =\int_{\dtorus} \frac{d}{d\tau} \varphi(\Sigma^1[s,\mu^{\tau}](t^{\tau},y)) \mu^{\tau}(dy) - \int_{\dtorus} (\nabla_y) [\varphi(\Sigma^1[s,\mu^{\tau}](t^{\tau},y))]w^{\tau}(y)\mu^{\tau} (dy)  ,   
\end{align*} 
 where $w^{\tau}$ is the velocity vector field of the geodesic $\mu^{\tau},$ and we have again used the definition of velocity, since $\varphi\circ\Sigma^1$ is $C^{\infty}$ in $y.$ Doing the differentiation with respect to $\tau$ and $y,$ we get 
 \begin{align*}  
 & \hspace{3cm} \frac{d}{d\tau} \int_{\dtorus}\varphi(y)\sigma^{\tau}(dy) =  \\ &  = \int_{\dtorus} \big[ \nabla\varphi (\Sigma^1[s,\mu^{\tau}](t^{\tau},y)\cdot  \bigg( \partial_t \Sigma^1[s,\mu^{\tau}](t^{\tau},y)(t'-t)  + \int_{\dtorus}\bar\nabla_{\mu}\Sigma^1[s,\mu^{\tau}](t^{\tau},y)(x) w^{\tau}(x)\mu^{\tau}(dx)\bigg)\big]\mu^{\tau}(dy) \\ & \quad   - \int_{\dtorus} \nabla\varphi (\Sigma^1[s,\mu^{\tau}](t^{\tau},y))\cdot [\nabla_q \Sigma^1[s,\mu^{\tau}](t^{\tau},y) w^{\tau}(y)]\mu^{\tau}(dy).
\end{align*} 
Since $\Sigma^1_{t^{\tau}}[s,\mu^{\tau}]_{\#}\mu^{\tau} = \sigma^{\tau}$ by definition, we obtain (\ref{eq:vtau}) after writing the latter expression as an integral with respect to $\sigma^{\tau}$ and comparing against the right hand side of (\ref{eq:leftright}).

\textit{(ii)} This follows by Remark \ref{remark:littlefact}(i) and the boundedness of $\partial_t\Sigma^1,$ $\nabla_q\Sigma^1,$ $\bar\nabla_{\mu}\Sigma^1$.
\end{proof}
\begin{remark}\label{remark:vtau2}
It follows that, evaluated at $\Sigma^1_{t^{\tau}}[s,\mu^{\tau}](y),$ $v^{\tau}$ has the simpler expression:
\begin{align}
v^{\tau}(\Sigma^1_{t^{\tau}}[s,\mu^{\tau}](y)) = &\ (t'-t)\partial_t\Sigma^1_{t^{\tau}}[s,\mu^{\tau}](y) + \int_{\dtorus}\bar\nabla_{\mu}\Sigma^1_{t^{\tau}}[s,\mu^{\tau}](y)(x)w^{\tau}(x)\mu^{\tau}(dx)  \notag \\
 & \ \ -\nabla_q\Sigma^1_{t^{\tau}}[s,\mu^{\tau}](y)w^{\tau}(y), \qquad y\in\dtorus. \qquad \qquad \qquad \quad    \qquad \sslash \notag
\end{align}
\end{remark}
\begin{lemma}
\label{lemma:diffofF} 
Let $t,t'\in(0,T),$ $\mu,\nu\in\wasstwospace,$ $q,q'\in\dtorus$ be arbitrary, and put $\sigma_t = \Sigma^1_t[s,\mu]_{\#}\mu,$ $\sigma'_{t'} = \Sigma^1_{t'}[s,\nu]_{\#}\nu.$ Then there exists a constant $C$ such that \begin{align}  \big| & F(q',\sigma'_{t'}) - F(q,\sigma_t) - \nabla_q F(q,\sigma_t)\cdot (q'-q) \notag \\ &  - \int_{\dtorus}\nabla_{\mu}F(q,\sigma_t)(\Sigma^1_t[s,\mu](x))\cdot \partial_t \Sigma^1[s,\mu](t,x)\mu(dx)(t'-t) \notag \\  & - \int_{\dtorus\times\dtorus} \mathcal{N}^F_t[s,\mu](q)(x)\cdot (y-x)\gamma(dx,dy) \big| \notag  \\ \leq & \    C(|t'-t|^2 + |q'-q|^2 + \mathscr{W}^2(\mu,\nu)), \notag \end{align}
for any $\gamma\in\Gamma_0(\mu,\nu).$ The functions $\mathcal{N}^F_t[s,\mu](q)(x)$ and $(\partial_t)(F(q,\Sigma_t^1[s,\mu]_{\#}\mu)$ are continuous in all its variables.
\end{lemma} 
Naturally, the same result holds for $g$ in place of $F.$
\begin{proof}
The assertion about continuity follows immediately from the continuity of the functions that enter in the definition of $\mathcal{N}^F_t[s,\mu](q)(x)$ and $(\partial_t)(F(q,\Sigma_t^1[s,\mu]_{\#}\mu)$.

Let $\gamma\in\Gamma_0(\mu,\nu),$ and define $\mu^{\tau}$, $0\leq \tau\leq 1$ as in Remark \ref{remark:littlefact}. Let the notation (\ref{eq:interpols}) be in effect, so that $\tau\mapsto\sigma^{\tau}$ is a continuous path joining $\sigma_t$ with $\sigma'_{t'}$ and Proposition \ref{prop:vvfinterpol} holds, with $w^{\tau}$ as defined therein. Denote by $E$ the expression inside the bars on the left-hand side of the inequality of the lemma.

\textit{\underline{Step 1.}} 

\textit{Claim 1.} 
\begin{align*}
\int_{\dtorus}\nabla_{\mu}F(q,\sigma_t)(x)\cdot v^0(x)\sigma_t(dx) = &\ \int_{\dtorus}\nabla_{\mu}F(q,\sigma_t)(\Sigma^1_t[s,\mu](x))\cdot \partial_t\Sigma^1_t[s,\mu](x)\mu(dx)(t'-t) \\ 
& \  +\int_{\dtorus\times\dtorus}\mathcal{N}^F_t[s,\mu](q)(x)(y-x)\gamma(dx,dy).
\end{align*}

\textit{Proof of Claim 1.} Using Remark \ref{remark:vtau2} with $\tau = 0,$
\begin{align}
& \ \qquad \int_{\dtorus}\nabla_{\mu}F(q,\sigma_t)(x)\cdot v^0(x)\sigma_t(dx) \notag \\
  = &\  \int_{\dtorus}\nabla_{\mu}F(q,\sigma_t)(\Sigma^1_t[s,\mu](x))\cdot \partial_t \Sigma^1_{t}[s,\mu](y)\mu(dx)(t'-t) \notag
  \\ 
  & - \int_{\dtorus}\nabla_{\mu}F(q,\sigma_t)(\Sigma^1_t[s,\mu](x))\cdot \big[\nabla_q\Sigma^1_t[s,\mu](x)w^{0}(x)  -  \int_{\dtorus} \bar\nabla_{\mu}\Sigma^1_t[s,\mu](x)(b)w^0(b)\mu(db)  \big]\mu(dx). \label{eq:beforefubini}  
\end{align}
Note that
\begin{align*}
 & \int_{\dtorus}\nabla_{\mu}F(q,\sigma_t)(\Sigma_t^1[s,\mu](x))\cdot \big[\int_{\dtorus}\bar\nabla_{\mu}\Sigma_t^1[s,\mu](x)(b)w^0(b)\mu(db)\big]\mu (dx) \\ = & \ \int_{\dtorus}\int_{\dtorus}\nabla_{\mu}F(q,\sigma_t)(\Sigma_t^1[s,\mu](x))\cdot \bar\nabla_{\mu}\Sigma_t^1[s,\mu](x)(b)\mu(dx) w^0(b)\mu(db) \\ 
  = &  \ \int_{\dtorus}\int_{\dtorus}\nabla_{\mu}F(q,\sigma_t)(\Sigma_t^1[s,\mu](r))\cdot \bar\nabla_{\mu}\Sigma_t^1[s,\mu](r)(x)\mu(dr) w^0(x)\mu(dx).
\end{align*}
Substituting this identity into (\ref{eq:beforefubini}), we see that 
\begin{align}
& \ \qquad \int_{\dtorus}\nabla_{\mu}F(q,\sigma_t)(x)\cdot v^0(x)\sigma_t(dx)  \notag \\
  = &\  \int_{\dtorus}\nabla_{\mu}F(q,\sigma_t)(\Sigma^1_t[s,\mu](x))\cdot \partial_t \Sigma^1_{t}[s,\mu](y)\mu(dx)(t'-t) + \int_{\dtorus}\mathcal{N}^F_t[s,\mu](q)(x) \cdot w^0(x) \mu(dx) \notag 
\end{align}
Finally, we use Remark \ref{remark:littlefact}(ii) for the last integral in the latter expression, and the claim is proved.

\textit{Claim 2.} 
\begin{align*}
E = F(q',\sigma'_{t'}) - F(q,\sigma_t) - \nabla_q F(q,\sigma_t)\cdot (q'-q) - \int_{\dtorus}\nabla_{\mu}F(q,\sigma_t)(x)\cdot v^0(x) \sigma_t(dx). 
\end{align*}

\textit{Proof of Claim 2.} This follows immediately from Claim 1 and the definition of $E$. 

Thus, the lemma will be proved if we show that
\begin{align}
& \ |F(q',\sigma'_{t'}) - F(q,\sigma_t) - \nabla_q F(q,\sigma_t)\cdot (q'-q) - \int_{\dtorus}\nabla_{\mu}F(q,\sigma_t)(x)\cdot v^0(x) \sigma_t(dx)|  \notag \\ 
\leq & \ C(|t'-t|^2+|q'-q|^2 +\mathscr{W}^2(\mu,\nu)) \label{eq:wanttoshow2} 
\end{align}
for some constant $C.$ 

\textit{\underline{Step 2.}} Before we set out to prove (\ref{eq:wanttoshow2}), we transform the expression for $E$ some more. By the chain rule, we have that 
\begin{align} 
& F(q',\sigma'_{t'}) - F(q,\sigma_t) = \int_0^1 \big[ \nabla_qF(q^{\tau},\sigma^{\tau})\cdot (q'-q) + \nabla_{\mu}F(q^{\tau},\sigma^{\tau})(x)\cdot v^{\tau}(x)\sigma^{\tau}(dx)  \big] d\tau ,  \notag    
\end{align}
so
\begin{align}
E = &  \int_0^1 \bigg( [ \nabla_q F(q^{\tau},\sigma^{\tau}) - \nabla_q F(q,\sigma_t)] \cdot (q'-q) \notag \\ & \qquad + \int_{\dtorus}\nabla_{\mu}F(q^{\tau},\sigma^{\tau})\cdot v^{\tau}(x)\sigma^{\tau}(dx) - \int_{\dtorus}\nabla_{\mu}F(q,\sigma_t)\cdot v^{0}(x)\sigma_t(dx)  \bigg) d\tau. \notag
\end{align}
With our knowledge of $v^{\tau},$ that is, from Remark \ref{remark:vtau2}), we may rewrite $E$ as 
\begin{align*}
E = \int_0^1 \bigg[ (\nabla_q & F(q^{\tau},\sigma^{\tau}) - \nabla_q F(q,\sigma_t)) \cdot (q'-q) \\ & + \int_{\dtorus} \nabla_{\mu}F(q^{\tau},\sigma^{\tau})(\Sigma^1_{t^{\tau}}[s,\mu^{\tau}](x))\cdot \bigg( \partial_t\Sigma^1_{t^{\tau}}[s,\mu^{\tau}](x)(t'-t) \notag \\ & \qquad \qquad + \int_{\dtorus} \bar\nabla_{\mu}\Sigma^1_{t^{\tau}}[s,\mu^{\tau}](x)(r)w^{\tau}(r)\mu^{\tau}(dr) - \nabla_q \Sigma^1_{t^{\tau}}[s,\mu^{\tau}](x)w^{\tau}(x)\bigg) \mu^{\tau}(dx) \notag \\  &- \int_{\dtorus} \nabla_{\mu}F(q,\sigma_t)(\Sigma_t^1[s,\mu](x))\cdot \bigg(\partial_t \Sigma_t^1[s,\mu](x)(t'-t) \notag \\ & \qquad \qquad + \int_{\dtorus}\bar\nabla_{\mu}\Sigma^1_t[s,\mu](x)(r)w^0(r)\mu(dr)  - \nabla_q\Sigma^1_t[s,\mu](x)w^0(x) \bigg)\mu(dx) \bigg] d\tau.
\end{align*}
Now, let 
$$
\gamma^{\tau} \in \Gamma(\mu,\mu^{\tau}), \quad 0<\tau<1.
$$ 
Then $E$ takes the form
\begin{align} 
& E =  \int_0^1  (\nabla_qF(q^{\tau},\sigma^{\tau})-\nabla_q F(q,\sigma_t))\cdot (q'-q) d\tau \notag \\ &\ \quad + \int_0^1  \int_{\dtorus\times \dtorus} \bigg[  \nabla_{\mu}F(q^{\tau},\sigma^{\tau})(\Sigma^1_{t^{\tau}}[s,\mu^{\tau}](y))\cdot \bigg( \partial_t\Sigma^1_{t^{\tau}}[s,\mu^{\tau}](y)(t'-t) \notag  \\ & \qquad \qquad \qquad \qquad + \int_{\dtorus} \bar\nabla_{\mu}\Sigma^1_{t^{\tau}}[s,\mu^{\tau}](y)(r)w^{\tau}(r)\mu^{\tau}(dr) - \nabla_q \Sigma^1_{t^{\tau}}[s,\mu^{\tau}](y)w^{\tau}(y)\bigg) \notag  \\ 
 & \qquad \qquad \ \  \qquad  -  \nabla_{\mu}F(q,\sigma_t)(\Sigma_t^1[s,\mu](x))\cdot \bigg(\partial_t \Sigma_t^1[s,\mu](x)(t'-t) \notag \\ & \ \qquad \qquad \qquad \qquad  + \int_{\dtorus}\bar\nabla_{\mu}\Sigma^1_t[s,\mu](x)(r)w^0(r)\mu(dr) - \nabla_q\Sigma^1_t[s,\mu](x)w^0(x) \bigg)\bigg] \gamma^{\tau}(dx,dy) d\tau. \notag \\
 & =: E_1 + E_2, \notag
\end{align}
where
\begin{equation*}
E_1 = \int_0^1 (\nabla_q  F(q^{\tau},\sigma^{\tau}) - \nabla_q F(q,\sigma_t)) \cdot (q'-q) d\tau, \quad E_2 = E - E_1.
\end{equation*}
\textit{\underline{Step 3} (estimates).} In the following, we will be making use of the boundedness of the quantities displayed in Section \ref{subsection:medata}.
With an abuse of notation, for each fixed $\tau\in(0,1],$ let $\tau(h) = h\tau,$ $0\leq h \leq 1,$ so that $t^{\tau(\cdot)}$ has endpoints $t,$ $t^{\tau},$ $q^{\tau(\cdot)}$ has endpoints $q,$ $q^{\tau}$ and $\sigma^{\tau(\cdot)}$ has endpoints $\sigma_t,$ $\sigma^{\tau}.$ The velocity vector field of the path $h \mapsto v^{\tau(h)}$ is $\tau v^{\tau(h)}$ (see, e.g., \cite{gradientflows}). 
An argument similar to the one in the proof of Proposition \ref{prop:timeder}, based on the continuity of $\nabla^2_{qq}F$ and $\nabla^2_{\mu q}F,$ shows that  
\begin{align*} 
&  \nabla_q F(q^{\tau},\sigma^{\tau}) - \nabla_q F(q,\sigma_t) = \\
  = &\ \int_0^1 \big[ \tau \nabla^2_{qq}F(q^{\tau(h)},\sigma^{\tau(h)})(q'-q) 
 +   \int_{\dtorus} \tau \nabla^2_{\mu q}  F(q^{\tau(h)},\sigma^{\tau(h)})(x) v^{\tau(h)}(x) \sigma^{\tau(h)}(dx) \big] dh.
\end{align*}
Therefore, by Proposition (\ref{prop:vvfinterpol})(ii), we get \begin{equation}\label{eq:C1} |E_1| \leq C|q'-q|\big(|q'-q|+|t'-t|+\mathscr{W}(\mu,\nu)\big)\end{equation}
for some positive constant $C$. 
We break $E_2$ down as follows:
\begin{equation*} E_2 = \int_0^1 \int_{\dtorus\times\dtorus} ( B_1 + B_2 + B_3 ) \gamma^{\tau} (dx,dy) d\tau , \end{equation*}
where
\begin{align*}
B_1  := & \  \big[ \nabla_{\mu}F(q^{\tau},\sigma^{\tau})(\Sigma^1_{t^{\tau}}[s,\mu^{\tau}](y)) - \nabla_{\mu} F(q,\sigma_t)(\Sigma_t^1[s,\mu](y)) \big] \cdot \\ &  \ \cdot \big[ \partial_t\Sigma^1_{t^{\tau}}[s,\mu^{\tau}](y)(t'-t) - \nabla_q\Sigma^1_{t^{\tau}}[s,\mu^{\tau}](y)w^{\tau}(y) + \int_{\dtorus}\bar\nabla_{\mu}\Sigma^1_{t^{\tau}}[s,\mu^{\tau}](y)(r)w^{\tau}(r)\mu^{\tau}(dr) \big] , 
\end{align*}
\begin{align*}
B_2 := \ \nabla_{\mu}F(q,\sigma_t)(\Sigma^1_t[s,\mu](y))\cdot \big[ & \partial_t\Sigma^1_{t^{\tau}}[s,\mu^{\tau}](y)(t'-t) - \nabla_q \Sigma^1_{t^{\tau}}[s,\mu^{\tau}](y)w^{\tau}(y) \\ 
& \qquad + \int_{\dtorus}\bar\nabla_{\mu}\Sigma^1_{t^{\tau}}[s,\mu^{\tau}](y)(r)w^{\tau}(r)\mu^{\tau}(dr) \\ &  - \partial_t\Sigma^1_{t}[s,\mu](x)(t'-t) + \nabla_q \Sigma^1_{t}[s,\mu](x)w^{0}(x) \\ 
& \qquad - \int_{\dtorus}\bar\nabla_{\mu}\Sigma^1_{t}[s,\mu](x)(r)w^{0}(r)\mu(dr) \big] ,
\end{align*}
and
\begin{align*}
B_3 := & \  \big[ \nabla_{\mu}F(q,\sigma_t)(\Sigma^1_{t}[s,\mu](y)) - \nabla_{\mu} F(q,\sigma_t)(\Sigma_t^1[s,\mu](x) \big] \cdot \\ &  \ \cdot \big[ \partial_t\Sigma^1_{t}[s,\mu](x)(t'-t) - \nabla_q\Sigma^1_{t}[s,\mu](x)w^{0}(x) + \int_{\dtorus}\bar\nabla_{\mu}\Sigma^1_{t}[s,\mu](x)(r)w^{0}(r)\mu(dr) \big].
\end{align*}
To estimate $\int_0^1\int_{(\dtorus)^2}B_1\gamma^{\tau}d\tau,$ we address the first square bracket in the definition of $B_1.$ With $\tau w^{\tau(h)}$ being the velocity vector field of the path $h\mapsto w^{\tau(h)},$ we have
\begin{align*}
& \nabla_{\mu}F(q^{\tau},\sigma^{\tau})(\Sigma_{t^{\tau}}^1[s,\mu^{\tau}](y)) - \nabla_{\mu}F(q,\sigma_t)(\Sigma_{t}^1[s,\mu](y))  = 
\\ 
= & \  \int_0^1 \bigg[ \tau \nabla^2_{q\mu}F(q^{\tau(h)},\sigma^{\tau(h)})(\Sigma^1_{t^{\tau(h)}}[s,\mu^{\tau(h)}](y))(q'-q) \\ &\qquad + \int_{\dtorus}\tau \nabla^2_{\mu\mu}F(q^{\tau(h)},\sigma^{\tau(h)})(\Sigma^1_{t^{\tau(h)}}[s,\mu^{\tau(h)}](y))(b)v^{\tau(h)}(b)\sigma^{\tau(h)}(db) 
\\ 
& \qquad + \nabla^2_{x \mu}F(q^{\tau(h)},\sigma^{\tau(h)})(\Sigma^1_{t^{\tau(h)}}[s,\mu^{\tau(h)}](y))\big[ \tau(t'-t)\partial_t \Sigma^1_{t^{\tau(h)}}[s,\mu^{\tau(h)}](y) \\ & \hspace{7.4cm}  + \int_{\dtorus}\tau \bar\nabla_{\mu}\Sigma^1_{t^{\tau(h)}}[s,\mu^{\tau(h)}](y)(b) w^{\tau(h)}(b)\mu^{\tau(h)}(db) \big] \bigg] dh ,
\end{align*}
so
\begin{align}
& \ | \nabla_{\mu}F(q^{\tau},\sigma^{\tau})(\Sigma_{t^{\tau}}^1[s,\mu^{\tau}](y)) - \nabla_{\mu}F(q,\sigma_t)(\Sigma_{t}^1[s,\mu](y)) | \notag \\ \leq  \ &  C|q'-q| + C(|t'-t| + \mathscr{W}(\mu,\nu)) + C(|t'-t| + \mathscr{W}(\mu,\nu))  \notag
\end{align} 
and thus,
\begin{align}
 \nabla_{\mu}F(q^{\tau},\sigma^{\tau})(\Sigma_{t^{\tau}}^1[s,\mu^{\tau}](y)) - \nabla_{\mu}F(q,\sigma_t)(\Sigma_{t}^1[s,\mu](y)) | \leq C(|t'-t|+|q'-q|+\mathscr{W}(\mu,\nu)) \notag
 \end{align}
for some constant $C.$
Invoking the boundedness of $\partial_t\Sigma^1,$ $\nabla_q \Sigma^1,$ $\bar\nabla_{\mu}\Sigma^1$ and Remark \ref{remark:littlefact}(i), we get
\begin{align}
\big| \int_0^1 \int_{\dtorus\times\dtorus} B_1 \gamma^{\tau}(dx,dy)d\tau \big| \leq C(|t'-t|+|q'-q|+\mathscr{W}(\mu,\nu))(|t'-t|+\mathscr{W}(\mu,\nu)). \label{eq:fromB1}
\end{align}
Recall that we denote by $\nabla^2_{x\mu}F(q,\mu)(x)$ the gradient at $x$ of the mapping $x\mapsto \nabla_{\mu}F(q,\mu)(x),$ and $\nabla^2_{x\mu}F(q,\mu)(x)$ is uniformly bounded, by assumption. Thus, 
\begin{align*}
|\nabla_{\mu} F(q,\sigma_t)(\Sigma^1_{t}[s,\mu](y)) - \nabla_{\mu} F(q,\sigma_t)(\Sigma_t^1[s,\mu](x)| &\  \leq \|\nabla^2_{x\mu}F(q,\sigma_t)\|_{\infty}|\Sigma^1_{t}[s,\mu](y)-\Sigma^1_{t}[s,\mu](x)| \\ &\ \leq 2\|\nabla^2_{x\mu}F(q,\sigma_t)\|_{\infty}A_1|y-x|,   
\end{align*}
for any $x,y\in\dtorus.$ Therefore, for some constant $C,$ 
\begin{align}
\big| \int_0^1 \int_{\dtorus\times\dtorus} B_3 \gamma^{\tau}(dx,dy)d\tau \big| &\ \leq  C \int_0^1 \int_{\dtorus\times\dtorus} |x-y|(|t'-t|+\mathscr{W}(\mu,\nu))\gamma^{\tau}(dx,dy) d\tau \notag \\ &\ \leq C|x-y|(|t'-t|+\mathscr{W}(\mu,\nu)). \label{eq:fromB3} 
\end{align}
Next, we are going to estimate $\int_0^1\int_{(\dtorus)^2}B_2\gamma^{\tau}d\tau.$ To ease notation, let us make the abbreviations
\begin{displaymath}
\phi_2 = \Sigma^1_{t^{\tau}}[s,\mu^{\tau}], \quad \phi_1 = \Sigma^1_{t}[s,\mu], \quad \psi_1 = \nabla_{\mu}F(q,\sigma_t)\circ \phi_1.
\end{displaymath}
Then $\int_{(\dtorus)^2} B_2 \gamma^{\tau}(dx,dy)$ reads:
\begin{align}
\int_{\dtorus\times\dtorus}  \psi_1(y) \cdot \big[& \partial_t \phi_2(y)(t'-t) - \nabla_q \phi_2(y) w^{\tau}(y)
+ \int_{\dtorus} \nabla_{\mu}\phi_2(y)(r_2)w^{\tau}(r_2)\mu^{\tau}(dr_2) \notag \\ 
& \ -  \partial_t \phi_1(x)(t'-t) + \nabla_q \phi_1(x) w^{0}(x)
- \int_{\dtorus} \nabla_{\mu}\phi_1(x)(r_1)w^0(r_1)\mu(dr_1)  \big] \gamma^{\tau}(dx,dy). \notag
\end{align}
Therefore,
\begin{align}
\big| \int\limits_{\dtorus\times\dtorus} B_2 \gamma^{\tau}(dx,dy) \big|  
\leq \| \psi_1 \|_{L^{\infty}(\gamma^{\tau})} \| D \|_{L^1(\gamma^{\tau})}, \notag
\end{align}
where, by applying Remark \ref{remark:littlefact}(iii),
\begin{align}
D = & \ (\partial_t \phi_2(y)-\partial_t \phi_1(x))(t'-t) - (\nabla_q\phi_2(y)-\nabla_q\phi_1(x))\frac{y-x}{\tau}
 \notag \\
& \ + \int_{\dtorus\times\dtorus} (\nabla_{\mu}\phi_2(y)(b)-\nabla_{\mu}\phi_1(x)(a))\frac{b-a}{\tau} \gamma^{\tau}(da,db), \notag
\end{align}
and we are left with estimating the $L^1(\gamma^{\tau})$-norm of $D.$ We write
$$
\tau D = \tau D + (\phi_2(y) - \phi_1(x))  + (\phi_1(x) - \phi_2(y)),
$$
to apply (\ref{eq:munugammazero}), once with $\mu=\mu,$ $\nu=\mu^{\tau},$ $t=t, t' = t^{\tau},$ $q=x,$ $q'=y,$ then with $\mu=\mu^{\tau},$ $\nu = \mu,$ $t=t^{\tau},$ $t'=t,$ $q=y,$ $q'=x,$ and obtain:
$$
|\tau D| \leq 2C(\tau^2|t'-t|^2 + \tau^2 \mathscr{W}^2(\mu,\nu) + |x-y|^2).
$$ 
Therefore
\begin{align*}
\int_{\dtorus\times\dtorus} |D| \gamma^{\tau}(dx,dy) \leq &\ 2C\big( \tau |t'-t|^2 + \tau\mathscr{W}^2(\mu,\nu) + \frac{1}{\tau}\int_{\dtorus\times\dtorus}|x-y|^2\gamma^{\tau}(dx,dy) \big) 
\\ 
\leq &\ 2C\big( \tau |t'-t|^2 + \tau\mathscr{W}^2(\mu,\nu) + \frac{1}{\tau}\tau^2\mathscr{W}^2(\mu,\nu) \big).
\end{align*}
Consequently, for some constant $C,$ 
\begin{align}
\int_0^1\int_{\dtorus\times\dtorus} |B_2|\gamma^{\tau}(dx,dy)d\tau \leq C(|t'-t|^2 + \mathscr{W}^2(\mu,\nu)). \label{eq:fromB2}
\end{align}
\textit{\underline{Step 4.}} Note that all the estimates derived in the previous step, namely, (\ref{eq:C1}), (\ref{eq:fromB1}), (\ref{eq:fromB3}), (\ref{eq:fromB2}), are quadratic in the increments. Hence,
\begin{align*}
  | E | \leq  \  C(|t'-t|^2 + |q'-q|^2 + \mathscr{W}^2(\mu,\nu)), 
\end{align*}
which is (\ref{eq:wanttoshow2}), for some constant $C$. This concludes the proof.
\end{proof}
\subsection{Gradient of $u(s,q,\cdot)$ and chain rule}\label{subsection:gradientofuandchainrule} 
We collect now the results on differentiability in $\mu$ of the functions $g,$ $F,$ $H(q,\mathcal{V})$ that constitute the full value function $u,$ with the following definition and corollary.
Define the $\R^d$-valued function 
\begin{align} 
\Upsilon[s,\mu](q,y) := \mathcal{N}^g_0[s,\mu](q)(y) + \int_0^s [ \bar\nabla_{\mu}H(q,\mathcal{V}_t[s,\mu](q,y)) + \mathcal{N}^F_t[s,\mu](q)(y) ] dt ,  
\label{eq:defofupsilon}    
\end{align} 
where $s\in[0,T]$, $q\in\dtorus,$ $y\in\R^d,$ $\mu\in\wasstwospace.$ 
\begin{cor}
\label{cor:upsilonisdiffble} 
The function $\Upsilon,$ just defined, is continuous on $[0,T]\times\dtorus\times\dtorus\times\wasstwospace,$ and $u(s,q,\cdot),$ defined by (\ref{eq:defofu}), is differentiable on $\wasstwospace,$ in the sense that there exists a constant $C$ such that \begin{align*} |u(s,q,\nu)-u(s,q,\mu) - \int_{\dtorus\times\dtorus} \Upsilon[s,\mu](q,y)\cdot (x-y)\gamma(dy,dx) | \leq C \mathscr{W}^2(\mu,\nu)   \end{align*} for every $\mu, \nu\in\wasstwospace,$ $\gamma\in\Gamma_0(\mu,\nu),$ $s\in[0,T],$ $q\in\dtorus.$ 
\end{cor} 
\begin{proof} 
The continuity of $\Upsilon$ is a consequence of the continuity of its parts, and combining Lemma \ref{lemma:Hisdiffble} with Lemma \ref{lemma:diffofF} produces the stated estimate.
\end{proof} 
We refer back to Section \ref{subsection:differentiabilityinwasserstein} for the definition of $\mathscr{T}_{\mu}\wasstwospace.$ Since we do not know whether $\Upsilon[s,\mu](q,\cdot)$ belongs to the $L^2(\mu)$ closure of $\{\nabla\varphi \ | \ \varphi\in C_c^{\infty}(\dtorus)\},$ we make the following definition. \begin{defn}\label{defn:wassgradientofu} Let $u=u(s,q,\mu)$ be as in (\ref{eq:specificu}), for $s\in[0,T],$ $q\in\dtorus,$ $\mu\in\wasstwospace,$ and $\Upsilon$ be as in (\ref{eq:defofupsilon}). At every $s,q,\mu,$ by $$ \nabla_{\mu}u(s,q,\mu) $$ we will mean the projection of $\Upsilon[s,\mu](q,\cdot)$ onto $\mathscr{T}_{\mu}\wasstwospace.$    \end{defn} We need to note that the velocity vector fields $v[s,\mu](t,\cdot)$ are not necessarily elements of $\mathscr{T}_{\sigma_t}\wasstwospace,$ even though this is true in the case $H(q,p)=\frac{1}{2}|p|^2$ (see \cite[Theorem 5.1]{mfgmain}). This leads to the following definition. \begin{defn}\label{defn:projectionofv} At every $s,t\in[0,T],$ $\mu\in\wasstwospace,$  by $$ \bar v[s,\mu](t,\cdot)$$ we will mean the projection of $v[s,\mu](t,\cdot)$ onto $\mathscr{T}_{\sigma_t}\wasstwospace,$ where $\sigma_t=\Sigma^1_t[s,\mu](\cdot)_{\#}\mu.$     \end{defn} Note that, if $w\in L^2(\dtorus,\mu)$ is arbitrary and $\bar w$ is its projection onto $\mathscr{T}_{\mu}\wasstwospace,$ then 
\begin{equation}
\label{eq:doubleproj} 
\int_{\dtorus} \Upsilon[s,\mu](q,x)\cdot \bar w(x)\mu(dx) = \int_{\dtorus} \nabla_{\mu}u(s,q,\mu)(x)\cdot w(x)\mu(dx),
\end{equation} 
which follows from (\ref{eq:projfact}).
We are now ready to prove the third main statement of the paper: \begin{thm}\label{thm:masterequation1} 
Let $H,F,g$ be as in Sections \ref{subsection:dataformfg}, \ref{subsection:medata}, and $\Sigma$ be the unique solution to the system (\ref{eq:hamODEs}) obtained in Corollary \ref{cor:hamODEshassolution}(ii). Let $u=u(s,q,\mu)$ be defined as in (\ref{eq:defofu}). Then: 
\begin{enumerate}[(i)] 
\item For any $s\in[0,T],$ $\mu\in\wasstwospace,$ there exists $\sigma\in AC^2(0,T;\wasstwospace)$ such that $\sigma_s=\mu$ and the continuity equation 
\begin{equation*} 
\partial_t \sigma_t + \textrm{div}(\nabla_p H(q,\nabla_q u(t,q,\sigma_t)\sigma_t) = 0 \quad \textrm{ in } \mathcal D'((0,T)\times\wasstwospace) 
\end{equation*}
holds; \item The function $u$ is a classical solution to the master equation 
\begin{equation}
\tag{\ref{eq:ME}}  
\left\{ \begin{aligned}   \partial_s &u(s,q,\mu) + \int_{\dtorus} \nabla_{\mu} u(s,q,\mu)(x) \cdot \nabla_pH(x,\nabla_q u(s,x,\mu)) \mu(dx) \\ &\ \quad  + H(q,\nabla_q u(s,q,\mu)) + F(q,\mu) = 0  \  \qquad \qquad \textrm{ in } (0,T)\times\dtorus\times\wasstwospace , \\ &\ \hspace{4cm} u(0,q,\mu) =   g(q,\mu)   \qquad  \textrm{ on } \dtorus\times\wasstwospace , \end{aligned} \right. 
\end{equation}
 in the sense explained in Section \ref{subsection:defnofclassicalmfg}. 
\end{enumerate} 
\end{thm} 
The full value function $u=u(s,q,\mu)$ is the value of the solution $U$ of the MFG system's Hamilton-Jacobi equation (\ref{eq:mfg1}) at the time ($t=s$) at which the terminal condition $\sigma_{t=s}=\mu$ is prescribed for the continuity equation (\ref{eq:mfg2}).
\begin{proof} (i) Let $s\in[0,T],$ $\mu\in\wasstwospace.$ Set $\sigma_t := \Sigma_t^1[s,\mu]_{\#}\mu.$ Then the statement follows from Proposition \ref{prop:hmvisavelocity}, Corollary \ref{cor:cortozlemma}, formula (\ref{eq:observethat}) and Lemma \ref{lemma:dependsonbeingsymmetric}.

(ii) The regularity of $u$ in $q$ is the same as the regularity of $U$ in $q,$ which was discussed in Lemma \ref{lemma:zlemma}. Fix $0<s<T,$ $q\in\dtorus,$ $\mu\in\wasstwospace.$ As usual, $\sigma_t = \Sigma_t^1[s,\mu]_{\#}\mu,$ and $v_t = \partial_t\Sigma_t^1[s,\mu]\circ X_t[s,\mu],$ $0\leq t\leq T.$ Set $$ \hat\sigma_t := (id + (t-s)v_s)_{\#}\mu, \quad \bar\sigma_t := (id + (t-s)\bar v_s)_{\#}\mu , $$ where $\bar v_s$ is the projection of $v_s$ to $\mathscr{T}_{\mu}\wasstwospace.$ Through $\sigma_{s+h}\times\hat\sigma_{s+h},$ we estimate $\mathscr{W}(\sigma_{s+h},\hat\sigma_{s+h})$: \begin{align} \mathscr{W}^2(\sigma_{s+h},\hat\sigma_{s+h}) \leq  &\ \int\limits_{\dtorus\times\dtorus} |x-y|^2 (\sigma_{s+h}\times\hat\sigma_{s+h})(dx,dy)  \notag \\ = &\ \int\limits_{\dtorus} |\Sigma^1_{s+h}[s,\mu](y) - \Sigma^1_s[s,\mu](y) - h v_s[s,\mu](y)|^2 \notag \mu(dy).       \end{align} Note that $v_s[s,\mu](q) = \partial_t \Sigma^1_t[s,\mu](q) |_{t=s},$ since $X_s[s,\mu] = id.$ Therefore, \begin{equation}\label{eq:hfour} \mathscr{W}(\sigma_{s+h},\hat\sigma_{s+h}) \leq |h|^2 \|\partial^2_{tt}\Sigma^1\|^2_{\infty} .\end{equation} Let $$ \gamma_h := (id\times(id+hv_s))_{\#}\mu \ \in \ \Gamma(\mu,\hat\sigma_{s+h}). $$ Since, by definition, $\nabla_{\mu}u(s+h,q,\mu)\in\mathscr{T}_{\mu}\mathscr{P}(\dtorus),$ we apply Lemma \ref{lemma:keylemma} to write $$ |u(s+h,q,\hat\sigma_{s+h})-u(s+h,q,\mu)-\int\limits_{\dtorus\times\dtorus}\nabla_{\mu}u(s+h,q,\mu)(x)\cdot(y-x)\gamma_h(dx,dy) | = o(\|\pi^2-\pi^1\|_{\gamma_h}), $$ which is the same as 
\begin{equation}
\label{eq:littleogammah} 
|u(s+h,q,\hat\sigma_{s+h})-u(s+h,q,\mu)-h\int\limits_{\dtorus\times\dtorus}\nabla_{\mu}u(s+h,q,\mu)(x)\cdot v_s(x)\mu(dx) | = o(|h|),   
\end{equation} 
because $o(\|\pi^2-\pi^1\|_{\gamma_h})=o(|h|)$ as can be easily checked. Recall now (\ref{eq:doubleproj}). Formula (\ref{eq:defofupsilon}) shows that $\Upsilon[\cdot,\mu](q,y)$ is continuous, so there is a modulus of continuity $\omega$ such that 
\begin{align*} 
& \int_{\dtorus} \nabla_{\mu} u(s+h,q,\mu)(x)\cdot v_s(x) \mu(dx) = \int_{\dtorus}\Upsilon[s+h,\mu](q,x)\cdot \bar v_s(x)\mu(dx) \\ = &\ \int_{\dtorus}\Upsilon[s,\mu](q,x)\cdot \bar v_s(x) \mu(dx) + \omega(|h|) \\ = &\  \int_{\dtorus} \nabla_{\mu}u(s,q,\mu)(x)\cdot v_s(x) \mu(dx) + \omega(|h|). 
\end{align*} 
Therefore (\ref{eq:littleogammah}) is improved to 
\begin{align} 
|u&(s+h,q,\hat\sigma_{s+h})-u(s+h,q,\mu) -h\int_{\dtorus}\nabla_{\mu}u(s,q,\mu)(x)\cdot v_s(x)\mu(dx)| \notag \\ &\ = o(|h|) + |h|\omega(|h|). 
\label{eq:intermediatestep1} 
\end{align} 
Corollary \ref{cor:upsilonisdiffble} shows that $u(s,q,\cdot)$ is $\kappa_1$-Lipschitz for some constant $\kappa_1,$ because $[0,T]\times\dtorus$ is compact. Using the bound (\ref{eq:hfour}), we then have \begin{equation}\label{eq:intermediatestep2} |u(s+h,q,\hat\sigma_{s+h})-u(s+h,q,\sigma_{s+h})|\leq \kappa_1 h^2 \|\partial^2_{tt}\Sigma^1\|^2_{\infty}.   \end{equation} Invoking (\ref{eq:totalofu}), we write \begin{equation}\label{eq:intermediatestep3} | u(s+h,q,\sigma_{s+h})-u(s,q,\mu) +h\big[ H(q,\nabla_q u(s,q,\sigma_s))+F(q,\sigma_s)\big]| = o(|h|).   \end{equation} Finally, (\ref{eq:intermediatestep1}), (\ref{eq:intermediatestep2}) and (\ref{eq:intermediatestep3}) are needed to obtain: \begin{align}  | u(s+h,q,\mu)& - u(s,q,\mu) + h \int_{\dtorus}\nabla_{\mu}u(s,q,\mu)(x)\cdot v_s(x)\mu(dx) + h \big[ H(q,\nabla_q u(s,q,\sigma_s))+F(q,\sigma_s)\big] |   \notag \\ =  &\ |u(s+h,q,\mu) - u(s+h,q,\hat\sigma_{s+h}) + h \int_{\dtorus}\nabla_{\mu}u(s,q,\mu)(x)\cdot v_s(x)\mu(dx) \notag \\ &\ \ +  u(s+h,q,\hat\sigma_{s+h}) - u(s+h,q,\sigma_{s+h}) \notag \\ &\ \  + u(s+h,q,\sigma_{s+h}) - u(s,q,\mu) +  h \big[ H(q,\nabla_q u(s,q,\sigma_s))+F(q,\sigma_s)\big] | \notag  \\ = &\ o(|h|) + |h|\omega(|h|) + \kappa_1 h^2 \|\partial^2_{tt}\Sigma^1\|^2_{\infty} + o(|h|) = o(|h|). \end{align} We divide by $h$, remember that $v_s(x)=\nabla_p H(x,\nabla_q u(s,x,\mu)),$ $\mu=\sigma_s$ and let $h\to 0$ to obtain \begin{equation*} - \partial_s u(s,q,\mu) = \int_{\dtorus}\nabla_{\mu}u(s,q,\mu)(x)\cdot \nabla_p H(x,\nabla_q u(s,x,\mu))\mu(dx) + H(q,\nabla_q u(s,q,\mu)) + F(q,\mu) .    \end{equation*} Let us check the continuity of $s\mapsto \partial_s u(s,q,\mu).$ Due to (\ref{eq:nablaquequalsmathcalV}), $\nabla_q u(s,q,\mu) = \mathcal{V}[s,\mu](s,q) = \Sigma^2[s,\mu](s,q),$ which is continuous in $s,$ and the continuity of $H$ and $F$ takes care of the non-integral term in the formula for $\partial_s u.$ For the integral term, we use once again (\ref{eq:doubleproj}). Let $s'\in (0,T).$ Then \begin{align*} |\int_{\dtorus}&\Upsilon[s,\mu](q,x)\cdot \bar v_s(x)\mu(dx) - \int_{\dtorus}\Upsilon[s',\mu](q,x)\cdot\bar v_{s'}(x)\mu(dx)| \\ \leq & \ |\int_{\dtorus}\Upsilon[s,\mu](q,x)\cdot \bar v_s(x)\mu(dx) -\int_{\dtorus}\Upsilon[s,\mu](q,x)\cdot \bar v_{s'}(x)\mu(dx) | \\ & \ + |  \int_{\dtorus}\Upsilon[s,\mu](q,x)\cdot \bar v_{s'}(x)\mu(dx) - \int_{\dtorus}\Upsilon[s',\mu](q,x)\cdot \bar v_{s'}(x)\mu(dx) | \\  \leq &\ \|\Upsilon[s,\mu](q,\cdot)\|_{L^2(\mu)}\|\bar v_s-\bar v_{s'}\|_{L^2(\mu)} + \|\Upsilon[s,\mu](q,\cdot) - \Upsilon[s',\mu](q,\cdot)\|_{L^2(\mu)}\|\bar v_{s'}\|_{L^2(\mu)}.    \end{align*} By the fact that $\bar v_s,$ $\bar v_{s'}$ are the projections of $v_s,$ $v_{s'}$ on a subspace of $L^2(\mu),$ we know that $\|\bar v_s-\bar v_{s'}\|_{L^2(\mu)} \leq\| v_s- v_{s'}\|_{L^2(\mu)}.$ Letting $s'\to s$ we conclude the continuity. The continuity of $\partial_s u(s,\cdot,\mu)$ is treated in the same fashion, since $v[s,\mu](s,\cdot)$ is continuous. This completes the proof. 
\end{proof}
\begin{remark}
\label{remark:finalremark} 
We do not claim that the function $\nabla_{\mu}u(s,q,\cdot)$ is continuous on $\wasstwospace,$ which is true in the case \cite{mfgmain} of $H(q,p)=\frac{1}{2}|p|^2.$ The reason is that we have had to define $\nabla_{\mu} u$ as the projection of a vector field (Definition \ref{defn:wassgradientofu}) that, in general, is not in the tangent space $\mathscr{T}_{\mu}\wasstwospace,$ whereas for the quadratic Hamiltonian, $\Upsilon[s,\mu](q,\cdot)$ and $\nabla_{\mu}u(s,q,\mu)$ are the same \cite{majorga}. $\sslash$ 
\end{remark}
\section{Appendix}
\textbf{Proof of Lemma \ref{lemma:keylemma}}.
We begin by noting the following.~Let $\mu, \nu \in \wasstwospace,$ and $\gamma,\bar\gamma\in\Gamma(\mu,\nu).$ Then, for any $\varphi\in C^2(\dtorus),$ \begin{equation}\label{eq:ancillarywg1} \big| \int_{\R^d\times\R^d} \nabla\varphi(x) \cdot (y-x) (\gamma-\bar\gamma) (dx,dy) \big| \leq \frac{ \|\pi^1-\pi^2\|_{\gamma}^2 +  \|\pi^1-\pi^2\|_{\bar\gamma}^2  }{2} \|\nabla^2\varphi\|_{\infty}. \end{equation} Indeed, Taylor expansion gives a Borel function $r:\dtorus\times\dtorus\to[-1,1]$ such that $$ \varphi(y)-\varphi(x) - \nabla\varphi(x)\cdot (y-x) = r(x,y) \|\nabla^2\varphi\|_{\infty}\frac{|x-y|^2}{2}. $$ Integrating both sides of this equality over $\R^d\times\R^d$ once with respect to $\gamma$ and then with respect to $\gamma',$ remembering that $\gamma$ and $\gamma'$ have the same marginals $\mu$ and $\nu,$ and substracting one of the resulting expressions from the other, yields (\ref{eq:ancillarywg1}). Fix now $\mu\in\wasstwospace.$ Let $\nu\in\wasstwospace$ and $\gamma\in\Gamma_0(\mu,\nu),$ $\bar\gamma\in\Gamma(\mu,\nu).$ Let $\varphi\in C^{\infty}(\dtorus).$ Write $$ e(\nu,\nabla_{\mu}\mathcal{W}(\mu),\gamma) = e(\nu,\nabla\varphi,\gamma) - \int_{\R^d\times\R^d} (\nabla_{\mu}\mathcal{W}(\mu)(x)-\nabla\varphi(x))\cdot(y-x)\gamma(dx,dy) , $$ and the same expression which holds with $\bar\gamma$ in place of $\gamma.$ Substracting one from the other and taking absolute value gives, using H\"older's inequality, \begin{align*} |e(\nu,\nabla_{\mu}\mathcal{W}(\mu),\gamma)-e(\nu,\nabla_{\mu}\mathcal{W}(\mu),\bar\gamma) | \leq &\  |e(\nu,\nabla\varphi,\gamma) - e(\nu,\nabla\varphi,\bar\gamma) | \\  &\ + \|\nabla_{\mu}\mathcal{W}(\mu)-\nabla\varphi\|_{L^2(\mu)} (\|\pi^2-\pi^1\|_{\gamma} + \|\pi^2-\pi^1\|_{\bar\gamma} ).      \end{align*} Now, $\|\pi^2-\pi^1\|_{\gamma} \leq \| \pi^2-\pi^1 \|_{\bar\gamma},$ and, using (\ref{eq:ancillarywg1}),  \begin{align*} |e(\nu,\nabla_{\mu}\mathcal{W}(\mu),\gamma)-e(\nu,\nabla_{\mu}\mathcal{W}(\mu),\bar\gamma) | \leq  \|\pi^1-\pi^2\|_{\bar\gamma}\big(\|\pi^1-\pi^2\|_{\bar\gamma}\|\nabla^2\varphi\|_{\infty} + 2 \|\nabla_{\mu}\mathcal{W}(\mu)-\nabla\varphi\|_{L^2(\mu)}\big).\end{align*} Dividing by $\|\pi^2-\pi^1\|_{\bar\gamma}$ and once again because $\|\pi^2-\pi^1\|_{\gamma} \leq \| \pi^2-\pi^1 \|_{\bar\gamma},$ we obtain \begin{align*} 
\big|\frac{e(\nu,\nabla_{\mu}\mathcal{W}(\mu),\bar\gamma)}{\|\pi^1-\pi^2\|_{\bar\gamma}} \big| 
\leq \big|\frac{e(\nu,\nabla_{\mu}\mathcal{W}(\mu),\gamma)}{\|\pi^1-\pi^2\|_{\gamma}}\big| 
+ \big(\|\pi^1-\pi^2\|_{\bar\gamma}\|\nabla^2\varphi\|_{\infty} + 2 \|\nabla_{\mu}\mathcal{W}(\mu)-\nabla\varphi\|_{L^2(\mu)}\big) . 
\end{align*}  
This holds for any $\nu\in\wasstwospace,$ $\gamma\in\Gamma_0(\mu,\nu),$ $\bar\gamma\in\Gamma(\mu,\nu),$ $\varphi\in C^{\infty}(\dtorus).$ Fix $r>0$, and, on the right-hand side, fix $\bar\gamma\in\Gamma(\mu,\nu)$ such that $\|\pi^2-\pi^1\|_{\bar\gamma}<r.$ Take then the supremum on the left-hand side over $\nu\in\wasstwospace,$ $\bar\gamma\in\Gamma(\mu,\nu)$ such that $\|\pi^2-\pi^1\|_{\bar\gamma} < r$, to obtain $$  e[\nabla_{\mu}\mathcal{W}(\mu),r] 
\leq \big|\frac{e(\nu,\nabla_{\mu}\mathcal{W}(\mu),\gamma)}{\|\pi^1-\pi^2\|_{\gamma}}\big| 
+ \big(r \|\nabla^2\varphi\|_{\infty} + 2 \|\nabla_{\mu}\mathcal{W}(\mu)-\nabla\varphi\|_{L^2(\mu)}\big)  $$ holding for any $\nu\in\wasstwospace,$ $\gamma\in\Gamma_0(\mu,\nu), $ $\varphi\in C^{\infty}(\dtorus).$ Taking now the supremum on the right-hand side over $\nu\in\wasstwospace,$ $\gamma\in\Gamma_0(\mu,\nu)$ such that $\|\pi^2-\pi^1\|_{\gamma} < r,$ and then letting $r\to 0^+$ on both sides yields $$ \lim_{r\to 0^+} e[\nabla_{\mu}\mathcal{W}(\mu),r] \leq \lim_{r\to 0^+} e_0[\nabla_{\mu}\mathcal{W}(\mu),r] + 2 \|\nabla_{\mu}\mathcal{W}(\mu)-\nabla\varphi\|_{L^2(\mu)} = 2 \|\nabla_{\mu}\mathcal{W}(\mu)-\nabla\varphi\|_{L^2(\mu)},  $$ by the hypothesis, for any $\varphi\in C^{\infty}(\dtorus).$  By the fact that $\nabla_{\mu}\mathcal{W}(\mu)$ is an $L^2(\mu)$ limit of gradients of smooth periodic functions $\varphi,$ the conclusion follows. \null\hfill\qedsymbol \\

\noindent \textbf{Proof of Proposition \ref{prop:timeder}}.
\textit{(i)} Let us invoke Proposition 8.4.6 of \cite{gradientflows}, to say that there exists a subset $J\in I$ whose measure equals that of $I,$ such that, for every $V\in C(\dtorus\times\dtorus),$ $h_0 \in J$, we have
\begin{equation}\label{eq:846gflows}
\lim\limits_{h\to 0}\int\limits_{\dtorus\times\dtorus}V\big(x,\frac{y-x}{h}\big)\gamma_h(dx,dy) = \int_{\dtorus}
V(x,\bar{v}_{h_0}(x))\mu_{h_0}(dx),
\end{equation}
where $\bar{v}_{h_0}$ is the velocity vector field of minimal norm for $\mu_h$ at $h_0$, and $\{\gamma_h\}_{|h|>0}$ are optimal plans between $\mu_{h_0}$ and $\mu_{h_0+h}.$ Let then $h_0\in J.$ For $h$ such that $h_0+h\in I,$ let $\gamma_h\in\Gamma_0(\mu_{h_0},\mu_{h_0+h}).$ By the twice differentiability of $V,$ we have
\begin{align}
& \big| \nabla_{\mu}V(q^{h_0+h},\mu_{h_0+h})(x^{h_0+h}) - \nabla_{\mu}V(q^{h_0},\mu_{h_0})(x^{h_0})   \notag
\\ 
&\ \qquad - \nabla^2_{q\mu}V(q^{h_0},\mu_{h_0})(x^{h_0})(q^{h}-q^{h_0}) - P_{\gamma}[\mu_{h_0}](q^{h_0},x^{h_0},x^{h_0+h}) \big| \notag
\\
 \leq & \  o(|q^{h_0+h}-q^{h_0}|) \notag
\\
& \quad \qquad  + \big(\mathscr{W}(\mu_{h_0},\mu_{h_0+h})+|x^{h_0+h}-x^{h_0}|\big)\big(\rho(\mathscr{W}(\mu_{h_0},\mu_{h_0+h}))+
\epsilon(|x^{h_0+h}-x^{h_0}|)\big). \label{eq:twicediffproof}
\end{align}
Let $\bar v_{h_0}$ be the projection of $v_{h_0}$ onto $\mathscr{T}_{\mu_{h_0}}\wasstwospace.$ Since 
$$ 
\nabla^2_{\mu\mu}V(q^{h_0},\mu_{h_0})(x^{h_0},\cdot)\in \mathscr{T}_{\mu_{h_0}}\wasstwospace,
$$
(\ref{eq:projfact}) and (\ref{eq:846gflows}) give us
\begin{equation*}
\lim\limits_{h\to 0}\int_{\dtorus\times\dtorus}\nabla^2_{\mu\mu}V(q^{h_0},\mu_{h_0})(x^{h_0},r)\frac{b-r}{h}\gamma_h(dr,dy)  = \int_{\dtorus} \nabla^2_{\mu\mu}V(q^{h_0},\mu_{h_0})(x^{h_0},r)v_{h_0}(r)\mu_{h_0}(dr).
\end{equation*}
Therefore, dividing both sides of inequality (\ref{eq:twicediffproof}) by $h,$ and passing to the limit as $h\to 0,$ we obtain the desired formula.

\textit{(ii)} Under those conditions, the formula for $\frac{d}{dh}\nabla_{\mu}V(q^h,\mu_h)(x^h)$ is continuous in $h$, and the claim follows. \null\hfill\qedsymbol

\section{Acknowledgements} 
This research was partially supported by AFOSR MURI FA9550-18-1-0502.

The work presented here would not have been possible at all without the careful, patient and diligent advise of my advisors, Professors Wilfrid Gangbo and Andrzej \'Swi\k{e}ch, who taught me the main ideas and methods involved in it, and provided me with frequent and valuable feedback.

\bibliographystyle{alpha}
\bibliography{/Users/sergio_mayorgatatari/Desktop/mybibfile.bib}

\begin{thebibliography}{CDLL19}

\bibitem[AGS08]{gradientflows}
L.~Ambrosio, N.~Gigli, and G.~Savar\'e.
\newblock {\em Gradient Flows in Metric Spaces and in the Space of Probability
  Measures}.
\newblock Birkh\"auser, 2nd edition, 2008.

\bibitem[Amb16]{ambrosesmallstrong}
D.~M. Ambrose.
\newblock Small strong solutions for time-dependent mean field games with local
  coupling.
\newblock {\em C.~R.~Acad.~Sci.~Paris}, I(354):589--594, 2016.

\bibitem[Amb18]{ambrosestrong}
D.~M. Ambrose.
\newblock Strong solutions for time-dependent mean field games with
  non-separable {H}amiltonians.
\newblock {\em J.~Math.~Pures Appl}, 113:141--154, 2018.

\bibitem[Ave13]{averbukh14}
Y.~Averboukh.
\newblock Minimax approach to first-order mean field games.
\newblock arXiv:1312.6627v2, 2013.

\bibitem[BC17]{bardicirant}
M.~Bardi and M.~Cirant.
\newblock Uniqueness of solutions in mean field games with several populations
  and {N}eumann conditions.
\newblock arXiv:1709.02158v1, 2017.

\bibitem[BCS17]{benamousantambrogio}
J-D. Benamou, G.~Carlier, and F.~Santambrogio.
\newblock Variational mean field games.
\newblock In {\em Modeling and Simulation in Science, Engineering and
  Technology}, volume~1, pages 141--171. Springer Basel, 2017.

\bibitem[Bes16]{bessi}
U.~Bessi.
\newblock Existence of solutions of the master equation in the smooth case.
\newblock {\em SIAM J.~Math.~Anal}, 48(1):204--228, 2016.

\bibitem[BFY13]{bensoussanbook}
A.~Bensoussan, J.~Frehse, and P.~Yam.
\newblock {\em Mean Field Games and Mean Field Type Control Theory}.
\newblock Springer, 2013.

\bibitem[BFY15]{bensoussanthree}
A.~Bensoussan, J.~Frehse, and P.~Yam.
\newblock The master equation in mean field theory.
\newblock {\em J.~Math.~Pures Appl.}, 103(6):1441--1474, 2015.

\bibitem[BFY17]{bensoussanthree2}
A.~Bensoussan, J.~Frehse, and P.~Yam.
\newblock On the interpretation of the master equation.
\newblock {\em Stochastic Processes and their Applications}, 127(7):2093--2137,
  2017.

\bibitem[Bog07]{bogachev}
V.~Bogachev.
\newblock {\em Measure Theory}, volume~2.
\newblock Springer, 2007.

\bibitem[Car12]{cardaliaguetnotes}
P.~Cardaliaguet.
\newblock Notes on mean field games.
\newblock \url{https://www.ceremade.dauphine.fr/~cardaliaguet/MFG20130420.pdf},
  2012.

\bibitem[Car13]{cardaliaguetweakkam}
P.~Cardaliaguet.
\newblock Long time average of first order mean field games and weak {KAM}
  theory.
\newblock {\em P.~Dyn Games Appl}, 3(4):473--488, 2013.

\bibitem[CCD14]{chassagneux}
J-F. Chassagneaux, D.~Crisan, and F.~Delarue.
\newblock A probabilistic approach to classical solutions of the master
  equation for large population equilibria.
\newblock arXiv:1411.3009v2, 2014.

\bibitem[CD14]{carmonadelaruelarge}
R.~Carmona and F.~Delarue.
\newblock The master equation for large population equlibriums.
\newblock In {\em Stochastic Analysis and Applications, Springer Proceedings in
  Mathematics \& Statistics}, volume 100. Springer, Cham, 2014.

\bibitem[CD18]{carmonabible}
R.~Carmona and F.~Delarue.
\newblock {\em Probabilistic Theory of Mean Field Games with Applications},
  volume I, II.
\newblock Springer, 2018.

\bibitem[CDLL19]{fourhorses}
P.~Cardaliaguet, F.~Delarue, J-M. Lasry, and P.~L. Lions.
\newblock {\em The Master Equation and the Convergence Problem in Mean Field
  Games}.
\newblock Princeton University Press, 2019.

\bibitem[CG15]{cardaliaguetgraber}
P.~Cardaliaguet and P.~J. Graber.
\newblock Mean field games systems of first order.
\newblock {\em ESAIM: COCV}, 21(3):690--722, 2015.

\bibitem[CPT15]{cardaliaguetporretta}
P.~Cardaliaguet, A.~Porretta, and D.~Tonon.
\newblock Sobolev regularity for the first order {H}amilton-{J}acobi equation.
\newblock {\em Calc. Var.}, 54(3):3037--3065, 2015.

\bibitem[CT17]{ciranttonon}
M.~Cirant and D.~Tonon.
\newblock Time-dependent focusing mean-field games: the sub-critical case.
\newblock arXiv:1704.04014v1, 2017.

\bibitem[FG95]{degreetheory}
I.~Fonseca and W.Gangbo Gangbo.
\newblock {\em Degree Theory in Analysis and its Applications}.
\newblock Oxford University Press, 1995.

\bibitem[Gan18]{gangboberkeleynotes}
W.~Gangbo.
\newblock On some analytical aspects of mean field games.
\newblock \url{https://math.berkeley.edu/~wgangbo/math_278/MFG-1-18.pdf}, 2018.

\bibitem[GC17]{chowgangbo}
W.~Gangbo and Y.~T. Chow.
\newblock A partial laplacian as an infinitesimal generator on the wasserstein
  space.
\newblock arXiv:1710.10536v1, 2017.

\bibitem[Gho17]{ghoussoub}
N.~Ghoussoub.
\newblock Optimal ballistic transport and {H}opf-{L}ax formulae on
  {W}asserstein space.
\newblock arXiv:1705.05951, 2017.

\bibitem[GJ13]{diogomodelssurvey}
D.~Gomes and Sa\'ude J.
\newblock Mean field games models---a brief survey.
\newblock {\em Dyn Games Appl}, 4(2):110--154, 2013.

\bibitem[GLL10]{gueant}
O.~Gu\'eant, J-M. Lasry, and P.~L. Lions.
\newblock Mean field games and applications.
\newblock In {\em Paris-Princeton Lectures on Mathematical Finance 2010}.
  Springer, Berlin, 2010.

\bibitem[GM18]{grabermeszaros}
P.~J. Graber and A.~M\'esz\'aros.
\newblock Sobolev regularity for first order mean field games.
\newblock {\em Ann.~I.~H.~Poincar\'e}, 14(1):1557--1576, 2018.

\bibitem[GPV16]{diogoregularitybook}
D.~Gomes, E.~Pimentel, and V.~Voskanyan.
\newblock {\em Regularity Theory for Mean-Field Game Systems}.
\newblock Springer, 2016.

\bibitem[GS15]{mfgmain}
W.~Gangbo and A.~\'Swi\k{e}ch.
\newblock Existence of a solution to an equation arising from the theory of
  mean field games.
\newblock {\em Journal of Differential Equations}, 259(11):6573--6643, 2015.

\bibitem[GT14]{weakkam}
W.~Gangbo and A.~Tudorascu.
\newblock Weak {KAM} theory on the {W}asserstein torus with multidimensional
  underlying space.
\newblock {\em Journal of Differential Equations}, 67(3):408--463, 2014.

\bibitem[GT18]{gangtud2017}
W.~Gangbo and A.~Tudorascu.
\newblock On differentiability in the {W}asserstein space and well-posedness
  for {H}amilton-{J}acobi equations.
\newblock {\em J.~Math.~Pures Appl.}, 2018.
\newblock In press.

\bibitem[HCM07]{huangcainesmalhame}
M.~Huang, P.~E. Caines, and R.~P. Malham\'e.
\newblock Large-population cost-coupled {LQG} problems with nonuniform agents:
  individual-mass behavior and decentralized $\epsilon$-{N}ash equilibria.
\newblock {\em IEEE Trans.~Automat.~Control}, 52(9):1560--1571, 2007.

\bibitem[Lac15]{lackerlimit}
D.~Lacker.
\newblock A general characterization of the mean field limit for stochastic
  differential games.
\newblock {\em Probab.~Theory Relat.~Fields}, 165(3-4):581--648, 2015.

\bibitem[Lio]{courscollegefrance}
P.~L. Lions.
\newblock Cours au coll\`ege de france.
\newblock \url{http://www.college-de-france.fr}.

\bibitem[LL07]{lionsjapanese}
J-M. Lasry and P.~L. Lions.
\newblock Mean field games.
\newblock {\em Jpn.~J.~Math.}, 2(1):229--260, 2007.

\bibitem[LS17]{lavenantsantambrogio}
H.~Lavenant and F.~Santambrogio.
\newblock Optimal density evolution with congestion: ${L}^{\infty}$ bounds via
  flow interchange techniques and applications to variational mean field game.
\newblock arXiv:1705.05658v1, 2017.

\bibitem[May]{majorga}
S.~Mayorga.
\newblock On a {C}lassical {S}olution to the {M}aster {E}quation of a {F}irst
  {O}rder {M}ean {F}ield {G}ame.
\newblock Doctoral dissertation, to be available in 2019.

\bibitem[Rud76]{babyrudin}
W.~Rudin.
\newblock {\em Principles of Mathematical Analysis}.
\newblock Mc-Graw Hill, 3rd edition, 1976.

\bibitem[San18]{santambrogioregularity}
F.~Santambrogio.
\newblock Regularity via duality in calculus of variations and degenerate
  elliptic {PDE}s.
\newblock {\em J.~Math.~Anal.~Appl}, 457(2):1649--1674, 2018.

\bibitem[Tra16]{trannonconvex}
H.~Tran.
\newblock A note on nonconvex mean field games.
\newblock arXiv:1612.04725, 2016.

\end{thebibliography}

\end{document}